\documentclass[12pt]{amsart}
\usepackage[margin=1in]{geometry}
\usepackage[utf8]{inputenc}
\usepackage[T2A]{fontenc}
\usepackage[english]{babel}
\usepackage{hyperref}
\usepackage{listings}
\usepackage{amssymb}
\usepackage{amsmath}
\usepackage{tikz-cd}
\usepackage{upgreek}
\usepackage{mathrsfs} 
\usepackage{amsthm}
\usepackage{overpic}
\usepackage{lastpage}
\usepackage{textcomp}
\usepackage{graphicx}
\usepackage{tabulary}
\usepackage{caption}
\usepackage{algorithm} 
\usepackage{faktor} 
\usepackage{enumitem} 
\usepackage{algpseudocode} 

\usepackage{soul}
\usepackage[normalem]{ulem}
\newcommand{\N}{\mathbb{N}}
\newcommand{\Z}{\mathbb{Z}}
\newcommand{\R}{\mathbb{R}}
\renewcommand{\C}{\mathbb{C}}
\renewcommand{\mod}{\mathrm{mod}\,}

\newcommand{\I}{\mathbb{I}}
\newcommand{\D}{\mathbb{D}}

\newcommand{\Sp}{S^2}

\newcommand{\Crit}{C}
\newcommand{\Post}{P}
\newcommand{\Fix}{\mathrm{Fix}}
\newcommand{\Tisch}{\mathrm{Tisch}}
\newcommand{\Charge}{\mathrm{Charge}}
\newcommand{\AdmTrees}{\mathrm{AdmTrees}}
\newcommand{\id}{\mathrm{id}}
\newcommand{\Homeo}{\mathrm{Homeo}}

\newcommand{\PMCG}{\mathrm{PMCG}}

\newcommand{\inter}{\operatorname{int}}

\def\hide #1{}
\long\def\longhide #1 {}

\newcounter{main}
\theoremstyle{plain}
        \newtheorem{theorem}{Theorem}[section]
        \newtheorem*{theorem*}{Theorem}
        \newtheorem*{conj*}{Conjecture}
        \newtheorem{lemma}[theorem]{Lemma}
        \newtheorem{corollary}[theorem]{Corollary}
        \newtheorem{proposition}[theorem]{Proposition}
        \newtheorem{maintheorem}[main]{Main Theorem}      

\theoremstyle{definition}
        \newtheorem{definition}[theorem]{Definition}
        \newtheorem*{definition*}{Definition}

\theoremstyle{remark}
        \newtheorem*{remark}{Remark}
        
        \newtheorem{rem}[theorem]{Remark}
        \newtheorem{example}[theorem]{Example}

        \newtheorem{question}{Question}
        
        \newtheorem*{example*}{Example}
        \newtheorem*{examples*}{Examples}        
        \newtheorem*{claim}{Claim}
        \newtheorem*{claim1}{Claim 1}
        \newtheorem*{claim2}{Claim 2}
        \newtheorem*{claim3}{Claim 3}
        \newtheorem*{claim4}{Claim 4}   
        \newtheorem*{case1}{Case 1}     
        \newtheorem*{case2}{Case 2}

\numberwithin{equation}{section}

\begin{document}
\title[Critically fixed Thurston maps: classification, recognition, and twisting]{Critically fixed Thurston maps: \\ classification, recognition, and twisting} 

\emph{ }

\author{Mikhail Hlushchanka}
\address{Korteweg-de Vries Instituut voor Wiskunde, Universiteit van Amsterdam,  1090 GE \newline Amsterdam, The Netherlands}
\address{Mathematisch Instituut,  Universiteit Utrecht,
 %Postbus 80.010, 
3508 TA Utrecht,  The Netherlands}
\email{mikhail.hlushchanka@gmail.com}

\author{Nikolai Prochorov}

\address{Department of Mathematics, The University of Manchester, Manchester M13 9PL, United Kingdom}
\address {Aix--Marseille Universit\'{e}, CNRS, Institut de Math\'{e}matiques de Marseille, %163 Avenue de Luminy, 
13009 Marseille, France}
\email{nikolai.prochorov@manchester.ac.uk, prochorov41@gmail.com}

\date{\today}


\begin{abstract}
    An orientation-preserving branched covering map $f\colon \Sp \to \Sp$ is called a \emph{critically fixed Thurston map} if $f$ fixes each of its critical points. It was recently shown  that there is an explicit one-to-one correspondence between M\"obius conjugacy classes of critically fixed rational maps and isomorphism classes of planar embedded connected graphs. In this paper, we generalize the result to the whole family of critically fixed Thurston maps. Namely, we show that each critically fixed Thurston map $f$ is obtained by applying the \emph{blow-up operation}, introduced by Kevin Pilgrim and Tan Lei, to a pair $(G,\varphi)$, where $G$ is a planar embedded graph in $\Sp$ without isolated vertices and $\varphi$ is an orientation-preserving homeomorphism of $\Sp$ that fixes each vertex of $G$. This result allows us to provide a classification of combinatorial equivalence classes of critically fixed Thurston maps. 
    We also develop an algorithm that reconstructs (up to isotopy) the pair $(G,\varphi)$ associated with a critically fixed Thurston map $f$. Finally, we solve some special instances of the \emph{Twisting Problem} for the family of critically fixed Thurston maps obtained by blowing up pairs $(G, \id_{\Sp})$. 
\end{abstract}

\maketitle
\setcounter{equation}{0}

\tableofcontents

\section{Introduction}\label{sec: Introduction}

One-dimensional holomorphic dynamics studies the iteration of complex analytic functions on the Riemann sphere $\widehat{\C}$ or on the complex plane $\C$. One of the most influential ideas in the subject, and in the dynamics of complex rational maps in particular, has been to abstract from the rigid underlying complex structure and consider the more general setup of \emph{branched self-coverings} of a topological $2$-sphere $\Sp$ or of a topological plane. Such a branched self-covering is called \emph{postcritically-finite}, or \emph{pcf} for short, if all its critical points are periodic or preperiodic under iteration. Nowadays, orientation-preserving pcf branched covering maps $f\colon\Sp\to\Sp$ of topological degree $\deg(f)\geq 2$ are called \emph{Thurston maps}, in honor of William Thurston, who introduced them in order to better understand the dynamics of pcf rational maps on $\widehat{\C}$. 

At first glance, pcf rational maps may seem like a highly specialized class, as there are only countably many such maps of a given degree up to M\"{o}bius conjugation (except for the well-understood family of \emph{flexible Latt\`{e}s maps}). However, they play a fundamental role in the study of the dynamics of rational maps in general. In particular, the combinatorial structure of the famous \emph{Mandelbrot set} can be described in terms of pcf polynomials \cite{DH_Orsay}. Furthermore, the existence of a strong dynamical similarity between a dense subset of the space of all rational maps (of any fixed degree) and pcf rational maps was conjectured in \cite[Conjecture 1.1]{McM_renorm}; see also the discussion in \cite{McMullen_Aut}.

In this paper, we study the following special subclass of Thurston maps.

\begin{definition}
    A branched covering map $f\colon \Sp\to \Sp$ is called a \emph{critically fixed Thurston map} if $\deg(f) \geq 2$ and each of its critical points is fixed under $f$.
\end{definition}

Critically fixed rational maps, along with their orientation-reversing analogs, have been in the focus of intense research in holomorphic dynamics over the last decade. This interest stems, on the one hand, from their remarkable structural properties, which have enabled progress toward important open problems in the field---such as the Combinatorial Classification Problem \cite{Tischler,Pilgrim_crit_fixed,H_Tischler,LLM_Classification,LMM_Hubbard_trees,ParryPilgrim_Hurwitz}, the Global Curve Attractor Conjecture \cite{H_Tischler,GH_anti_Thurston}, the Crofoot--Sarason Conjecture \cite{Geyer_SharpBounds,LMM_Hubbard_trees}, as well as various questions concerning iterated monodromy groups \cite{H_Thesis}. On the other hand, this research has revealed unexpected connections to other areas, including the theory of Kleinian groups \cite{LLM_Classification, LLM_Deformation,LMMN_welding} and Grothendieck's theory of ``Dessins d'enfants'' \cite{Pakovich}.

The aim of this paper is to develop a comprehensive combinatorial theory for critically fixed Thurston maps and to introduce new perspectives and connections for future exploration. In particular, we
\begin{enumerate}[label=\normalfont(\Alph*)]
\item\label{Obj1} construct canonical combinatorial models for critically fixed Thurston maps;
\item\label{Obj2} provide an algorithm to determine these canonical models;
\item\label{Obj3} analyze the action on these models induced by the post-composition of a given critically fixed rational (or Thurston) map with (special) sphere homeomorphisms;
\item\label{Obj4} propose new questions for further investigation that connect holomorphic dynamics with $3$-manifold theory and topological graph theory.
\end{enumerate}
We begin by providing the context and basic terminology, followed by a brief overview of our main results and their broader connections. Afterwards, we present our findings and methodology in detail.

\subsection{Context}

We denote the set of critical points of a Thurston map $f$ by $\Crit(f)$. The set $\Post(f):=\bigcup_{n=1}^{\infty}f^n(\Crit(f))$ of the forward orbits of the critical points is called the \emph{postcritical set} of~$f$. For critically fixed Thurston maps, we obviously have $\Post(f)=\Crit(f)$. Two Thurston maps are called \emph{combinatorially} (or \emph{Thurston}) \emph{equivalent} (see Definition~\ref{def: Thurston equivalence}) if they are conjugate up to isotopy relative to their postcritical sets; see also Definition~\ref{def: isotopy equivalence} for the \emph{isotopy equivalence}.

One of the key features of Thurston maps is that they often admit a description in purely combinatorial terms. A fundamental question in this context is whether a given Thurston map $f$ can be \emph{realized} by a rational map with the same combinatorics, that is, if $f$ is combinatorially equivalent to a rational map. 
William Thurston answered this question in his celebrated \emph{characterization of rational maps}: 
If a Thurston map $f$ has a \emph{hyperbolic orbifold} (this is always true, except for some well-understood special maps), then $f$ is realized by a rational map $F$ if and only if $f$ has no \emph{Thurston obstruction} \cite{DH_Th_char}. Such an obstruction is given by a finite collection of disjoint Jordan curves in $\Sp\setminus\Post(f)$ with certain invariance properties under the map $f$. Furthermore, the rational map $F$ is unique up to M\"{o}bius conjugation (see
Section \ref{subsec: Thurston's characterization of rational maps} for more discussion). 

Thurston's characterization of rational maps allows one to address the \emph{Combinatorial Classification Problem}, which asks to describe all possible \emph{combinatorial models} that are realized by rational maps within a given family. More precisely, we want to assign some finite combinatorial certificate to each map from the family so that different certificates correspond to different maps and vice versa. This question is well-understood for pcf complex polynomials. Specifically, they can be classified by the so-called \emph{Hubbard trees} \cite{Poirier} or \emph{external angles} \cite{Poirier_Thesis}. Recently, the Combinatorial Classification Problem was also solved for pcf \emph{Newton maps} \cite{Newton, RussellDierk_Class} and for \emph{critically fixed rational maps} \cite{H_Tischler}. In both cases, the classification is given in terms of \emph{planar embedded graphs} on the 2-sphere. Nevertheless, the Combinatorial Classification Problem for the family of \emph{all} pcf rational maps is extremely challenging and remains wide open.

Once we have an answer to the Combinatorial Classification Problem within some family of maps, we can address the \emph{Recognition Problem} of computing the canonical combinatorial model for a given map within the family. For example, this question was recently solved for the class of pcf polynomial maps \cite{LifitingTrees}; see the same paper for a survey of known results.
A related question is the \emph{Conjugacy Problem}, which asks if two given pcf rational maps (or, more generally, two Thurston maps) are combinatorially equivalent. In particular, it is known that the Conjugacy Problem for Thurston maps is decidable \cite{BD_Dec, RSY_Decidability}.

A special instance of the Recognition Problem above is the \emph{Twisting Problem}. Suppose $f \colon \widehat{\C} \to \widehat{\C}$ is a pcf rational map and $\varphi$ is an element of $\Homeo^+(\widehat{\C},\Post(f))$, that is, $\varphi$ is an orientation-preserving homeomorphism of $\widehat{\C}$ that fixes each postcritical point of $f$. The Twisting Problem asks to determine if the \emph{twisted map} $g:=\varphi\circ f$ is realized by a rational map and, if yes, find the canonical combinatorial model for $g$. The Twisting Problem was initially investigated for the case when $f$ is a complex polynomial using algebraic machinery provided by \emph{iterated monodromy groups} \cite{BarNekr_Twist}. An alternative approach in the polynomial case, using \emph{mapping class groups} methods, was recently suggested in \cite{LifitingTrees}. However, in the case of non-polynomial rational maps, the landscape remains largely unexplored. To the best of our knowledge, the only results addressing the Twisting Problem for non-polynomial rational maps---specifically in the low-degree, four-postcritical-point case---are found in \cite{Lodge_Boundary, KelseyLodge}.

\subsection{Overview of results and connections.}
In this paper, we resolve the three problems discussed above---the Combinatorial Classification Problem, the Recognition Problem, and the Twisting Problem---in the setting of critically fixed Thurston maps; see Objectives~\ref{Obj1}, \ref{Obj2}, and \ref{Obj3}.

\subsubsection{} Our first main result, Main Theorem~\ref{thm_intro: classification}, establishes a combinatorial classification of critically fixed Thurston maps. The combinatorial models we use for this classification are given by certain pairs $(G,\varphi)$, where $G\subset \Sp$ is a planar embedded graph and $\varphi\in \Homeo^+(\Sp,V(G))$ is a homeomorphism. The \emph{blow-up operation} \cite{PT} allows one to assign a critically fixed Thurston map $f_{(G,\varphi)}$ to each such model. The main challenge is to clarify the ``injectivity'' of this assignment---which catalog of models yields \emph{all} critically fixed Thurston maps up to combinatorial equivalence (or isotopy). To resolve this issue, we characterize the decompositions of critically fixed Thurston maps with respect to their \emph{canonical Thurston obstructions} \cite{Pilgrim_Canonical}; see Corollary~\ref{cor: can_crit_fix_decomposition}. In fact, we provide a complete description of all possible \emph{completely invariant} multicurves for critically fixed Thurston maps, see Theorem~\ref{thm: crit-fix-decomposition}. To our knowledge, no analogous results exist outside the polynomial setting. Furthermore, these novel structural insights into critically fixed Thurston maps raise a deeper question about their place in the broader landscape of conformal dynamics.

Since Sullivan’s seminal work \cite{Sullivan_QC}, numerous connections between the theory of Kleinian groups and rational dynamics have been uncovered. These analogies between
the two branches of conformal dynamics, now commonly known as \emph{Sullivan’s dictionary}, not only provide a
conceptual framework for understanding these connections but also motivate modern research in both
fields. Recently, Lodge, Luo, and Mukherjee established a strikingly explicit correspondence between dynamics of \emph{Kleinian reflection groups} (groups generated by reflections along the circles of finite circle packings in $\widehat{\C}$) and \emph{critically fixed anti-rational maps} (orientation-reversing analogs of critically fixed rational maps) \cite{LLM_Classification}. Our work, along with its extension to the orientation-reversing setting \cite{GH_anti_Thurston}, holds the potential to formalize the ``decomposition theory'' link between the two sides of Sullivan’s dictionary.

\begin{question}
    Find an explicit correspondence between W.~Thurston's geometric decomposition theory for $3$-manifolds and Pilgrim's decomposition theory for pcf (orientation-reversing) branched self-coverings of $\Sp$.
\end{question}

\subsubsection{} We utilize the developed combinatorial models to tackle the Recognition Problem in the setting of critically fixed Thurston maps. Specifically, we have designed an algorithm that identifies a canonical pair $(G_f, \varphi_f)$ associated with a critically fixed Thurston map $f$. This algorithm is based on the iteration of a \emph{pullback relation} $\xleftarrow{f}$ on (admissible) planar embedded trees induced by the given map $f$, see Section~\ref{sec: Lifting algorithm} for details. This approach is inspired by the \emph{lifting algorithm} for pcf polynomial maps from \cite{LifitingTrees} and the \emph{ivy iteration} for pcf quadratic rational maps from \cite{TimorinTrees}. It is important to note that, in contrast to our algorithm, the lifting algorithm of \cite{LifitingTrees} does not have any complexity estimates, and the ivy iteration of \cite{TimorinTrees} lacks convergence results.

Our second main result, Main Theorem~\ref{thm_intro: contraction}, establishes that the pullback relation $\xleftarrow{f}$ on isotopy classes of (admissible) 
trees rel.\ $\Post(f)=\Crit(f)$ has an explicit \emph{global attractor} $\mathcal{N}_f$, which depends solely on the model graph $G_f$. In particular, the set $\mathcal{N}_f$ has only finitely many elements when $f$ is realized by a rational map. This phenomenon invites a broader question regarding the dynamics of the pullback relation on trees for general pcf rational maps. The following question can be seen as a natural extension of our result and an analog of the \emph{Global Curve Attractor Conjecture} for the pullback relation on Jordan curves; see \cite{Pilgrim_Pullback} and the references therein.

\begin{question}\label{question: global tree attractor conjecture}
    Let $f\colon \widehat{\C} \to \widehat{\C}$ be a pcf rational map with a hyperbolic orbifold, and $\xleftarrow{f}$ be the induced pullback relation on the isotopy classes of admissible planar embedded trees rel.\ $\Post(f)$, see Section~\ref{subsec: The pullback operation}. Does $\xleftarrow{f}$ admit a \emph{finite global attractor}, that is, there is a finite set $\mathcal{N}_f$ such that for every infinite chain
    \[[T_0]\xleftarrow{f}[T_1]\xleftarrow{f}[T_2]\xleftarrow{f}\dots \]
    we have that $[T_n]\in \mathcal{N}_f$ for all sufficiently large $n$?
\end{question}

In a recent preprint, Bartholdi, Dudko, and Pilgrim resolve positively Question~\ref{question: global tree attractor conjecture} (as well as the Global Curve Attractor Conjecture) for rational maps with four postcritical points \cite[Corollary~B]{BartholdiDudkoPilrim_attractor} by establishing a weak form of hyperbolicity for certain naturally associated correspondences between Riemann surfaces. Our Main Theorem~\ref{thm_intro: contraction} provides supporting evidence for such a hyperbolicity statement in the higher-dimensional setting (see \cite[Conjecture~D]{BartholdiDudkoPilrim_attractor}).

We remark that our work reduces the Conjugacy Problem for critically fixed Thurston maps to the following two well-studied computational problems (compare \cite{RSY_Decidability}):
\begin{itemize}
    \item the Isomorphism Problem for planar embedded graphs; 
    \item the Conjugacy Problem for mapping class groups of compact surfaces (more specifically, for the $2$-sphere with finitely many open disks with disjoint closures removed). 
\end{itemize}
Both problems have been extensively studied over the  last decades; see, for instance, \cite{IsomGraphs-Weinberg, IsomGraphNLOGN, IsomGraph-Linear, IsomGrah-Survey, IsomGraphMaps} and \cite{MPG-Anosov, MPG-Masur-Minsky, MPG-Tao, MPG-Bell, MPG-Bell-Webb}. Nowadays, the isomorphism problem for planar embedded connected graphs is known to be solvable in linear time \cite[Theorem~2]{IsomGraphMaps}. Furthermore, Bell--Webb and Margalit--Strenner--Yurtas have announced quadratic time solutions to the Conjugacy Problem in mapping class groups; however, details have not yet appeared.

\subsubsection{} 
Our final main result, Main Theorem~\ref{thm_intro: twisting}, provides an explicit solution to certain cases of the Twisting Problem for critically fixed rational and Thurston maps. Namely, we study compositions of the form $T_\gamma^n\circ f$, where $n \in \Z$, $f$ is a critically fixed Thurston map corresponding to a pair $(G,\id_{\Sp})$, and $T_\gamma\in\Homeo^+(\Sp,\Crit(f))$ is the Dehn twist about an (essential) Jordan curve $\gamma$ intersecting (transversely) each edge of $G$ at most once---we call such curves \emph{simple transversals} with respect to $G$.  We show that a canonical pair $(H,\psi)$ for this twisted map may be explicitly described, and specifically the model graph $H$ is obtained from $G$ by a simple combinatorial operation called the ($-n$)\emph{-rotation about the curve $\gamma$}; see Section~\ref{sec: Twisting problems} for details. This result leads naturally to the following question in topological graph theory.

\begin{question}
 Fix a finite subset $Z \subset \Sp$ with $|Z|\geq 2$, an integer $d \geq 2$, and a vector $\vec{m}:=(m_z)_{z\in Z}$ of positive integers such that $\sum_{z\in Z} m_z = 2d-2$. Let $\mathscr{G}(Z,d,\vec{m})$  denote the set of isotopy classes $[G]$ rel.\ $Z$ of planar embedded graphs $G$ in $\Sp$ with the vertex set $Z$ and vertex degrees $\deg_G(z)=m_z$. Consider a graph $\mathscr{R}(\vec{m})$ whose vertices are the elements of $\mathscr{G}(Z,d,\vec{m})$, and with an edge between $[G]$ and $[G']$ whenever $G'$ can be obtained from $G$ by applying the rotation operation about a simple transversal $\gamma$ with respect to $G$ (or vice versa). Is the graph $\mathscr{R}(\vec{m})$ connected? If so, how big is its diameter?
\end{question}

An affirmative answer would yield a new proof of the fact that the \emph{Hurwitz class} \cite{Koch_Teich} of a critically fixed Thurston map depends only on its branch data; see \cite[Theorem~4]{Pilgrim_crit_fixed}. It will also imply a closely related statement that the \emph{``pure-cycle'' Hurwitz spaces}, parametrizing branched coverings of $\widehat{\C}$ having only one ramified point over each branch point, are irreducible \cite{LiuOsserman}. More broadly, the question above falls into the category of classical transitivity problems in combinatorics that study connectivity of various ``reconfiguration graphs'', such as \emph{realization graphs} of degree sequences \cite{RealizationGraphs}, \emph{recoloring graphs} \cite{RecoloringGraphs1,RecoloringGraphs2}, or \emph{flip graphs} for perfect matchings \cite{FlipGraphs2,FlipGraphs1}.

\subsection{Main results and methodology}\label{subsec: Main results}
Below, we provide a detailed presentation of the main results of this paper.

\subsubsection{Classification and decomposition of critically fixed Thurston maps} 
The Combinatorial Classification Problem for the family of critically fixed rational maps has been studied in several works \cite{Tischler, Pilgrim_crit_fixed, H_Tischler}. The following result provides a complete solution; see \cite[Theorem 2]{H_Tischler}.

\begin{theorem}\label{thm: class_criti_fix_rational}
There is a canonical bijection between the combinatorial equivalence classes (or, equivalently, M\"{o}bius conjugacy classes) of critically fixed rational maps (of degree at least two) and the isomorphism classes of planar embedded connected graphs (with at least one edge).
\end{theorem}

Here and in the following, a planar embedded graph is allowed to have multiple edges but no loops. To associate a critically fixed rational map with a planar embedded connected graph, one uses the so-called ``\emph{blow-up operation}'' introduced by Pilgrim and Tan Lei in \cite{PT}. It is a special surgery on Thurston maps, which we now roughly describe in the context relevant to us (see a detailed discussion in Section \ref{subsec: The blow-up operation}).

Let $G$ be a (finite) planar embedded graph in $\Sp$ with the vertex set $V(G)$ and the edge set $E(G)$, and let $\varphi$ be a homeomorphism in $\Homeo^+(\Sp,V(G))$. 
First, we cut the sphere $\Sp$ open along every edge $e\in E(G)$ and glue in a closed Jordan region $D_e$ in each slit along the boundary. Then we define a branched covering map $\widehat f\colon \widehat\Sp \to \Sp$ on the resulting $2$-sphere $\widehat\Sp$ as follows: $\widehat{f}$ maps the complement of $\bigcup_{e\in E(G)} \inter(D_e)$ onto $\Sp$ in the same way as the homeomorphism $\varphi$, and it maps each $\inter(D_e)$, $e\in E(G)$, to the complement of $\varphi(e)$ in $\Sp$ by a homeomorphism whose extension to $\partial D_e$ matches the map $\varphi|e$. After a natural identification of the domain sphere $\widehat\Sp$ with the target sphere $\Sp$, we get a critically fixed Thurston map $f\colon\Sp\to\Sp$ of degree $d=|E(G)|+1$ and with critical points at the non-isolated vertices of~$G$. We say that the map $f$ is obtained by \emph{blowing up the pair $(G,\varphi)$}.

It is easy to check that a critically fixed Thurston map obtained by blowing up a pair $(G,\id_{\Sp})$ is realized by a rational map whenever $G$ is connected \cite[Theorem 9]{Pilgrim_crit_fixed}. It is shown in \cite[Proposition 7]{H_Tischler} that the converse is also true: up to isotopy, every critically fixed rational map~$F\colon \widehat{\C} \to \widehat{\C}$ (with $\deg(F)\geq 2$) is obtained by blowing up a pair $(G_F,\id_{\widehat{\C}})$ for some planar embedded connected graph $G_F$ with $V(G_F) = \Crit(F)$. The graph $G_F$ is called the \emph{charge graph} of $F$; it is uniquely defined up to isotopy rel.\ $\Crit(F)$. The existence of such a graph is one of the crucial ingredients in the proof of Theorem \ref{thm: class_criti_fix_rational}. 

In this paper, we provide an extension of the classification result above to the class of \emph{all} critically fixed Thurston maps (including the obstructed ones). In particular, we prove that every critically fixed Thurston map $f \colon \Sp \to \Sp$ can be obtained by blowing a pair $(G_f, \varphi_f)$, where $G_f$ is a planar embedded graph in $\Sp$ with the vertex set $V(G_f)=\Crit(f)$ and exactly $\deg(f)-1$ edges, and $\varphi_f$ is a homeomorphism in $\Homeo^+(\Sp,V(G_f))$ satisfying some extra assumptions. Namely, the image $\varphi_f(e)$ of each edge $e\in E(G_f)$ is isotopic to $e$ rel.\ $V(G_f)$. 
Equivalently, this means that, up to isotopy rel.\ $V(G_f)$, the homeomorphism $\varphi_f$ must fix every face of $G_f$; in particular, it is allowed to ``twist'' (only) the multiply connected faces. 

Similar to the rational case, we call the graph $G_f$ the \emph{charge graph} of $f$ and denote it by $\Charge(f)$. We note that the charge graph is \emph{invariant} under $f$. More precisely, $G_f\subset f^{-1}(G_f)$ up to isotopy rel.\ $\Crit(f)$.

We show that the pair $(G_f, \varphi_f)$ as above is uniquely defined (up to isotopy rel.\ $\Crit(f)$) within the isotopy class of $f$; see Proposition \ref{prop: admis_equiv}\ref{item: case isotopy}. In other words, $(G_f, \varphi_f)$ provides a canonical combinatorial model for $f$. This allows us to classify all critically fixed Thurston maps. Before we can provide the precise statement of this classification, we need the following definitions.

\begin{definition}\label{def: admissibility_and_equivalence}
    Let $G$ be a planar embedded graph in $\Sp$ and $\varphi$ be a homeomorphism in $\Homeo^+(\Sp, V(G))$. We say that $(G, \varphi)$ is an \emph{admissible pair} (in $\Sp$) if $G$ does not have isolated vertices and $\varphi(e)$ is isotopic to $e$ rel.\ $V(G)$ for every edge $e\in E(G)$.
    
    Two admissible pairs $(G_1, \varphi_1)$ and $(G_2, \varphi_2)$ in two topological $2$-spheres $\Sp$ and $\widehat{\Sp}$, respectively, are called \emph{equivalent} if there exists an orientation-preserving homeomorphism $\psi \colon \Sp \to \widehat{\Sp}$ such that $\psi(G_1)=G_2$, $\psi(V(G_1))=V(G_2)$, and $\psi \circ \varphi_1 \circ \psi^{-1}$ is isotopic to $\varphi_2$ rel.\ $V(G_2)$. If $\Sp=\widehat{\Sp}$ and $\psi$ is also isotopic to $\id_{\Sp}$ rel.\ $V(G_1)$ we say that the pairs $(G_1,\varphi_1)$ and $(G_2,\varphi_2)$ are \emph{isotopic}.
\end{definition}

The next theorem provides a complete combinatorial classification of critically fixed Thurston maps.

\begin{maintheorem}\label{thm_intro: classification}
    
    There is a canonical bijection between the combinatorial equivalence classes of critically fixed Thurston maps and the equivalence classes of admissible pairs. 
    
    Furthermore, given a finite subset $Z\subset \Sp$ with $|Z|\geq 2$ and an integer $d\geq 2$, there is a canonical bijection between the isotopy classes of critically fixed Thurston maps $f\colon \Sp\to\Sp$ with $\Crit(f)=Z$ and $\deg(f)=d$ and the isotopy classes of admissible pairs $(G,\varphi)$ with $V(G)=Z$ and $|E(G)|=d-1$.
\end{maintheorem}

The proof of Main Theorem \ref{thm_intro: classification} is based on \emph{Pilgrim's decomposition theory} of Thurston maps \cite{Pilgrim_Canonical,Pilgrim_Comb}. In particular, we show that each critically fixed Thurston map $f$ can be canonically decomposed into homeomorphisms and critically fixed Thurston maps that are realized (Corollary \ref{cor: can_crit_fix_decomposition}). The charge graph of $f$ is then defined as the (disjoint) union of the charge graphs of the ``rational pieces'' in this decomposition; see Section \ref{subsec: The charge graph and its applications} for details.

We also point out the following simple criterion for realizability of critically fixed Thurston maps, which also follows from a decomposition result (see Theorem~\ref{thm: crit-fix-decomposition}).  

\begin{proposition}\label{main: obstr-levy-fixed}
    Let $f\colon \Sp \to \Sp$ be a critically fixed Thurston map. Then $f$ is obstructed if and only if $f$ has a Levy fixed curve, that is, there is an essential Jordan curve $\gamma\subset \Sp \setminus \Crit(f)$ and a component $\gamma'$ of $f^{-1}(\gamma)$ such that $\gamma'$ is isotopic to $\gamma$ rel.\ $\Crit(f)$ and $f|\gamma'\colon \gamma'\to \gamma$ is a homeomorphism.
\end{proposition}

Here, a Jordan curve $\gamma\subset \Sp \setminus \Crit(f)$ is \emph{essential} if each connected component of $\Sp\setminus \gamma$ contains at least two critical points of $f$.

\subsubsection{Pullback relation on trees and the Recognition Problem}

Let $f \colon \Sp \to \Sp$ be a critically fixed Thurston map. We say that a planar embedded tree $T$ in $\Sp$ is  \emph{admissible} (for $f$) if $\Crit(f) \subset V(T)$ and every vertex $v\in V(T)\setminus \Crit(f)$ has degree at least $3$ in $T$. The map $f$ induces a natural relation on the set $\AdmTrees(f)$ of all admissible planar embedded trees for $f$. Namely, suppose $T\in \AdmTrees(f)$ is an admissible tree. It is easy to check that the preimage $f^{-1}(T)$ is a planar embedded connected graph with  $V(f^{-1}(T))\supset \Crit(f)$. Hence, we can take a spanning subtree of the critical points in the graph $f^{-1}(T)$. We ``simplify'' the chosen spanning subtree by forgetting all non-critical vertices of degree $2$ (if there any). The resulting tree $T'$ is again an admissible tree for $f$, which we call a \emph{pullback} of the tree $T$ under the map $f$. We denote by $\Pi_f(T)$ the set of all pullbacks of $T$ and use the notation $\xleftarrow{f}$ for the induced \emph{pullback relation} on $\AdmTrees(f)$, i.e., we write $T\xleftarrow{f}T'$ if $T'$ is a pullback of the tree $T$.
It is easy to see that the relation $\xleftarrow{f}$ descends to the quotient of $\AdmTrees(f)$ by isotopies rel.\ $\Crit(f)$. Moreover, the pullback relation may be naturally extended to the case of general Thurston maps by considering planar embedded trees $T$ with $\Post(f)\subset V(T)$ instead of $\Crit(f)\subset V(T)$; see Section~\ref{subsec: The pullback operation}.

We point out that a pullback of $T$ is not uniquely defined, as we may choose different spanning subtrees of $\Crit(f)$ in $f^{-1}(T)$. This contrasts with the lifting operation in the polynomial setting from \cite{LifitingTrees}. Nevertheless, we can still iterate our pullback relation and consider a sequence $\{T_n\}_{n \geq 0}\subset \AdmTrees(f)$ of admissible trees that satisfy $T_0\,\xleftarrow{f}\,T_1\,\xleftarrow{f}\,T_2\,\xleftarrow{f}\cdots$,
that is, $T_{n+1}\in\Pi_f(T_n)$ for all $n\geq0$. We show that this sequence eventually lands in a special set, depending only on the function $f$ and not on the initial tree~$T_0$.

To be more precise, let us consider the following set \[\mathcal{N}_f:=\Bigl\{[T]\colon T\in \AdmTrees(f) \text{ with $T\cap \Charge(f)= \Crit(f)$}\Bigr\},\]
where $[T]$ denotes the equivalence class of a tree $T\in \AdmTrees(f)$ modulo isotopy rel.\ $\Crit(f)$. 
We note that $\mathcal{N}_f$ is finite if and only if $f$ is realized by a rational map. This follows easily from the connectivity of the charge graph in the rational case. Moreover, if $f$ is realized, then every tree $T$ as above 
is invariant under $f$ (up to isotopy rel.\ $\Crit(f)$), since---again up to isotopy rel.\ $\Crit(f)$---the tree $T$ lies in a region of the sphere on which $f$ is isotopic to the identity.

Now the following statement is true.

\begin{maintheorem}\label{thm_intro: contraction}
    Let $f\colon \Sp \to \Sp$ be a critically fixed Thurston map,  and suppose $\{T_n\}_{n \geq 0}\subset \AdmTrees(f)$ is a sequence of admissible trees such that $T_{n + 1} \in \Pi_f(T_n)$ for all $n \geq 0$. Then there exists $N\in\N$, depending only on $f$ and $T_0$, such that $[T_n] \in \mathcal{N}_f$ for every $n \geq N$.
\end{maintheorem}

In other words, up to isotopy, the tree $T_n$ intersects the charge graph of $f$ only in critical points for every sufficiently large $n$. To prove this result, we establish a ``topological contraction'' property for the pullback relation, see Proposition~\ref{prop: Intersections decreasing}.

Main Theorem~\ref{thm_intro: contraction} allows us to develop an algorithm that computes the charge graph for a given critically fixed Thurston map $f$; see Algorithm~\ref{alg: Lifting algorithm}. The key idea here is that once we get a tree $T$ with $[T]\in\mathcal{N}_f$, we can easily reconstruct $\Charge(f)$; see Section~\ref{subsec: Lifting algorithm} for details. We provide an upper bound for the speed of convergence of Algorithm~\ref{alg: Lifting algorithm} in Theorem~\ref{thm: speed of alg}. Knowing $\Charge(f)$, we can decide if $f$ is realized by a rational map and, if not, find its canonical Thurston obstruction (see Theorem \ref{thm: can_obstr}). Moreover, using a simple combinatorial construction, we can retrieve the canonical admissible pair that is associated with $f$, which solves the Recognition Problem; see Section \ref{subsec: recover-homeo}.

\subsubsection{Graph rotations and the Twisting Problem}

The tools we develop may be used to study the Twisting Problem for critically fixed Thurston maps. Let $f\colon \Sp \to \Sp$ be a critically fixed Thurston map and let $\varphi\in \Homeo^+(\Sp, \Crit(f))$ be a homeomorphism. The Twisting Problem asks to determine the combinatorial equivalence class of the twisted map $g := \varphi \circ f$. In our case, the map $g$ is again a critically fixed Thurston map with $\Crit(g) = \Crit(f)$. Thus, we may use Algorithm~\ref{alg: Lifting algorithm} for solving the Twisting Problem; see Example \ref{ex: Twisting by lifting}.  At the same time, in some special cases, we may solve the Twisting Problem explicitly using a simple combinatorial operation applied to the charge graph of the original map $f$.

To be exact, we consider critically fixed Thurston maps $f\colon \Sp \to \Sp$ obtained by blowing up admissible pairs $(G, \id_{\Sp})$. 
Note that up to isotopy this family includes all critically fixed rational maps (i.e., maps for which the graph $G$ is connected). Further, let $\gamma$ be an essential Jordan curve in $\Sp \setminus \Crit(f)$ such that $i_{C(f)}(G,\gamma)=|G\cap \gamma|$ and $|\gamma \cap e| \leq 1$ for each edge $e\in E(G)$. Here, $i_{C(f)}(\cdot, \cdot)$ denotes the (unsigned) \emph{intersection number} rel.\ $\Crit(f)$; see Section~\ref{subsec: Isotopies and intersection numbers} for the precise definition. Finally, we consider the twisted maps of the form $T_\gamma^n \circ f$, where $n\in\mathbb{Z}$, $f$ and $\gamma$ are as described above, and $T_\gamma$ is the \emph{Dehn twist} about the curve $\gamma$. We prove that the combinatorial equivalence class of such a map can be determined by applying a combinatorial operation, which we call the \emph{$(-n)$-rotation about the curve $\gamma$}, to the original graph $G=\Charge(f)$. Roughly speaking, this operation acts as a ``fractional Dehn twist about the curve $\gamma$'' on $G$; see the precise definition in Section~\ref{subsec: Combinatorial operation}.

\begin{maintheorem}\label{thm_intro: twisting}
    Let $f \colon \Sp \to \Sp$ be a critically fixed Thurston map obtained by blowing up an admissible pair $(G, \id_{\Sp})$. Suppose $n\in \Z$ is arbitrary and $\gamma$ is an essential Jordan curve in $\Sp \setminus \Crit(f)$ such that $i_{C(f)}(G,\gamma)=|G\cap \gamma|$ and $|\gamma \cap e| \leq 1$ for each $e\in E(G)$. 
\begin{enumerate}[label=\normalfont(\roman*)]
\item If $|G\cap \gamma|\geq 1$,  then the twisted map $T_\gamma^n \circ f$ is isotopic to a critically fixed Thurston map obtained by blowing up the admissible pair $(H, \id_{\Sp})$, where $H$ is the result of the $(-n)$-rotation about the curve $\gamma$ applied to $G$.
\item If $|G\cap \gamma|= 0$, then the twisted map $T_\gamma^n \circ f$ is isotopic to a critically fixed Thurston map obtained by blowing up the admissible pair $(G, T_\gamma^n )$.
\end{enumerate}
\end{maintheorem}

 This theorem allows to conclude the following statement.

\begin{corollary}\label{cor_intro: periodic} 
  Suppose we are in the setting of Main Theorem \ref{thm_intro: twisting} with $|G\cap \gamma|\geq 1$. Then the sequence of the combinatorial equivalence classes of $\{T_\gamma^n \circ f\}_{n\in\Z}$ is strictly periodic with the period dividing $|G\cap \gamma|$. In other words, if  $n_1 \equiv n_2 \,\,\, \mod\, |G\cap \gamma|$, then the twisted maps $T_\gamma^{n_1} \circ f$ and $T_\gamma^{n_2} \circ f$ are combinatorially equivalent. 
\end{corollary}

\subsection{Organization of the paper}\label{subsec: Organization of the paper}

Our paper is organized as follows. In Section \ref{sec: Background}, we review some general background. In Section \ref{subsec: Notation and basic concepts}, we fix the notation and state some basic definitions. We discuss planar embedded graphs and intersection numbers in Sections \ref{subsec: Planar embedded graphs} and \ref{subsec: Isotopies and intersection numbers}, respectively. In Sections \ref{subsec: Thurston maps} and \ref{subsec: Thurston's characterization of rational maps}, we provide the necessary background on Thurston maps and formulate Thurston's characterization of rational maps. In Section \ref{subsec: Decomposition theory}, we discuss the setup and results from Pilgrim's decomposition theory.  We also formulate some auxiliary results about branched coverings and planar embedded graphs in Section~\ref{subsec: covers and graphs}.

In Section \ref{sec: Critically fixed Thurston maps}, we construct canonical combinatorial models for critically fixed Thurston maps. In Sections \ref{subsec: The blow-up operation} and \ref{subsec: admissble pairs}, we introduce the blow-up operation and discuss its properties. We review the classification of critically fixed rational maps in Section \ref{subsec: Rational case}. 
In Section \ref{subsec: Decomposition}, we study invariant multicurves and decompositions of critically fixed Thurston maps. We discuss the properties of the canonical Thurston obstruction and canonical decomposition for critically fixed Thurston maps in Section \ref{subsec: crit-fix-can-obstr}. In Section \ref{subsec: The charge graph and its applications}, we prove Main Theorem \ref{thm_intro: classification}, i.e., provide a complete classification of critically fixed Thurston maps.

Further, in Section~\ref{sec: Lifting algorithm}, we work on the algorithmic recognition of the combinatorial equivalence classes for critically fixed Thurston maps. In Section \ref{subsec: The pullback operation}, we introduce the pullback relation on (admissible) planar embedded trees. Next, we explore the contraction properties of the pullback relation and prove Main Theorem \ref{thm_intro: contraction} in Section \ref{subsec: Topological contraction of the pullback operation}. In Sections~\ref{subsec: Lifting algorithm} and~\ref{subsec: recover-homeo}, we discuss how to retrieve the canonical combinatorial model for a given critically fixed Thurston map. In particular, we present the  Lifting Algorithm (Algorithm~\ref{alg: Lifting algorithm}) that recovers the charge graph and discuss its complexity.

In Section \ref{sec: Twisting problems}, we study the Twisting Problem for the family of critically fixed Thurston maps. We start  by briefly reviewing the previous work on the problem and looking at a specific example. In Section \ref{subsec: special curves}, we introduce the concept of \emph{simple transversals} with respect to a planar embedded graph. We define the $n$-rotation of a planar embedded graph about its simple transversal in Section \ref{subsec: Combinatorial operation}. Finally, we prove Main Theorem \ref{thm_intro: twisting} and discuss its implications in Section \ref{subsec: Main theorem C}.

\vspace{1cm}

\textbf{Acknowledgments.} This material is based upon work supported by the National Science Foundation under Grant No.\ 1440140 for the first author and Grant No.\ DMS-1928930 for the second author, while the authors were in residence at the Mathematical Sciences Research Institute in Berkeley, California, during the Spring semester of 2022. The authors were also partially supported by the ERC advanced grant ``HOLOGRAM''. We would like to thank Laurent Bartholdi, Mario Bonk, Kostya Drach, Dima Dudko, Lukas Geyer, Russell Lodge, Yusheng Luo, Daniel Meyer, Sabyasachi Mukherjee, Kevin Pilgrim, Palina Salanevich, Dierk Schleicher, Vladlen Timorin, and Rebecca Winarski for valuable comments, remarks, and discussions. We are also grateful to Laurent Bartholdi for his help with the GAP-package~IMG. Finally, we thank the anonymous referees for their thoughtful comments and suggestions on the manuscript.

\section{Background}\label{sec: Background}

\subsection{Notation and basic concepts}\label{subsec: Notation and basic concepts} The sets of positive integers, integers, real numbers, and complex numbers are denoted by $\N$, $\Z$, $\R$, and $\C$, respectively. We use the notation $i$ for the imaginary unit in $\C$, $\I:=[0,1]$ for the closed unit interval on the real line, $\D :=\{z \in \C\colon |z|<1\}$ for the
open unit disk in $\C$, and $\widehat{\C}:=\C\cup\{\infty\}$ for the Riemann sphere. 

The cardinality of a set $X$ is denoted by $|X|$ and the identity map on $X$ by $\id_X$. If $f\colon X\to X$ is a map and $n\in \N$, we denote the $n$-th iterate of $f$ by $f^n$. If $f\colon X\to Y$ is a map and $U\subset X$, then $f|U$ stands for the restriction of $f$ to $U$.

If $X$ is a topological space and $U\subset X$, then $\overline U$ denotes the closure, $\inter(U)$ the
interior, and $\partial U$ the boundary of $U$ in $X$.

Let $S$ be a connected and oriented topological $2$-manifold. We denote its \emph{Euler characteristic} by $\chi(S)$. We use the notation $\Sp$ for an (oriented) \emph{topological $2$-sphere}, that is, a $2$-manifold homeomorphic to the Riemann sphere $\widehat{\C}$. 
Let $V \subset \Sp$ be an open and connected subset of $\Sp$. Then $\chi(V)$ is given by $2 - k_V$, where $k_V$ is the number of complementary components of $V$.

A \emph{Jordan curve} in $\Sp$ is the image of an injective continuous map from $\partial \D$ into the sphere~$\Sp$. A \emph{Jordan arc} $e$ in $\Sp$ is the image $e=\iota(\I)$ of an injective continuous map $\iota\colon \I \to \Sp$. We will use the notation $\partial e := \{\iota(0),\iota(1)\}$ for the \emph{endpoints} and $\inter(e) := e \setminus \partial e$ for the \emph{interior} of the Jordan arc $e$ (these, of course, differ from the boundary and interior of $e$ as a subset of $\Sp$).

A subset $U\subset \Sp$ is called an \emph{open} or \emph{closed Jordan region} (in $\Sp$) if there exists an injective continuous map $\eta\colon \overline{\D} \to \Sp$ such that $U = \eta(\D)$ or $U=\eta(\overline{\D})$, respectively. In both cases, $\partial U = \eta(\partial \D)$ is a Jordan curve in $\Sp$. A \emph{crosscut} in an open or closed Jordan region~$U$ is a Jordan arc $e$ such that $\inter(e)\subset \inter(U)$ and $\partial e \subset \partial U$. 

A \emph{closed annulus} in $\Sp$ is the image $A=\phi( \partial \D \times \I)$ of an injective continuous map $\phi\colon \partial \D \times \I \to \Sp$. A \emph{core curve} of $A$ is a Jordan curve $\gamma \subset \inter(A)$ such that the two boundary curves of $A$ are in distinct components of $\Sp\setminus \gamma$.

Usually, we work with a \emph{finitely marked sphere}, that is, a pair $(\Sp, Z)$, where $Z$ is a finite set of \emph{marked points} in $\Sp$. We say that $e\subset \Sp$ is a Jordan arc in a marked sphere $(\Sp, Z)$ if $e$ is a Jordan arc in $\Sp$ with $\partial e \subset Z$ and $\inter(e) \subset \Sp \setminus Z$. A Jordan curve in $(\Sp,Z)$ is a Jordan curve $\gamma$ in $\Sp$ with $\gamma\subset \Sp\setminus Z$. Such a Jordan curve $\gamma$ is called \emph{essential} if each of the two connected components of $\Sp\setminus \gamma$ contains at least two points from $Z$; otherwise, we say that $\gamma$ is \emph{non-essential}. A non-essential Jordan curve $\gamma$ in $(\Sp,Z)$ is called \emph{null-homotopic} if one of the components of $\Sp\setminus \gamma$ contains no points from $Z$; otherwise, we say that $\gamma$ is \emph{peripheral}.

Let $X$ and $Y$ be topological spaces. A continuous map $H\colon X\times \I \to Y$ is called a \emph{homotopy} from $X$ to $Y$. 
The homotopy $H$ is called an \emph{isotopy} if the \emph{time-$t$ map} $H_t:=H(t,\cdot)\colon X\to Y$ is a homeomorphism for each $t\in \I$. 

Suppose $Z\subset X$. A homotopy $H\colon X\times \I \to Y$ is said to be a homotopy \emph{relative to $Z$} (abbreviated ``$H$ is a homotopy rel.\ $Z$'') if $H_t(p) = H_0(p)$ for all $p\in Z$ and $t\in \I$. Similarly, we define an isotopy rel.\ $Z$.

Two homeomorphisms $h_0, h_1\colon X\to Y$ are called \emph{isotopic} (\emph{rel.\ $Z\subset X$}) if there exists an
isotopy $H\colon X\times \I \to Y$ (rel.\ $Z$) with $H_0=h_0$ and $H_1 = h_1$. 

Given $M, N \subset X$, we say that \emph{$M$ is homotopic to $N$} (\emph{rel.\ $Z\subset X$}) if there exists a homotopy $H\colon X\times \I \to X$  (rel.\ $Z$) with $H_0 = \id_X$ and $H_1(M) = N$. If $H$ is an isotopy rel.~$Z$ we say that \emph{$M$ is isotopic to $N$ rel. $Z$} (or  \emph{$M$ can be isotoped into $N$ rel.\ $Z$}) and denote this by $M\sim N$ rel.\ $Z$.

Let $(\Sp, Z)$ be a finitely marked sphere. We say that two Jordan curves $\gamma_0$ and $\gamma_1$ in $(\Sp, Z)$ are \emph{non-ambient isotopic rel.\ $Z$} if $\gamma_0$ can be continuously deformed through Jordan curves in $(\Sp, Z)$ to $\gamma_1$. More formally, $\gamma_0$ and $\gamma_1$ are non-ambient isotopic rel.\ $Z$ if there exists a continuous map $\gamma\colon \partial \D \times \I \to \Sp$ such that $\gamma(\partial\D\times\{t\})$ is a Jordan curve in  $(\Sp, Z)$ for all $t\in \I$ with  $\gamma(\partial\D\times\{0\})=\gamma_0$ and $\gamma(\partial\D\times\{1\})=\gamma_1$. We define a non-ambient isotopy for Jordan arcs in  $(\Sp, Z)$ in a similar way. Clearly, if two Jordan curves (or arcs) in $(\Sp, Z)$ are isotopic rel.~$Z$, then they are non-ambient isotopic rel.\ $Z$. It is a standard fact that the converse is also true \cite[Theorems A.3 and A.5]{BuserGeometry}; see also \cite[Sections 1.2.5--1.2.7]{FarbMargalit}.

We denote by $\Homeo^+(\Sp, Z)$ the group of all orientation-preserving homeomorphisms of $\Sp$ that fix the set $Z$ element-wise. We will use the notation $\Homeo_0^+(\Sp, Z)$ for the subgroup of $\Homeo^+(\Sp, Z)$ consisting of homeomorphisms isotopic to the identity rel.~$Z$. The \emph{pure mapping class group} of the marked sphere $(\Sp, Z)$ is then defined as the quotient
     $$
        \PMCG(\Sp, Z) := \faktor{\Homeo^+(\Sp, Z)}{ \Homeo_0^+(\Sp, Z)}.
     $$

Let $\gamma$ be a Jordan curve in $(\Sp,Z)$. We will use the notation $T_\gamma$ to denote a \emph{Dehn twist} about $\gamma$ in $(\Sp,Z)$. To define it, first consider the (left) \emph{twist map} $T\colon \partial \D \times \I \to \partial \D \times \I$ given by the formula
$T(e^{2\pi i \theta}, t) = \left(e^{2\pi i(\theta + t)}, t\right)$. We assume that the cylinder $\partial \D \times \I$ is oriented so that its embedding into the complex plane $\C$ via the map $(e^{2\pi i \theta}, t)\mapsto e^{2\pi i \theta}(t+1)$ is orientation-preserving. Now let $A\subset \Sp \setminus Z$ be a closed annulus in $\Sp$ with core curve $\gamma$ and $\phi\colon \partial \D \times \I \to A$ be an orientation-preserving homeomorphism. Then a (left) Dehn twist $T_\gamma$ about the curve $\gamma$ is defined by
$$
    T_\gamma(p)=\begin{cases}
        (\phi \circ T \circ \phi^{-1})(p) & \text{ if } p \in A\\
        p & \text{ if } p \in \Sp \setminus A.
    \end{cases}
$$
By construction, the map $T_\gamma$ is in $\Homeo^+(\Sp, Z)$. It is uniquely defined up to isotopy rel.~$Z$ independently of the choice of $A$ and $\phi$. Furthermore, the isotopy class of $T_\gamma$ does not depend on the choice of the Jordan curve $\gamma$ within its isotopy class rel.\ $Z$; see \cite[Appendix~A2]{HubbardBook1} and \cite[Section~3.1.1]{FarbMargalit}.

\subsection{Planar embedded graphs} \label{subsec: Planar embedded graphs}
We refer the reader to \cite{DiestelGraph} for general background in graph theory. Below we conduct a discussion in the setting of planar embedded graphs, though many of the concepts are also relevant for abstract graphs.

A \emph{planar embedded graph} in a sphere $\Sp$ is a pair ${G=(V, E)}$, where $V$ is a finite set of points in $\Sp$ and $E$ is a finite set of Jordan arcs in $(\Sp, V)$ with pairwise disjoint interiors. The sets $V$ and $E$ are called the \emph{vertex} and \emph{edge sets} of $G$, respectively. Note that our notion of a planar embedded graph allows \emph{multiple edges}, that is, distinct edges  that connect the same pair of vertices. 
However, it does not allow \emph{loops}, that is, edges that connect a vertex to itself.

In the following, suppose $G = (V, E)$ is a planar embedded graph in $\Sp$. The \emph{degree} of a vertex $v$ in $G$, denoted by $\deg_G(v)$, is the number of edges in $G$ incident to $v$. If $\deg_G(v)=0$, we say that the vertex $v$ is \emph{isolated}, and if  $\deg_G(v)=1$, we say that the vertex $v$ is a \emph{leaf}.

The subset $\mathcal{G}:=V \cup \bigcup_{e \in E} e$ of $\Sp$ is called the \emph{realization} of $G$. A \emph{face} of the graph $G$ is a connected component of $\Sp\setminus \mathcal{G}$. Given a planar embedded graph $G$, we denote by $V(G)$, $E(G)$, and $F(G)$ the sets of vertices, edges, and faces of $G$, respectively. 

It will be convenient to conflate a planar embedded graph $G$ with its realization $\mathcal{G}$. In this case, we will specify a finite set $V(\mathcal{G}) \subset \mathcal{G}$ of distinguished points that serve as the vertices of the graph. Then the edge set $E(\mathcal{G})$ consists of the closures of the components of $\mathcal{G}\setminus V(\mathcal{G})$.

A \emph{walk $P$ of length $n$ between vertices $v$ and $v'$} in $G$ is a sequence $(v_0 = v, e_0, v_1, e_1, \dots, \\ v_{n-1}, e_{n - 1}, v_n = v')$, where $e_j$ is an edge in $G$ incident to the vertices $v_j$ and $v_{j + 1}$ for each $j = 0, \dots, n - 1$. If it does not create ambiguity, we may describe the walk $P$ by the sequence $(v_0,v_1,\dots, v_n)$ of its consecutive vertices. 
The walk $P$ is called a \emph{path} if all its edges $e_0,e_1,\dots,e_{n-1}$ are distinct, and it is called a \emph{simple path} if all its vertices $v_0,v_1,\dots,v_n$ are distinct.

A path $(v_0, e_0, v_1, e_1, \dots, v_{n-1}, e_{n - 1}, v_n)$ in $G$ with $v_0=v_n$ and $n\geq 2$ is called a \emph{cycle of length~$n$}.   
Such a cycle is called \emph{simple} if all vertices $v_j, j = 0,\dots, n-1$, are distinct.

The graph $G$ is called \emph{connected} if there is a path in $G$ between every two vertices $v, v' \in V$. In other words, $G$ is connected if its realization is a connected subset of $\Sp$. It follows that the graph $G$ is connected if and only if each face of $G$ is simply connected. We say that the graph $G$ is a \emph{tree} if $G$ is connected and there are no cycles in $G$.

A \emph{subgraph} of $G$ is a planar embedded graph $G' = (V', E')$ such that $V' \subset V$ and $E' \subset E$.  A \emph{connected component of $G$} is a maximal connected subgraph of $G$. The number $k_G$ of connected components of $G$ is given by the \emph{Euler formula}
\[ k_G=|F(G)|-|E(G)|+|V(G)|-1.\]

Let $A$ be a non-empty subset of $V$. A \emph{spanning subtree of $A$ in $G$} is a minimal subtree $T$ of $G$ with $A\subset V(T)$. Such a subtree $T$ exists if and only if all vertices in $A$ belong to the same connected component of $G$. Note that each leaf of $T$ must be in $A$ due to minimality.

Suppose $\Sp$ and $\widehat{\Sp}$ are two topological $2$-spheres. Let $G = (V, E)$ and 
$\widehat{G}=(\widehat{V}, \widehat{E})$ be two planar embedded graphs in $\Sp$ and $\widehat{\Sp}$, respectively. We say that $G$ is \emph{isomorphic} to $\widehat{G}$ if there exists an orientation-preserving homeomorphism $\psi\colon \Sp \to \widehat{\Sp}$ that bijectively maps the vertices and edges of $G$ to the vertices and edges of $\widehat{G}$, that is, $\widehat{V}=\psi(V)$ and $\widehat{E}=\{\psi(e): e\in E\}$.  In this case, we call $\psi$ an \emph{isomorphism} between $G$ and $\widehat{G}$. Clearly, isomorphisms induce an equivalence relation on the set of all planar embedded graphs. An equivalence class of this relation is called an \emph{isomorphism class} of planar embedded graphs.

\subsection{Isotopies and intersection numbers} \label{subsec: Isotopies and intersection numbers} In various constructions of isotopies throughout this paper, we use the following fact without explicit reference (the proof is immediate from \cite[Theorem A.6(ii)]{BuserGeometry}).

\begin{lemma}\label{lem: Buser_isotopy_arcs}
Let $W$ be an open Jordan region in $\Sp$. Suppose $\alpha$ and $\beta$ are Jordan arcs in~$\Sp$ with $\inter(\alpha),\inter(\beta)\subset W$ and $\partial \alpha =\partial \beta$. Then $\alpha$ and $\beta$ are isotopic rel.\ $\partial \alpha \cup (\Sp\setminus W)$.
\end{lemma}

We will frequently consider planar embedded graphs up to isotopy rel.\ finite number of marked points in $\Sp$.

\begin{definition}
    Let $G$ and $G'$ be two planar embedded graphs in $\Sp$ and $Z\subset \Sp$ be a finite set of points. We say that $G$ and $G'$ are \emph{isotopic rel.\ $Z$}, denoted by $G\sim G'$\ rel. $Z$, if there exists an isotopy $H\colon \Sp\times \I \to \Sp$\ rel. $Z$ such that the following conditions are satisfied: 
    \begin{enumerate}[label=(\roman*)]
        \item $H_0=\id_{\Sp}$; \label{item: isotopy_i}
        \item $H_1(V(G)) = V(G')$; \label{item: isotopy_ii}
        \item $H_1(G) = G'$. \label{item: isotopy_iii}
    \end{enumerate}
\end{definition}

Note that \ref{item: isotopy_ii} and \ref{item: isotopy_iii} imply that $H_1$ provides a one-to-one   correspondence between the edges of $G$ and $G'$.

The next statement guarantees that we do not run into topological difficulties while
studying planar embedded graphs (the proof follows from \cite[Lemma A.8]{BuserGeometry}).

\begin{proposition}\label{prop: Graph isotopic to piecewise geodesic}
Let $G$ be a planar embedded graph in $\Sp$ and $Z\subset \Sp$ be a finite set of points. Then there
exists a planar embedded graph $G'$ such that $G\sim G'$ rel.\ $V(G)\cup Z$ and such that each edge of $G'$ is a piecewise geodesic arc in $\Sp$ (with respect to some fixed Riemannian metric on $\Sp$).
\end{proposition}

Using \cite[Theorem A.5]{BuserGeometry} and Lemma \ref{lem: Buser_isotopy_arcs}, we obtain the following criterion for two planar embedded graphs to be isotopic. 

\begin{proposition}\label{prop: Graph isotopic criterion}
    Let $G$ and $G'$ be two planar embedded graphs with a common vertex set~$V$. Then $G$ and $G'$ are isotopic rel.\ $V$ if and only if $|E(G)| = |E(G')|$ and for each edge $e\in E(G)$ there is an edge $e'\in E(G')$ such that $e$ and $e'$ are isotopic rel.\ $V$ and $m_G(e)=m_{G'}(e')$.
\end{proposition}

Here, $m_G(e)$ denotes the \emph{multiplicity} of an edge $e$ in a planar embedded graph $G$, that is, the total number of edges of $G$ that are isotopic to $e$ rel.\ $V(G)$. 

\medskip

In the following, let $(\Sp, Z)$ be a finitely marked sphere. The (unsigned) \emph{intersection number} between two Jordan arcs or curves $\alpha$ and $\beta$ in $(\Sp,Z)$ is defined as
$$
    i_Z(\alpha, \beta) := \inf_{\substack{\text{$\alpha'\sim \alpha$ rel.\ $Z$}, \\ \text{$\beta'\sim \beta$ rel.\ $Z$}}} \left|(\alpha' \cap \beta') \setminus Z\right|,
$$
where the infimum is taken over all Jordan arcs or curves $\alpha'$ and $\beta'$ in $(\Sp,Z)$ that are isotopic to $\alpha$ and $\beta$ rel.\ $Z$, respectively. 
Note that the intersection number is finite, because we can always reduce to the case when $\alpha$ and $\beta$ are piecewise geodesics with respect to some
Riemannian metric on $\Sp$. 
We say that $\alpha$ and $\beta$ are in \emph{minimal position} rel.\ $Z$ if $|(\alpha \cap \beta) \setminus Z| = i_Z(\alpha, \beta)$.

Let $\alpha$ and $\beta$ be two Jordan arcs or curves in $(\Sp, Z)$. We say that subarcs $\alpha' \subset \alpha$ and $\beta' \subset \beta$ form a \emph{bigon} $U$ in $(\Sp, Z)$  
if $\partial\alpha'=\partial\beta'$, $\inter(\alpha')\cap\inter(\beta')=\emptyset$, and $U$ is a connected component of $\Sp \setminus (\alpha' \cup \beta')$ with $U \cap Z = \emptyset$; see the left part of Figure \ref{fig: Removing bigons} for an illustration. It is easy to see that in this situation $\alpha$ and $\beta$ are not in minimal position rel.\ $Z$. Indeed, one of the curves, say $\alpha$, may be isotoped into a new curve $\widetilde{\alpha}$ rel.\ $Z$ with $|(\widetilde{\alpha} \cap \beta) \setminus Z| < |(\alpha \cap \beta) \setminus Z|$; see the right part of Figure \ref{fig: Removing bigons}. 
We call this procedure ``\emph{removing a bigon}'' between $\alpha$ and~$\beta$. In fact, the converse is also true. If two Jordan arcs or curves $\alpha$ and $\beta$ in $(\Sp, Z)$ with transverse intersections are not in minimal position, then there are subarcs $\alpha' \subset \alpha$ and $\beta' \subset \beta$ forming a bigon (see \cite[Proposition~1.7 and Section 1.2.7]{FarbMargalit}).

\begin{figure}[t]
    \centering   
     \begin{overpic}[width=14cm]{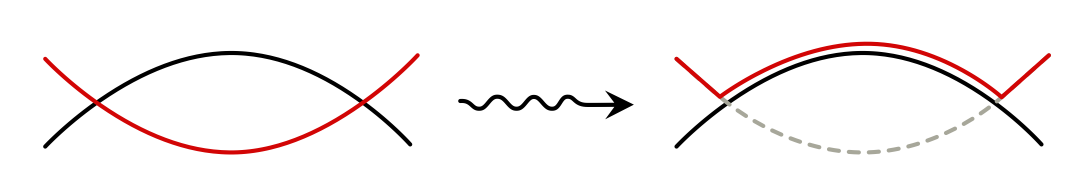}
%        \put(44.5,14){isotopy}
%        \put(42,7){rel.\ $\Crit(f)$}    
        \put(17.5,13){$\beta'\subset \beta$}
        \put(17.5,-0.5){\textcolor{red}{$\alpha'\subset \alpha$}}    
%        \put(1,3){$\beta$}
%        \put(1,10){\textcolor{red}{$\alpha$}}    
        \put(98.5,3){$\beta$}
        \put(98.5,10){\textcolor{red}{$\widetilde{\alpha}$}}           
        \put(20.5,6){$U$}
    \end{overpic}
    
    \caption{Removing a bigon between two Jordan arcs or curves $\alpha$ and $\beta$.}
        
    \label{fig: Removing bigons}
        
\end{figure}
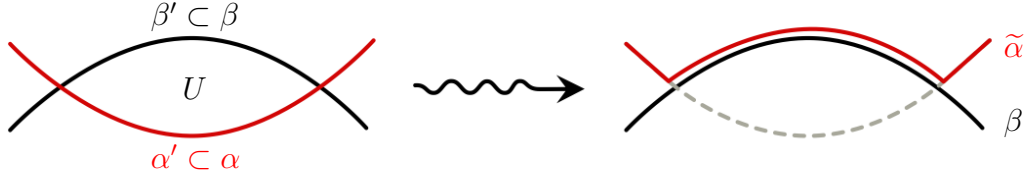

     Let $G$ be a planar embedded graph in $\Sp$ 
     and $\alpha$ be a Jordan arc (or curve) in $(\Sp, Z)$. The \emph{intersection number} $i_Z(G, \alpha)$ between $G$ and $\alpha$ rel.\ $Z$ is defined as
     $$
     i_Z(G, \alpha) := \inf_{\substack{\text{$G'\sim G$\ rel. $Z$},\\\text{$\alpha'\sim \alpha$\ rel. $Z$}}}\,\,{
     |(G'\cap \alpha')\setminus Z|},
     $$
     where the infimum is taken over all planar embedded graphs $G'$ in $\Sp$ and Jordan arcs (curves) $\alpha'$ in $(\Sp, Z)$  that are isotopic to $G$ and $\alpha$ rel.\ $Z$, respectively. We say that $G$ and~$\alpha$ are in \emph{minimal position} rel.\ $Z$ if they satisfy 
     $$i_Z(G, \alpha)= 
    |(G\cap \alpha)\setminus Z|.$$

 The following lemma follows easily from the definitions and Proposition \ref{prop: Graph isotopic to piecewise geodesic} above. 

\begin{lemma}\label{lem: intersection_graph_arc}
    Let $G$ be a planar embedded graph in $\Sp$ and $\alpha$ be a Jordan arc (or curve) in $(\Sp, Z)$, where $Z \subset \Sp$ is a finite set of points. Then the following statements are true:
    \begin{enumerate}[label=\normalfont{(\roman*)}]
        \item The intersection number $i_Z(G, \alpha)$ is finite.
        
        \item\label{item: graph_arc}  There exists a Jordan arc (curve) $\alpha'$ in $(\Sp, Z)$ that is isotopic to $\alpha$ rel.\ $Z$ such that $G$ and $\alpha'$ are in minimal position rel.\ $Z$.
        
        \item There exists a planar embedded graph $G'$ that is isotopic to $G$ rel.\ $Z$ such that $G'$ and $\alpha$ are in minimal position rel.\ $Z$.
        
    \end{enumerate}
\end{lemma}

Lemma \ref{lem: intersection_graph_arc}\ref{item: graph_arc} implies that if $i_Z(G,\alpha)=0$, then we may isotope $\alpha$ so that $G\cap \alpha \subset Z$. We record the following extension of this fact, where we replace a single arc $\alpha$ in $(\Sp, Z)$ by a planar embedded graph $H$ with vertices in $Z$.

\begin{proposition}\label{prop: isotop_to_0_intersections}
Let $(\Sp, Z)$ be a finitely marked sphere and $G$ be a planar embedded graph in $\Sp$. 
Suppose $H$ is a planar embedded graph in $\Sp$ with $V(H) \subset Z$ and such that $i_Z(G, e) = 0$ for each $e \in E(H)$. Then there exists a planar embedded graph $H'$ isotopic to $H$ rel.\ $Z$ such that $H' \cap G \subset Z$.    
\end{proposition}

\begin{proof}
    We only give an outline of the proof, leaving some straightforward details to the reader.
    
    First, it is sufficient to consider the case when $V(H)=Z$. We prove the statement by induction on $|E(H)|$. If $|E(H)|=0$, then there is nothing to prove. Also, if $|E(H)| = 1$, then the statement follows from part \ref{item: graph_arc} of Lemma \ref{lem: intersection_graph_arc}.
    
    Assume the statement is true if $|E(H)| \leq n$, where $n\in \N$. Suppose now that $|E(H)|=n+1$, and consider the graph $H_{\alpha} := (V(H), E(H) \setminus \{\alpha\})$, which is obtained from $H$ by removing some edge $\alpha \in E(H)$. By the induction hypothesis, there is a planar embedded graph $H'_\alpha$ isotopic to $H_\alpha$ rel.\ $Z$ such that 
    $H'_\alpha\cap G \subset Z$.

\begin{claim}
There exists a Jordan arc $\alpha' \sim \alpha$ rel.\ $Z$ such that
$\alpha' \cap (G\cup H_\alpha') \subset Z$. 
\end{claim}

    Let $\mathcal{A}$ be the set of all Jordan arcs $\alpha'$ in $(\Sp,Z)$ that are isotopic to $\alpha$ rel.\ $Z$ and 
    satisfy $\alpha'\cap G \subset Z$. 
    Since $i_Z(G, \alpha)=0$, the set $\mathcal{A}$ is non-empty by Lemma \ref{lem: intersection_graph_arc}\ref{item: graph_arc}.
    Furthermore, $i_Z(e',\alpha')=0$ for every $e'\in E(H'_\alpha)$ and $\alpha'\in \mathcal{A}$. Now consider the following intersection number 
    \begin{equation}\label{eq: zero-intersection}
    N:= \inf_{\alpha'\in \mathcal{A}} {|(H'_\alpha \cap \alpha') \setminus Z|}.
    \end{equation}
    Proposition \ref{prop: Graph isotopic to piecewise geodesic} implies that $N$ is finite and there exists a Jordan arc $\alpha'\in \mathcal{A}$ that realizes the infimum in \eqref{eq: zero-intersection}. 
    We claim that $N=0$. For otherwise, there is an edge $e'\in E(H_\alpha')$ such that $|(e'\cap \alpha') \setminus Z|>0$, which means that $e'$ and $\alpha'$ are not in minimal position rel.\ $Z$. But then some subarcs of $e'$ and $\alpha'$ must form a bigon in $(\Sp, Z)$ or have a non-transverse intersection. 
   We may now remove this bigon or  non-transverse intersection between $e'$ and $\alpha'$ and get a Jordan arc $\widetilde \alpha \in \mathcal{A}$ that satisfies
\[|(H'_\alpha \cap \widetilde{\alpha}) \setminus Z| < |(H'_\alpha \cap \alpha') \setminus Z|.\]
    But this contradicts the choice of $\alpha'$. Thus, $N=0$ and the claim follows.
    
    \medskip
    
    Let $\alpha'$ be a Jordan arc as in the claim. Then $H':=H'_\alpha \cup \alpha'$ is a planar embedded graph with $V(H')=V(H)=Z$ and 
    $H'\cap G\subset Z$.
    By construction, $H'$ and $H$ satisfy the conditions of Proposition~\ref{prop: Graph isotopic criterion}, and thus they are isotopic rel.\ $Z$. This finishes the proof.
\end{proof}

\subsection{Thurston maps}\label{subsec: Thurston maps}
A continuous surjective map $f\colon \Sp \to \Sp$ is called an (orientation-preserving) \emph{branched covering map} if it locally acts as the power map $z \mapsto z^d$ for some $d\in\N$ in orientation-preserving coordinate charts in domain and target. More precisely, for each $p \in \Sp$ we require that there are two open Jordan regions $U$ and $V$ containing $p$ and $f(p)$, respectively, two orientation-preserving homeomorphisms $\varphi\colon \D \to U$ and $\psi\colon\D \to V$, and a number $d \in \N$ such that
\begin{enumerate}[label = (\roman*)]
    \item $\varphi(0) = p$ and $\psi(0)=f(p)$; \label{item:branched-cover-i}
    \item $(\psi^{-1} \circ f \circ \varphi)(z) = z^d$ for all $z \in \D$. \label{item:branched-cover-ii}
\end{enumerate}
The integer $d$ as in \ref{item:branched-cover-ii} is uniquely determined by $f$ and $p$. It is called the \emph{local degree} of the map $f$ at the point $p$ and denoted by $\deg(f,p)$. We also denote the topological degree of $f$ by $\deg(f)$, so that $\sum_{p\in f^{-1}(q)} \deg(f,p) = \deg(f)$ for all $q\in \Sp$.

In the following, let $f\colon \Sp \to \Sp$ be a branched covering map. A point $p\in \Sp$ is called a \emph{critical point} of $f$ if $\deg(f,p)>1$, that is, if $f$ is not locally injective at $p$. We denote the set of all critical points of $f$ by $\Crit(f)$.

Suppose $V\subset \Sp$ is an open and connected set, and  $U$ is a connected component of $f^{-1}(V)$. Then $f(U) = V$ and each point $q \in V$ has the same number $d\in \N$ of preimages under $f|U$ counting multiplicities (given by the local degrees of $f$ at the preimage points). This number $d$ is called the \emph{degree of $f$ on $U$} and denoted by $\deg(f|U)$. If the Euler characteristic $\chi(V )$ is finite, then $\chi(U)$ is also finite and we have the \emph{Riemann-Hurwitz formula}
$$
    \chi(U) + \sum\limits_{c \in U \cap \Crit(f)} (\deg(f, c) - 1) = \deg(f|U) \cdot \chi(V);
$$
see the discussion in \cite[Section 13.2]{THEBook}.

The set
$$
    \Post(f) := \bigcup\limits_{n \in \N} f^{n}(\Crit(f))
$$
is called the \emph{postcritical set} of the branched covering map $f\colon \Sp \to \Sp$. We say that the map $f$ is \emph{postcritically-finite} if $\Post(f)$ is finite.

\begin{definition} 
A \emph{Thurston map} is a postcritically-finite branched covering map $f\colon \Sp\to \Sp$ with $\deg(f)\geq 2$.  
\end{definition}

In other words, a branched covering map $f$ on $\Sp$ is called a Thurston map if it is not a homeomorphism and each critical point of $f$ has a finite orbit under iteration. Natural examples of Thurston maps are provided by \emph{rational Thurston maps}, that is, postcritically-finite rational maps on the Riemann sphere $\widehat{\mathbb{C}}$.

\begin{definition}\label{def: Thurston equivalence} Suppose $\Sp$ and $\widehat{\Sp}$ are two topological $2$-spheres. Two Thurston maps $f\colon \Sp\to \Sp$ and $g\colon \widehat{\Sp} \to \widehat{\Sp}$ 
are called \emph{combinatorially} (or \emph{Thurston}) \emph{equivalent} if there are orientation-preserving homeomorphisms $\psi_0, \psi_1\colon (\Sp, \Post(f)) \to (\widehat{\Sp}, \Post(g))$ that are isotopic rel.\ $\Post(f)$ and satisfy ${\psi_0 \circ f = g \circ \psi_1}$.
\end{definition}

We say that a Thurston map is \emph{realized} (by a rational map) if it is combinatorially equivalent to a rational map. Otherwise, we say that it is \emph{obstructed}.

Thurston maps have the following isotopy lifting property (see, for example, \cite[Proposition 11.3]{THEBook} and the remark after).

\begin{proposition}\label{prop: isotopy-lifting}
Suppose $f\colon \Sp\to \Sp$ and $g\colon \widehat{\Sp} \to \widehat{\Sp}$ are two Thurston maps, and $h_0,\widetilde h_0\colon \Sp \to \widehat{\Sp}$
are homeomorphisms such that $h_0|\Post(f) = \widetilde h_0|\Post(f)$  and $h_0 \circ f=g\circ \widetilde h_0$. Let $H\colon \Sp\times \I \to \widehat{\Sp}$ be an isotopy rel.\ $Q\supset\Post(f)$ with $H_0 = h_0$.

Then the isotopy $H$ uniquely lifts to an isotopy $\widetilde{H}\colon \Sp\times \I \to \widehat{\Sp}$ rel.\ $f^{-1}(Q)\supset \Post(f)$ such that $\widetilde H_0 = \widetilde h_0$ and $g \circ \widetilde H_t = H_t \circ f$ for all $t\in \I$. 

\end{proposition}

We will frequently work with Thurston maps defined by combinatorial constructions, and the following definition appears to be useful.

\begin{definition}\label{def: isotopy equivalence}
     Two Thurston maps $f\colon \Sp \to \Sp$ and $g\colon \Sp \to \Sp$ are called \emph{isotopic} (or \emph{isotopy equivalent}) if $\Post(f) = \Post(g)$ and there exist $\psi_0, \psi_1\in \Homeo_0^+(\Sp, \Post(f))$ such that $\psi_0 \circ f = g \circ \psi_1$.
\end{definition}

Note that Proposition \ref{prop: isotopy-lifting} implies that two Thurston maps $f,g\colon \Sp \to \Sp$ are isotopic if and only if $f=g\circ \psi$ for some $\psi \in \Homeo_0^+(\Sp, \Post(f))$.

The \emph{ramification function} of a Thurston map $f\colon\Sp\to\Sp$ is a function $\nu_f\colon \Sp\to\N\cup \{\infty\}$ defined as follows: $\nu_f(p)$ equals the lowest common multiple of all local degrees $\deg(f^n, q)$, where $q\in f^{-n}(p)$ and $n\in\N$ are arbitrary. It easily follows that $\nu_f(p)\geq 2$ if and only if $p\in \Post(f)$. 

The pair $\mathcal{O}_f:=(\Sp,\nu_f)$ is called the \emph{orbifold} associated with $f$. The \emph{Euler characteristic} of $\mathcal{O}_f$ is given by
$$ \chi(\mathcal{O}_f):= 2- \sum_{p\in \Post(f)} \left(1-\frac{1}{\nu_f(p)}\right).$$
One can check that $\chi(\mathcal{O}_f) \leq 0$ for every Thurston map $f$; see \cite[Proposition 2.12]{THEBook}. We say that $\mathcal{O}_f$ is \emph{hyperbolic} if $\chi(\mathcal{O}_f) < 0$, and \emph{parabolic} if $\chi(\mathcal{O}_f) = 0$. Thurston maps with a parabolic orbifold are rather special and may be completely classified; see, for example, \cite[Chapters~3 and 7]{THEBook}. We note that if $\Post(f) \geq 5$, then $\mathcal{O}_f$ is always hyperbolic. 
Moreover, when $f$ is a critically fixed Thurston map, it has a parabolic orbifold if and only if $|\Crit(f)|=2$ (in which case $f$ is combinatorially equivalent to the power map $z\mapsto z^{\deg(f)}$).

\subsection{Thurston's characterization of rational maps}\label{subsec: Thurston's characterization of rational maps}
In the following, let $f\colon\Sp\to\Sp$ be a Thurston map. A natural question to ask is when $f$ is combinatorially equivalent to a rational map. William Thurston provided a topological criterion that answers this question in his celebrated \emph{characterization of rational maps} \cite[Theorem 1]{DH_Th_char}. To formulate this result we need to introduce several concepts.    

A \emph{multicurve} is a finite collection $\Gamma$ of essential Jordan curves in $(\Sp,\Post(f))$ that are pairwise disjoint and pairwise
non-isotopic rel.\ $\Post(f)$. We say that a multicurve $\Gamma$ is \emph{invariant} if, for every curve $\gamma \in \Gamma$, each essential component of $f^{-1}(\gamma)$ is isotopic rel.\ $\Post(f)$ to a curve in $\Gamma$.

Let $\Gamma=\{\gamma_1,\dots,\gamma_n\}$, $n\in \N$, be an invariant multicurve for $f$. We can now associate an $(n\times n)$-matrix $M(f,\Gamma)=(m_{ij})$ with $\Gamma$ as follows. Fix $i,j\in\{1,\dots,n\}$, and let $\delta_1,\dots, \delta_K$, where $K=K(i,j)\geq 0$, be all the components of $f^{-1}(\gamma_j)$ that are isotopic to $\gamma_i$ rel.\ $\Post(f)$. We denote by $\deg(f|\delta_k)$ the (unsigned) mapping degree of the covering map $f|\delta_k\colon \delta_k \to \gamma_j$. Then the $(i,j)$-entry $m_{ij}$ of the matrix $M(f,\Gamma)$ is given by $$m_{ij}:=\sum_{k=1}^{K(i,j)}\frac{1}{\deg(f|\delta_k)}.$$
If $K(i,j) = 0$, then the sum is empty and $m_{ij}=0$. 

Note that $M(f,\Gamma)$ depends only on the isotopy classes of curves in $\Gamma$ (this easily follows from Proposition \ref{prop: isotopy-lifting}). The Perron-Frobenius theorem implies that the
spectral radius of $M(f,\Gamma)$ is given by the largest non-negative (real) eigenvalue $\lambda(f,\Gamma)$ of this matrix. The invariant multicurve $\Gamma$ is called a \emph{(Thurston) obstruction} for $f$ if $\lambda(f,\Gamma)\geq 1$.

With these definitions, we are finally in a position to state Thurston's theorem; the proof can be found in \cite{DH_Th_char}, see also \cite[Theorem 10.1.14]{HubbardBook2}.

\begin{theorem}\label{thm: Thurston_theorem}
    A Thurston map $f\colon \Sp\to\Sp$ with a hyperbolic orbifold is combinatorially equivalent to a rational map $F\colon\widehat{\C}\to\widehat{\C}$ if and only if $f$ does not have a Thurston obstruction. Moreover, the rational map $F$ is unique up to conjugation by a M\"{o}bius transformation.
\end{theorem}

The easiest examples of obstructions are provided by \emph{Levy fixed curves}.

\begin{definition}
    Let $f\colon\Sp\to\Sp$ be a Thurston map and $\gamma$ be an essential Jordan curve in $(\Sp, \Post(f))$. We call $\gamma$ a \emph{Levy fixed curve} if there is a connected component $\gamma'$ of $f^{-1}(\gamma)$ such that $\gamma$ and $\gamma'$ are isotopic rel.\ $\Post(f)$ and $f|\gamma'\colon \gamma' \to \gamma$ is a homeomorphism.
\end{definition}

Note that if a Thurston map $f$ has a Levy fixed curve $\gamma$, then it must be obstructed (since postcritically-finite rational maps are expanding with respect to the orbifold metric; see \cite[Theorem 19.6]{Milnor_Book}). A priori the multicurve $\{\gamma\}$ does not need to be invariant, but we can always find an invariant multicurve $\Gamma \supset \{\gamma\}$ by taking iterative preimages of $\gamma$; see \cite[Lemma 2.2]{TanLeiMatings} for details. Then this multicurve $\Gamma$ is a Thurston obstruction for $f$.

A Thurston obstruction may contain curves that are ``extraneous'' in some natural sense. The simplest instance of this is when one combines two invariant multicurves $\Gamma_1$ and $\Gamma_2$ having pairwise disjoint and pairwise non-isotopic curves, where $\Gamma_1$ is an obstruction and $\Gamma_2$ is not. Then $\Gamma_1\cup \Gamma_2$ is a Thurston obstruction as well, even though the curves from $\Gamma_2$ are obviously redundant in there. These considerations motivate the following definition.

\begin{definition}
Let $f$ be a Thurston map with an obstruction $\Gamma$. We say that $\Gamma$ is \emph{simple} if there is no permutation of the curves in $\Gamma$ that puts the matrix $M(f,\Gamma)$ in the block form
\[
M(f,\Gamma)=\begin{bmatrix} M_{11} & 0 \\ M_{21} & M_{22} \end{bmatrix},
\]
where the spectral radius of the square matrix $M_{11}$ is less than $1$. 
\end{definition}

One can easily check from the definition that every Thurston obstruction contains a simple one.

\begin{remark}
    We note that Thurston's characterization theorem remains valid in the more general setting of marked Thurston maps. A \emph{marked Thurston map} on $\Sp$ is a pair $(f, Q)$ where $f\colon\Sp\to\Sp$ is a Thurston map and $Q\subset \Sp$ is a finite set of marked points that satisfies $\Post(f)\subset Q$ and $f(Q)\subset Q$. The notions of combinatorial equivalence, isotopy, and Thurston obstructions naturally extend to this setting by considering isotopies rel.\ $Q$ and (multi)curves in $(\Sp,Q)$. Then Theorem \ref{thm: Thurston_theorem} holds in the same form for marked Thurston maps; see, for example, \cite[Theorem 2.1]{Buff_etal}.
\end{remark}

\subsection{Decomposition theory} \label{subsec: Decomposition theory}

We outline a procedure due to Pilgrim that allows one to naturally decompose a Thurston map into ``simpler'' pieces. We refer the reader to \cite{Pilgrim_Comb} for details.

Let $(\Sp,Z)$ be a marked sphere and $\Gamma$ be a finite collection of pairwise disjoint Jordan curves in $(\Sp,Z)$. We denote by $\mathscr{S}_\Gamma$ the set of all components of $\Sp\setminus \bigcup_{\gamma\in\Gamma} \gamma$.  Each such component~$S$ may be viewed as a punctured sphere by collapsing every component of $\partial S$ into a puncture. We call the corresponding closure, denoted by~$\widehat{S}$, a \emph{small sphere} with respect to $\Gamma$. The points in $Z$ and the curves in $\Gamma$ naturally induce a marking on small spheres. 
Namely, the small sphere $\widehat S$ is marked by the set $Q(\widehat S)$ corresponding to the points in $S\cap Z$ and the components of $\partial S$. 
The set of all (marked) small spheres with respect to $\Gamma$ is denoted by~$\widehat{\mathscr{S}}_\Gamma$; see Figure~\ref{fig: decomposition} (bottom) and Figure \ref{fig: small_maps} for an illustration.

We now strengthen the definition of an invariant multicurve for a Thurston map $f$. 
We say that a multicurve $\Gamma$ is \emph{completely invariant} if the following two conditions are satisfied:
\begin{enumerate}[label=(\roman*)]
  \item each essential component of $f^{-1}(\bigcup_{\gamma\in \Gamma} \gamma)$ is isotopic rel.\ $\Post(f)$ to a curve in $\Gamma$;
  
  \item each curve in $\Gamma$ is isotopic rel.\ $\Post(f)$ to a component of $f^{-1}(\bigcup_{\gamma\in \Gamma} \gamma)$.
\end{enumerate}
We note that every simple Thurston obstruction must be completely invariant.

    \begin{figure}[t]
        \centering
        \begin{overpic}[width=14cm]{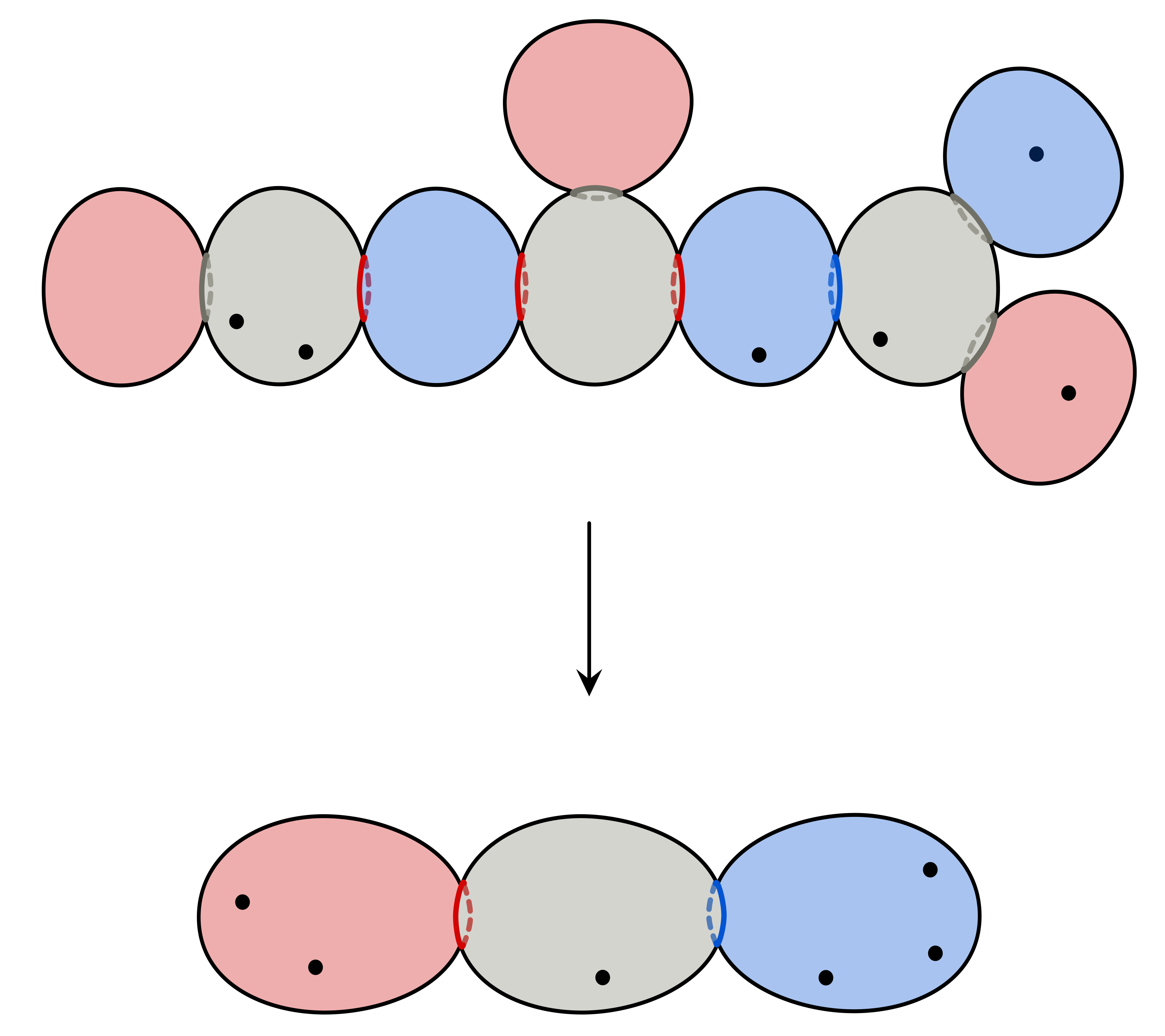}
        \put(52,36){$f$}
        \put(38.5,3.5){$\alpha$}
        \put(60.5,3.5){$\beta$}
        \put(26.5,11.5){$S_1$}   
        \put(49,11.5){$S_2$}   
        \put(70.5,11.5){$S_3$}   
        \put(23,62){$S'_1$}   
        \put(63,62){$S'_2$}   
        \put(77,62){$S'_3$} 
       \put(42.5,48){$(\Sp, \Post(f))$} 
       \put(42.5,21){$(\Sp, \Post(f))$}        
        \end{overpic}
        \caption{
        Decomposing a Thurston map $f\colon \Sp\to \Sp$ along a completely invariant multicurve $\Gamma$. The bottom indicates the multicurve $\Gamma = \{\alpha, \beta\}$ and $\mathscr{S}_\Gamma=\{S_1,S_2,S_3\}$. The top illustrates $f^{-1}(\Gamma)$ and $\mathscr{S}_{f^{-1}(\Gamma)}\supset \{S'_1, S'_2, S'_3\}$. The black dots correspond to the postcritical points of $f$. The map $f$ sends each component in $\mathscr{S}_{f^{-1}(\Gamma)}$ onto the component in $\mathscr{S}_\Gamma$ of the same color. At the top, the red curves are isotopic to $\alpha$, the blue curves are isotopic to $\beta$, and the gray curves are non-essential in $(\Sp, \Post(f))$.}
        \label{fig: decomposition}
    \end{figure}
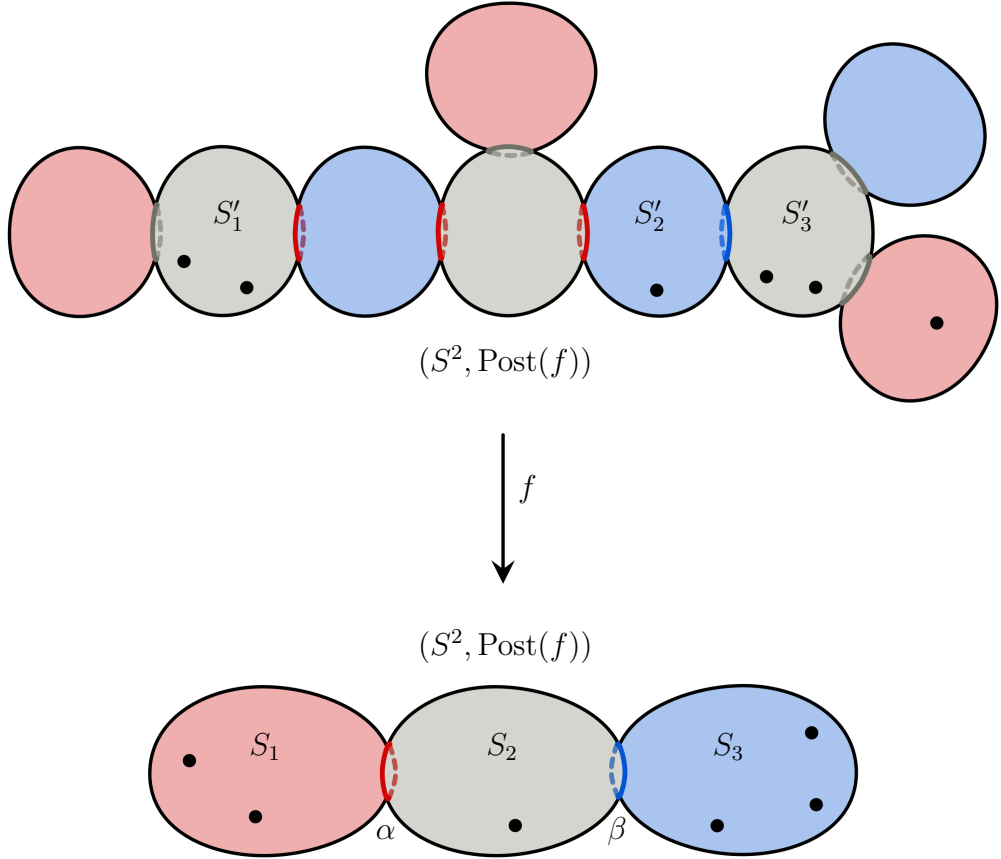

In the following, suppose that $f\colon\Sp\to\Sp$ is a Thurston map and $\Gamma$ is a completely invariant multicurve. For convenience, we denote by $f^{-1}(\Gamma)$ the set of all components of the set $f^{-1}(\bigcup_{\gamma\in \Gamma} \gamma)\subset \Sp$.

Let $\widehat{\mathscr{S}}_\Gamma=\{\widehat{S_j}\}_{j\in J}$ be the set of all small spheres with respect to $\Gamma$. Since $\Gamma$ is completely invariant, we may identify each small sphere $\widehat{S_j}$, $j\in J$, with a unique small sphere $\widehat{S'_j}$ 
 with respect to $f^{-1}(\Gamma)$ as follows. Let $S_j\in \mathscr{S}_\Gamma$ be the component corresponding to the small sphere $\widehat{S_j}$. Then there exists a unique component $S'_j\in \mathscr{S}_{f^{-1}(\Gamma)}$ such that $S'_j\setminus\Post(f)$ is homotopic to $S_j\setminus\Post(f)$ in $\Sp\setminus\Post(f)$; see the top of Figure \ref{fig: decomposition} for an illustration. Furthermore, each component $U$ of the complement $\Sp\setminus \bigcup_{j\in J} S'_j$ is either
\begin{enumerate}[label=(\alph*)]
\item\label{item:disc} a closed Jordan region with $|U \cap \Post(f)| \leq 1$, so that $\partial U$ is a non-essential Jordan curve; 
\item\label{item:annulus} a closed annulus whose boundary components are isotopic rel.\ $\Post(f)$ to a curve $\gamma_U\in \Gamma$;
\item\label{item:curve} or a Jordan curve from $f^{-1}(\Gamma)$ that is isotopic rel.\ $\Post(f)$ to a curve $\gamma_U\in \Gamma$.
\end{enumerate}

We now pick a homotopy that sends each component $S'_j$ onto $S_j$ and collapses each component $U$ of $\Sp\setminus \bigcup_{j\in J} S'_j$ to a point in case \ref{item:disc} or to the curve $\gamma_U\in \Gamma$ in cases \ref{item:annulus} and \ref{item:curve}. More precisely, we choose a homotopy $H\colon \Sp\times \I \to \Sp$ rel.\ $\Post(f)$ with the following properties:
 \begin{enumerate}[label=(\Alph*)]
 \item $H_t:= H(\cdot, t)$ is a homeomorphism for every $t\in[0,1)$;
 \item $H_0=\id_{\Sp}$;
 \item $H_1 (\overline{S_j'}) = \overline{S_j}$ for all $j\in J$;
 \item $H_1|S'_j$ is a homeomorphism of $S_j'$ onto the image $H_1(S'_j)\subset S_j$ for all $j\in J$;
\item Suppose $\gamma'$ is a component of $\partial S_j'$ for some $j\in J$.
\begin{itemize}
    \item If $\gamma'$ is essential, then $H_1$ sends $\gamma'$ homeomorphically onto the component $\gamma$ of $\partial S_j$ that is isotopic to $\gamma'$.
    \item If $\gamma'$ is non-essential, then $H_1$ collapses $\gamma'$ to a single point.
\end{itemize}
 \end{enumerate} 
Then for every $j\in J$ the inverse of $H_1|S_j'\colon S_j'\to S_j$ defines an identification $i^*\colon \big(\widehat{S_j}, Q(\widehat{S_j})\big) \to  \big(\widehat{S'_j},Q(\widehat{S'_j})\big)$ between the small spheres (sending marked points to marked points, but not necessarily bijectively).

By construction, for each $j\in J$ the image $f(S'_j)$ is a component in $\mathscr{S}_\Gamma$, which we denote by $S_{f(j)}$. Then, by filling in the punctures, we get a branched covering map $f_*\colon \big(\widehat{S'_j}, Q(\widehat{S'_j})\big) \to  \big(\widehat{S}_{f(j)}, Q(\widehat{S}_{f(j)})\big)$ between the corresponding small spheres (and respecting the marked points).

    \begin{figure}[t]
        \centering
        \begin{overpic}[width=12cm]{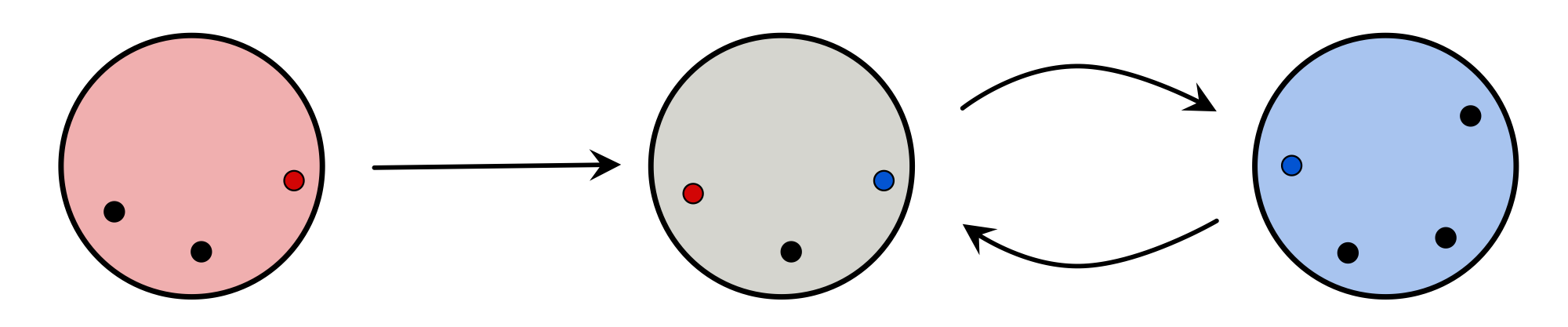}
        \put(10,12){$\widehat{S}_1$}   
        \put(48,12){$\widehat{S}_2$}   
        \put(87,12){$\widehat{S}_3$}   
        \end{overpic}
        \caption{
        The dynamics of $\widehat f \colon \widehat{\mathscr{S}}_\Gamma \to \widehat{\mathscr{S}}_\Gamma$ on small spheres for the example from Figure \ref{fig: decomposition}. The black, red, and blue dots correspond to the postcritical points of $f$, the curve $\alpha\in\Gamma$, and the curve $\beta\in\Gamma$, respectively.
        }
        \label{fig: small_maps}
    \end{figure}

The composition
\[\big(\widehat{S_j}, Q(\widehat{S_j})\big)\xrightarrow{\makebox[1cm]{$i^*$}} \big(\widehat{S'_j}, Q(\widehat{S'_j})\big)\xrightarrow{\makebox[1cm]{$f_*$}} \big(\widehat{S}_{f(j)}, Q(\widehat{S}_{f(j)})\big)\]
defines a branched covering map $\widehat{f}:=f_*\circ i^*\colon \big(\widehat{S_j}, Q(\widehat{S_j})\big) \to  \big(\widehat{S}_{f(j)}, Q(\widehat{S}_{f(j)})\big)$, which we call a \emph{small sphere map}. It is uniquely defined up to isotopy rel.\ $Q(\widehat{S_j})$ for every $j\in J$.

The considerations above imply that $f$ induces a map $$\widehat{f}\colon \bigsqcup_{j\in J} \big(\widehat{S}_j, Q(\widehat{S}_j)\big) \to \bigsqcup_{j\in J} \big(\widehat{S}_j, Q(\widehat{S}_j)\big)$$ on (the disjoint union of) the marked small spheres with respect to $\Gamma$; see Figure \ref{fig: small_maps}. With a slight abuse of notation, we will simply denote this map by $\widehat{f}\colon \widehat{\mathscr{S}}_\Gamma\to\widehat{\mathscr{S}}_\Gamma$. Since $\widehat{\mathscr{S}}_\Gamma$ consists of only finitely many spheres, each small sphere is eventually periodic under $\widehat{f}$. Suppose $\widehat{S}_j\in \widehat{\mathscr{S}}_\Gamma$ is a periodic small sphere. Then the first return map $\widehat{f}^{k(j)}\colon \big(\widehat{S}_j, Q(\widehat{S}_j)\big) \to \big(\widehat{S}_j, Q(\widehat{S}_j)\big)$ is a postcritically-finite branched covering map. Hence this first return map is either a (marked) Thurston map or a homeomorphism.

To summarize the discussion above, a completely invariant multicurve $\Gamma$ allows us to \emph{decompose} the dynamics of a Thurston map $f$ on $\Sp$ into the dynamics of the induced map~$\widehat f$ on the (periodic) small spheres with respect to $\Gamma$.

In \cite{Pilgrim_Canonical}, Pilgrim introduced the notion of a \emph{canonical Thurston obstruction} for a Thurston map $f$.  It is a special multicurve $\Gamma_{\operatorname{Th}}$, defined up to isotopy rel.\ $\Post(f)$, that has the following property: in the case $f$ has a hyperbolic orbifold, the map $f$ is realized by a rational map if and only if $\Gamma_{\operatorname{Th}}$ is empty. If $\Gamma_{\operatorname{Th}}\neq \emptyset$, the multicurve $\Gamma_{\operatorname{Th}}$ is a simple Thurston obstruction and provides the \emph{canonical decomposition} of the given Thurston map~$f$. Selinger gave the following topological characterization of the canonical Thurston obstruction in terms of the pieces of this decomposition; see \cite[Theorem 5.6]{Selinger_Top_Obstr} for a precise statement.

\begin{theorem}\label{thm: can_obstr}%\cite{Pilgrim_Canonical, Selinger_Top_Obstr} 
Let $f\colon\Sp\to\Sp$ be a Thurston map. Then the canonical Thurston obstruction of $f$ is a unique (up to isotopy rel.\ $\Post(f)$) minimal (with respect to inclusion) completely invariant Thurston obstruction $\Gamma$ such that for each periodic small sphere $\widehat S \in \widehat{\mathscr{S}}_\Gamma$ the first return map $\widehat{f}^k\colon \big(\widehat{S}, Q(\widehat S)\big)\to\big(\widehat{S}, Q(\widehat S)\big)$ is either
\begin{enumerate}[label=\normalfont{(\roman*)}]
\item a homeomorphism;
\item a marked Thurston map with a parabolic orbifold and $|\Post(\widehat{f}^k)|=4$ of special type (see \cite[Theorem~5.6]{Selinger_Top_Obstr} for details);
\item or a marked Thurston map that is realized by a marked rational map.
\end{enumerate}
\end{theorem}

\subsection{Branched coverings and graphs}\label{subsec: covers and graphs} Let $f\colon\Sp\to \Sp$ be a branched covering map and $\alpha$ be a Jordan arc in $\Sp$. We say that a Jordan arc $\widetilde{\alpha}\subset \Sp$ is a \emph{lift} of $\alpha$ under $f$ if $f|\widetilde{\alpha}$ is a homeomorphism of $\widetilde{\alpha}$ onto $\alpha$. It easily follows from the existence and uniqueness statements for lifts of paths under covering maps (see, for example, \cite[Lemma A.6]{THEBook}) that if $\alpha$ is a Jordan arc in $\Sp$ with $\inter(\alpha)\subset \Sp \setminus f(\Crit(f))$, $p\in \inter(\alpha)$, and $q\in f^{-1}(p)$, then there exists a unique lift $\widetilde\alpha$ of $\alpha$ under $f$ with $q\in\inter(\widetilde\alpha)$. 

Now suppose that $f\colon \Sp\to\Sp$ is a Thurston map and $G$ is a planar embedded graph in~$\Sp$ with $V(G)\supset \Post(f)$. Then the preimage $f^{-1}(G)$ may be viewed as a planar embedded graph with the vertex set $V(f^{-1}(G)) := f^{-1}(V(G))$ and the edge set $E(f^{-1}(G))$ consisting of all lifts of the edges of $G$ under $f$. The graph $f^{-1}(G)$ is then called the \emph{complete preimage} of $G$ under the map $f$. We note that $$V(f^{-1}(G)) = f^{-1}(V(G)) \supset f^{-1}(\Post(f)) \supset \Post(f) \cup \Crit(f).$$
Furthermore, each face $\widetilde W$ of $f^{-1}(G)$ is a component of $f^{-1}(W)$ for some face $W$ of $G$ and $f|\widetilde W \colon \widetilde W \to W$ is a covering map.

\begin{lemma}\label{lem:preimage-connected}
    Let $f$ be a Thurston map and $G$ be a planar embedded connected graph in $\Sp$ with $V(G) \supset \Post(f)$. Then the complete preimage $f^{-1}(G)$ is a planar embedded connected graph with $\Post(f) \subset V(f^{-1}(G))$.
\end{lemma}

\begin{proof} 
   Indeed, since $G$ is connected and $\Post(f) \subset V(G)$, each face $W$ of $G$ is simply connected and $W \cap \Post(f) = \emptyset$. Thus, by the Riemann-Hurwitz formula, each component of $f^{-1}(W)$ is simply connected as well. Hence, $f^{-1}(G)$ is connected and the statement follows.  
\end{proof}

Finally, we discuss extensions of maps between planar embedded graphs to maps between the underlying spheres. 

Suppose $\Sp$ and $\widehat{\Sp}$ are two topological $2$-spheres. Let $G=(V,E)$ and $\widehat{G}=(\widehat{V}, \widehat{E})$ be two planar embedded graphs in $\Sp$ and $\widehat{\Sp}$, respectively. A continuous map $f\colon G \to \widehat{G}$ is called a \emph{graph map} if forward and inverse images of vertices are vertices (i.e., $f(V)\subset \widehat{V}$ and $f^{-1}(\widehat{V})\subset V$), and $f$ is injective on each edge of $G$. An (orientation-preserving) branched covering map $\overline{f}\colon \Sp \to \widehat{\Sp}$ is called a \emph{regular extension} of a graph map $f\colon G \to \widehat{G}$ if $\overline{f}|G = f$ and $\overline{f}$ is injective on each face of $G$.

A criterion for the existence of regular extensions is provided in \cite[Proposition 6.4]{BFH_Class}. Here, we only record the following uniqueness result from the same paper, which we use in the sequel; see \cite[Corollary 6.3]{BFH_Class}.

\begin{proposition}\label{prop: Graph rigidity}
    Let $G$ and $\widehat{G}$ be two planar embedded connected graphs in $\Sp$ and $\widehat{\Sp}$, respectively. Suppose that $f, g\colon G \to \widehat{G}$ are two graph maps such that $f(v)=g(v)$ and $f(e)=g(e)$ for each $v\in V(G)$ and $e\in E(G)$. If $f$ and $g$ have regular extensions $\overline{f}$ and $\overline{g}$, respectively, then there exists $\psi\in \Homeo^+_0(\Sp, V(G))$ such that $\overline{f}=\overline{g}\circ \psi$.
\end{proposition}

The following corollary follows easily from the proposition above and Proposition \ref{prop: Graph isotopic criterion}.

\begin{corollary}\label{cor: Homeo rigidity}
    Let $G$ be a planar embedded connected graph in $\Sp$. Suppose $\varphi_1,\varphi_2\in \Homeo^+(\Sp)$ satisfy $\varphi_1(e)\sim \varphi_2(e)$ rel.\ $V(G)$ for all $e\in E(G)$. Then $\varphi_1,\varphi_2$ are isotopic rel.~$V(G)$. 
\end{corollary}

\section{Critically fixed Thurston maps}\label{sec: Critically fixed Thurston maps}

The main goal of this section is to provide a classification of \emph{critically fixed Thurston maps}, that is,  Thurston maps that fix (pointwise) each of their critical points. 

First, we define the ``\emph{blow-up operation}'', introduced by Pilgrim and Tan Lei in \cite[Section~2.5]{PT}, which provides a surgery for constructing and modifying Thurston maps and plays a crucial role in our classification result.   
We do not define this operation in complete generality, but in some rather particular setting. Namely, we will be blowing up only pairs 
$(G,\varphi)$, where $G$ is a planar embedded graph in $\Sp$ and $\varphi\in\Homeo^+(\Sp,V(G))$. The result of this operation is a critically fixed Thurston map.  (Recall that $G$ may have multiple edges but no loops and that $\Homeo^+(\Sp,Z)$ denotes the group of all orientation-preserving homeomorphisms of $\Sp$ that fix each point in a finite set $Z\subset \Sp$.)

Conversely, to every critically fixed Thurston map $f\colon \Sp\to\Sp$ we associate a pair $(G_f,\varphi_f)$, where $G_f$ is a planar embedded graph with vertices in $\Crit(f)$ and $\varphi_f$ is a homeomorphism in $\Homeo^+(\Sp, \Crit(f))$, so that $f$ is isotopic to the map obtained by blowing up the pair $(G_f,\varphi_f)$.

\subsection{The blow-up operation}\label{subsec: The blow-up operation}
In the following, let $G$ be a planar embedded graph in~$\Sp$ and $\varphi$ be an element of $\Homeo^+(\Sp,V(G))$.  We will now describe a construction that associates a critically fixed Thurston map $f\colon \Sp \to \Sp$ to every such pair~$(G, \varphi)$ by ``\emph{blowing up}'' each edge of~$G$ (see Figure~\ref{fig: blowup} for an illustration). We will follow the exposition from \cite[Section 4.1]{BHI_Obstructions}, but we will slightly simplify the definition based on the specific features of the considered case.

First, for each edge $e \in E(G)$ we choose an open Jordan region $W_e \subset \Sp$ such that the following conditions hold:
    
    \begin{enumerate}[label=(A\arabic*)]

        \item\label{A1} Every $e\in E(G)$ is a crosscut in $W_e$, that is, $\inter(e) \subset W_e$ and $\partial e \subset \partial W_e$;
        
        \item\label{A2} 
        $\overline{W_e} \cap V(G) = \partial e$ for each $e\in E(G)$;
        
        \item\label{A3} For distinct edges $e_1, e_2 \in E(G)$, we have $\overline{W_{e_1}} \cap \overline{W_{e_2}} = \partial e_1 \cap \partial e_2$. In particular, the open Jordan regions $W_e$ are pairwise disjoint.
        
     \end{enumerate}
 
 \begin{figure}[t]
        \centering
        \begin{overpic}[width=12cm]{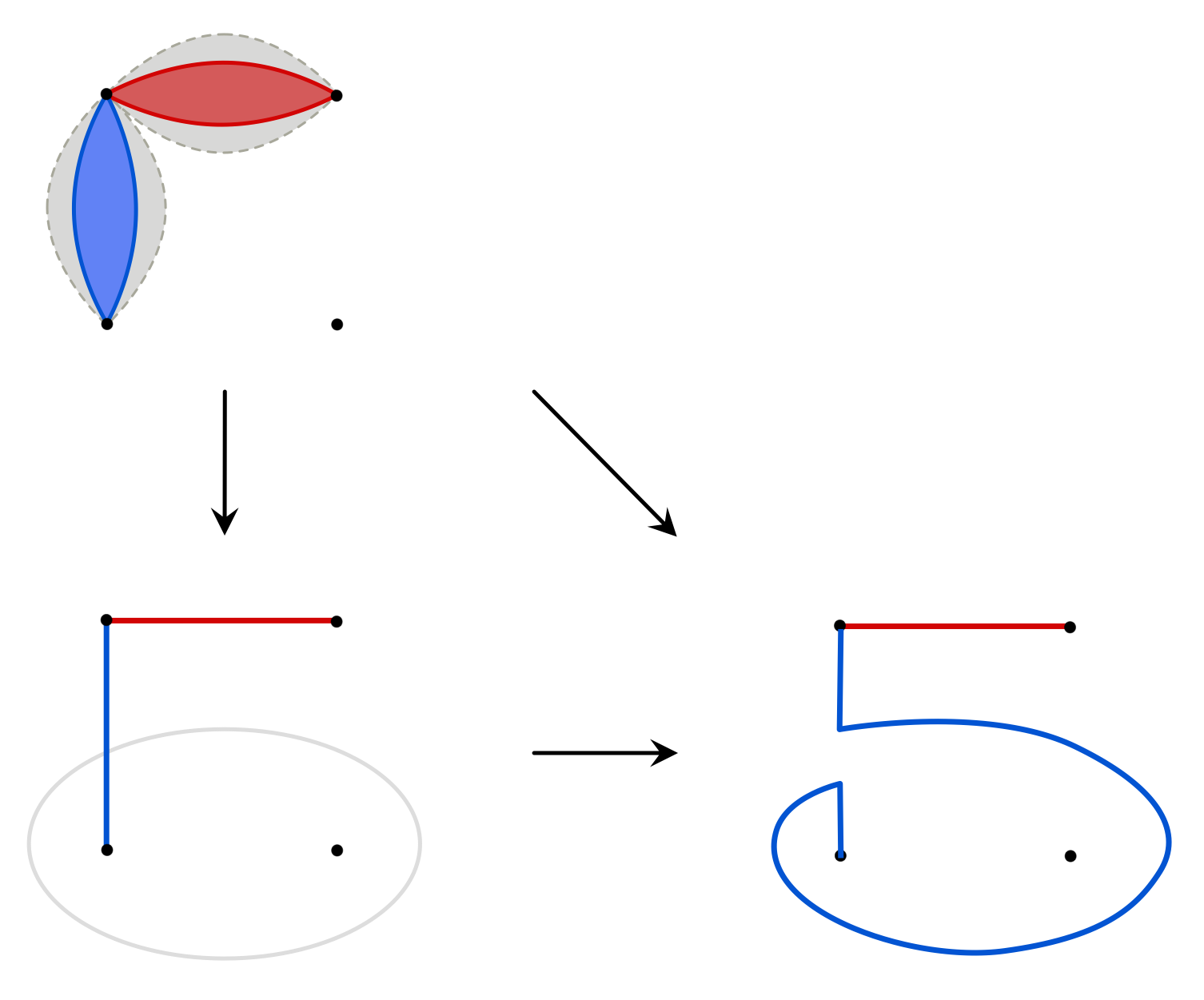}
        \put(48.7,22){$\varphi$}
        \put(51.2,45){$f$}
        \put(15,44){$h_1$}

        \put(29,3){$\gamma$}
        \put(5.5,23){$e_1$}
        \put(17,33){$e_2$}

        \put(16,74){$D_{e_2}$}
        \put(6,64.5){$D_{e_1}$}
        \put(25,69){$W_{e_2}$}
        \put(0,58){$W_{e_1}$}

        \put(87.5,23){$\varphi(e_1)$}
        \put(76,33){$\varphi(e_2)$}
        \end{overpic}
        \caption{
        Illustration of the blow-up operation applied to the pair $(G,\varphi)$, where $G\subset \Sp$ is the planar embedded graph on the bottom left with four vertices (in black) and two edges (in blue and red), and $\varphi=T_\gamma$ is the Dehn twist about the curve $\gamma$ (in gray). The graph on the bottom right depicts the image $\varphi(G)$. The picture on the top left shows the closed Jordan regions $D_{e_1}\supset e_1$ (in blue) and $D_{e_2}\supset e_2$ (in red), as well as the corresponding open Jordan regions $W_{e_1}\supset D_{e_1}$ and $W_{e_2}\supset D_{e_2}$ with dashed boundaries (in gray). The map $f$ denotes a critically fixed Thurston map on $\Sp$ obtained by blowing up the pair $(G,\varphi)$.} 
        \label{fig: blowup}
    \end{figure}    
    
       Next, we choose closed Jordan regions $D_e$, $e\in E(G)$, so that $e$ is a crosscut in $D_e$ and $D_e \setminus \partial e \subset W_e$. 
        The two endpoints of $e$ partition $\partial D_e$ into two Jordan arcs, which we denote by $\partial D_e^+$ and $\partial D_e^-$. One can think of $D_e$ as the resulting region if we cut the sphere $\Sp$ along the edge $e$ and ``open up'' the slit. 
    
    Now we define a map that collapses each $D_e$ back to $e$. More precisely, we choose a homotopy $h\colon \Sp \times \I \to \Sp$ with the following properties:
    \begin{enumerate}[label=(B\arabic*)]
        \item\label{B1} 
        $h_t:= h(\cdot, t)$ is a homeomorphism on $\Sp$ for all $t \in [0, 1)$;
        
        \item\label{B2} $h_0=\id_{\Sp}$;
        
        \item\label{B3} $h_t$ is the identity map on $\Sp \setminus \bigcup_{e \in E(G)} W_e$ for all $t \in \I$;
        
        \item\label{B4} $h_1$ is a homeomorphism of $\Sp \setminus \bigcup_{e \in E(G)} D_e$ onto $\Sp \setminus \bigcup_{e \in E(G)} e$;
        
        \item\label{B5} $h_1$ maps $\partial D_e^+$ and $\partial D_e^-$ homeomorphically onto $e$ for every $e \in E(G)$.
    \end{enumerate}
    It easily follows that $h_1(D_e) = e$ for all $e\in E(G)$. So the homotopy $h$ collapses each closed Jordan region $D_e$ onto $e$ at time $1$, while keeping each point in $\Sp \setminus \bigcup_{e \in E(G)} W_e$ fixed at all times.

    Finally, for every $e \in E(G)$ we choose a continuous map $f_e\colon D_e \to \Sp$ with the following properties:
    
    \begin{enumerate}[label=(C\arabic*)]
        \item\label{C1} $f_e|{\inter(D_e)}\colon \inter(D_e) \to \Sp \setminus \varphi(e)$ is an orientation-preserving homeomorphism;
        
        \item\label{C2} $f_e|\partial{D_e^+} = \varphi \circ h_1|{\partial D_e^+}$ and $f_e|\partial{D_e^-} = \varphi \circ h_1|{\partial D_e^-}$.
    \end{enumerate}
    
    Now we may define a map $f\colon \Sp \to \Sp$ as follows:
    \begin{equation}\label{eq: blow-up_def}
        f(p) = \begin{cases}
        (\varphi \circ h_1)(p) & \text{if } p \in \Sp \setminus \bigcup_{e \in E(G)} D_e\\
        f_e(p) &  \text{if }  p \in D_e.
        \end{cases}
    \end{equation}

    \begin{definition}\label{def: Blow up}
         We say that the map $f\colon \Sp \to \Sp$ as described above is obtained by \emph{blowing up} the pair $(G, \varphi)$. The procedure of constructing this map is called the \emph{blow-up operation}.
    \end{definition}

    One can check that the map $f\colon \Sp\to \Sp$ we just constructed is in fact a critically fixed Thurston map with the properties summarized in the following proposition. (For the proof, see \cite[Lemma 4.3]{BHI_Obstructions}.)

    \begin{proposition}\label{prop: Blow-up properties}
        Suppose a map $f\colon \Sp\to\Sp$ is obtained by blowing up a pair $(G, \varphi)$, where $G$ is a planar embedded graph in $\Sp$ and $\varphi\in\Homeo^+(\Sp,V(G))$. Then the following
statements are true:
        
        \begin{enumerate}[label=\normalfont{(\roman*)}]
            \item \label{item: degree-of-blow-up} $f$ is a Thurston map with $\deg(f) = |E(G)| + 1$;
            
            \item $\Crit(f) = \{v \in V(G): \deg_G(v) > 0\}$;

            \item $\deg(f, v) = \deg_G(v) + 1$ for every $v \in V(G)$;
            
            \item each $v \in V(G)$ is fixed under $f$.
        \end{enumerate}
        In particular, $f$ is a critically fixed Thurston map.
    \end{proposition}

\begin{example}\label{ex: Topological square map}

 \begin{figure}[t]
        \centering
        \begin{overpic}[width=12cm]{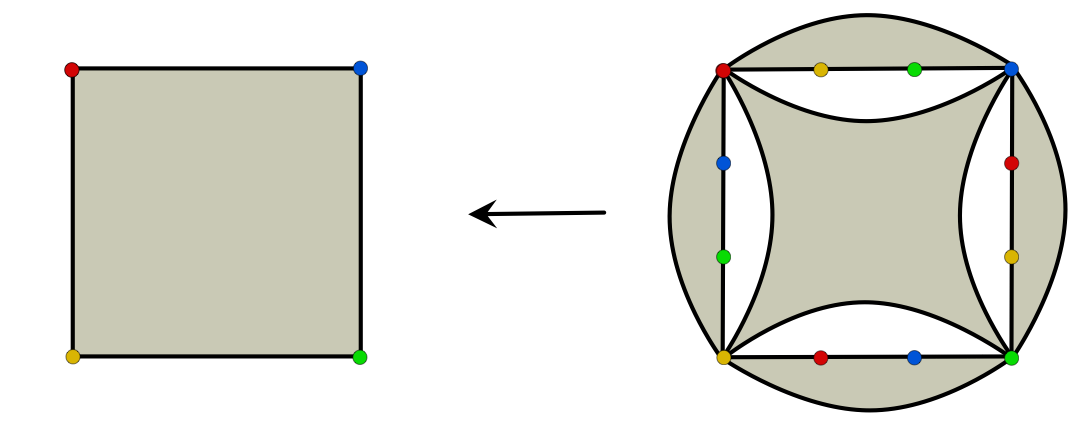}
        \put(48,22){$f_\square$}
        \end{overpic}
        \caption{
        A critically fixed Thurston map $f_\square$ obtained by blowing up the pair $(G_\square, \id_{\Sp})$, where $G_\square$ is the planar embedded graph on the left.
        }
        \label{fig: square_map}
    \end{figure}

    Consider the ``square graph'' $G_\square$ shown on the left in Figure \ref{fig: square_map}. 
    Here and in the following, all the graphs in figures are assumed to be embedded in an underlying $2$-sphere. In particular, the graph $G_\square$ has two simply connected faces, which we denote by $U_g$ and $U_w$. The face $U_g$ corresponds to the interior of the square, which is colored gray, and the face~$U_w$ corresponds to the exterior of the square, which is colored white. Figure \ref{fig: square_map} illustrates a critically fixed Thurston map $f_\square$ obtained by blowing up the pair $(G_{\square}, \id_{\Sp})$.  
        Namely, the closure of each gray region $U$ on the right is mapped by $f_\square$ homeomorphically onto the closure of the gray face $U_g$ on the left so that the marked vertices on $\partial U$ are sent to the vertices of the same color on $\partial U_g$. An analogous statement is true for each white region on the right.  
  Note that $\deg(f_{\square}) = 5$, and each vertex of $G_\square$ is a fixed critical point of $f_\square$ with the local degree $3$. We will use the map $f_\square$ as a prototypical example in our paper. 
\end{example}

    It follows from \cite[Proposition~2]{PT} (and the remark after it) that a critically fixed Thurston map $f$ obtained by \emph{blowing up} the pair $(G, \varphi)$ is uniquely defined up to isotopy (rel.\ $V(G)$) independently of the choices in the construction above. Moreover, up to isotopy $f$ depends only on the isotopy classes of $G$ and $\varphi$.

    \begin{proposition}\label{prop: Isotopy and blow-up}
   For $j=1,2$, suppose $G_j$ is a planar embedded graph in $\Sp$ and $\varphi_j$ is a homeomorphism in $\Homeo^+(\Sp, V(G_j))$. Let $f_1$ and $f_2$ be critically fixed Thurston maps obtained by blowing up the pairs $(G_1, \varphi_1)$ and $(G_2, \varphi_2)$, respectively. If the pairs $(G_1,\varphi_1)$ and $(G_2,\varphi_2)$ are isotopic, then the marked Thurston maps $(f_1, V(G_1))$ and $(f_2, V(G_2))$ are isotopic as well.
    \end{proposition}

 Here and in below, we say that the pairs $(G_1,\varphi_1)$ and $(G_2,\varphi_2)$ are \emph{isotopic}, if $G_1$ and $\varphi_1$ are isotopic to $G_2$ and $\varphi_2$ rel.\ $V(G_1)$, respectively. The proposition above easily implies the following fact.
 
 \begin{proposition}\label{prop: from_id_to_any_homeo}
    Suppose $f\colon \Sp\to\Sp$ and $g\colon \Sp\to \Sp$ are critically fixed Thurston maps obtained by blowing up pairs $(G, \varphi)$ and $(G, \id_{\Sp})$, respectively, where $G$ is a planar embedded graph in $\Sp$ and $\varphi\in\Homeo^+(\Sp,V(G))$. Then the marked Thurston maps $(f, V(G))$ and $(\varphi \circ g, V(G))$ are isotopic.
 \end{proposition}

Since we frequently work with Thurston maps defined up to isotopy (or modulo combinatorial equivalence), the next statement is useful for us.

\begin{proposition}\label{prop: comb_equiv_blowups}
Let $f\colon \Sp \to \Sp$ be a critically fixed Thurston map obtained by blowing up a pair $(G,\varphi)$, where $G$ is a planar embedded graph in $\Sp$ and $\varphi\in\Homeo^+(\Sp,V(G))$. 
Suppose that a marked Thurston map $(g, Q)$ on a topological $2$-sphere $\widehat{\Sp}$ is combinatorially equivalent to $(f, V(G))$.  Then $g$ is obtained by blowing a pair $(\widehat G,\widehat \varphi)$, where $\widehat G$ is a planar embedded graph in $\widehat \Sp$ and $\widehat \varphi\in\Homeo^+(\widehat \Sp,V(\widehat G))$.
\end{proposition}

\begin{proof} Suppose $f\colon \Sp \to \Sp$ is a critically fixed Thurston map obtained by blowing up a pair $(G,\varphi)$ as in the statement, that is, we fix a choice of $W_e$, $D_e$, $f_e$, and $h$ as in the construction above.  Moreover, we assume that $\psi_0, \psi_1\colon (\Sp, V(G)) \to 
(\widehat{\Sp}, Q)$ are orientation-preserving homeomorphisms that are isotopic rel.\ $V(G)$ and satisfy $\psi_0 \circ f = 
g\circ\psi_1$.

\begin{center}    
         \begin{tikzcd}[row sep=tiny]
         \Sp\setminus\bigcup_{e\in E(G)} D_e  \arrow[dd, "f"'] \arrow[dr, "h_1"'] \arrow[rrr, "\psi_1"] &&& \widehat\Sp\setminus\bigcup_{\widehat{e}\in E(\widehat{G})} \widehat{D}_{\widehat{e}} \arrow[dd, "%\widehat{f} 
         g"] \arrow[dl, "\widehat h_1" ] \\
       & \Sp\setminus G  \arrow[dl, "\varphi" ] \arrow[r, "\psi_1"]&\widehat\Sp \setminus \widehat{G} \arrow[dr, "\widehat \varphi"']  \\
        \Sp\setminus \varphi(G) \arrow[rrr, "\psi_0"'] &&& \widehat\Sp \setminus \widehat{\varphi}(\widehat{G})
    \end{tikzcd}
\end{center}    

Let $\widehat G$ be the planar embedded graph $\psi_1(G)$ in $\widehat{\Sp}$ with the vertex set $V(\widehat{G})=\psi_1(V(G))=
Q$. Then the edges of $\widehat{G}$ are given by the images $\widehat{e}:=\psi_1(e)$, $e\in E(G)$. Furthermore, let $\widehat \varphi:= \psi_0\circ\varphi\circ\psi_1^{-1}$. Then $\widehat \varphi\in \Homeo^+(\widehat{\Sp}, V(\widehat{G}))$. 
We claim that the map 
$g$ is obtained by blowing up the pair $(\widehat{G},\widehat{\varphi})$ with the following choices:
\begin{enumerate}[label=\normalfont{(\roman*)}]
    \item for each edge $\widehat{e}=\psi_1(e)$, $e\in E(G)$, we set
\[   \,\, 
\widehat{W}_{\widehat{e}}:=\psi_1(W_e), \,\, \widehat{D}_{\widehat{e}}:=\psi_1(D_e), \, \text{ and } \, 
g_{\widehat{e}}:=\psi_0\circ f_e\circ\psi_1^{-1};\]
    \item for each $t\in \I$ we set $\widehat{h}_t:=\psi_1\circ h_t\circ  \psi_1^{-1}$. 
\end{enumerate}
Indeed, this can be easily verified from the identity $
g_{\widehat{e}}=
g|\widehat{D}_{\widehat{e}}$   and the commutative diagram above, and we leave the straightforward details to the reader.
\end{proof}

\begin{rem}\label{rem: comb_equiv_and_adm_pairs} 
Suppose that we are in the setting of Proposition~\ref{prop: comb_equiv_blowups}, and assume that $\psi_0 \circ f = g \circ \psi_1$, where $\psi_0, \psi_1 \colon (S^2, V(G)) \to (\widehat{S^2}, Q)$ are orientation-preserving homeomorphisms that are isotopic rel.\ $V(G)$. Then the proof of Proposition~\ref{prop: comb_equiv_blowups} shows that $g$ is obtained by blowing up the pair $(\widehat{G}, \widehat{\varphi})$ with  $\widehat{G} := \psi_1(G)$ and $\widehat{\varphi} := \psi_0 \circ \varphi \circ \psi_1^{-1}$. The following facts are direct consequences:
\begin{enumerate}[label=\normalfont{(\roman*)}]
\item\label{item: comb_equiv_and_adm_pairs_0} If $\varphi\in \Homeo^+_0(\Sp, V(G))$, then $\widehat \varphi\in \Homeo^+_0(\widehat{\Sp}, V(\widehat{G}))$.
\item\label{item: comb_equiv_and_adm_pairs_1} If the pair $(G,\varphi)$ is admissible in $\Sp$, then the pair $(\widehat{G},\widehat{\varphi})$ is admissible in $\widehat{\Sp}$. Furthermore, the pairs $(G,\varphi)$ and $(\widehat{G},\widehat{\varphi})$ are equivalent (see Definition~\ref{def: admissibility_and_equivalence}).
\item\label{item: comb_equiv_and_adm_pairs_2} 
If $\Sp=\widehat{\Sp}$, $V(G)=Q$, and $\psi_1,\psi_2$ are isotopic to $\id_{\Sp}$ rel.\ $V(G)$, so that the marked Thurston maps $(f, V(G))$ and $(g, Q)$ are isotopic, then the pairs $(G,\varphi)$ and $(\widehat{G}, \widehat{\varphi})$ are isotopic as well.
\end{enumerate}    
\end{rem}

\subsection{Admissible pairs}\label{subsec: admissble pairs}
Let $G$ be a planar embedded graph in~$\Sp$ and $\varphi\in \Homeo^+(\Sp,V(G))$ be a homeomorphism. Recall from the introduction that the pair $(G,
\varphi)$ is called \emph{admissible} (in $\Sp$) if $G$ has no isolated vertices and $\varphi(e)$ is isotopic to~$e$ rel.\ $V(G)$ for each $e\in E(G)$. In this subsection, we establish some basic properties of critically fixed Thurston maps obtained by blowing up admissible pairs.

\subsubsection{A blow-up criterion}
\label{subsubsec: A blow-up criterion}
Our first goal is to provide a criterion for checking if a given critically fixed Thurston map arises (up to isotopy) by blowing up some admissible pair $(G, \varphi)$ for a given planar embedded graph $G$. First, we summarize the mapping properties of such maps.

\begin{proposition}\label{prop: blow-up-triples-properties}
    Let $f\colon \Sp \to \Sp$ be a critically fixed Thurston map obtained by blowing up an admissible pair $(G,\varphi)$ in $\Sp$. Suppose that $K$ is a planar embedded graph in $\Sp$ that is isotopic to $G$ rel.\ $V(G)$. Then for each $\alpha\in E(K)$ there is a triple $(\alpha^+, \alpha^-, U_\alpha)$ satisfying the following conditions:
    \begin{enumerate}[label=\normalfont{(D\arabic*)}]
        \item\label{item: blow-up-tripple-i} $\alpha^+$ and $\alpha^-$ are distinct lifts of $\alpha$ under $f$ that are isotopic to $\alpha$ rel.\ $V(K)$;
        
        \item\label{item: blow-up-tripple-ii} $U_\alpha$ is a connected component of $\Sp \setminus (\alpha^+ \cup \alpha^-)$ with $U_\alpha \cap V(K) = \emptyset$;
        
        \item\label{item: blow-up-tripple-iii} $\overline{U_{\alpha_1}} \cap \overline{U_{\alpha_2}} = \alpha_1\cap\alpha_2$ for distinct $\alpha_1, \alpha_2 \in E(K)$.   
    \end{enumerate}    

    Furthermore, let us consider the planar embedded graph $K^{\pm}:=\bigcup_{\alpha \in E(K)}(\alpha^+ \cup \alpha^-)$ with the vertex set $V(K^{\pm})=V(K)=V(G)$. Then $f$ sends each face $\widetilde{W}$ of $K^{\pm}$ homeomorphically onto its image. More precisely, the following statements are true:
    \begin{enumerate}[label=\normalfont{(E\arabic*)}]
        \item\label{item: blow-up-mapping-i} If $\widetilde{W}=U_\alpha$ for some $\alpha\in E(K)$, then $f$ sends $\widetilde{W}$ homeomorphically onto $\Sp \setminus \alpha$.
        \item\label{item: blow-up-mapping-ii} If $\widetilde{W}\neq U_\alpha$ for every $\alpha\in E(K)$, then $f(\widetilde{W})$ is a face of $K$ with $\partial f(\widetilde{W}) = f(\partial \widetilde{W})$ and  $f|\widetilde{W} \colon \widetilde{W} \to f(\widetilde{W})$ is a homeomorphism.
        \item\label{item: blow-up-mapping-iii} $f$ sends $\Sp \setminus \bigcup_{\alpha\in E(K)} \overline{U_\alpha}$ homeomorphically onto $\Sp \setminus K$. 
        \item\label{item: blow-up-mapping-iv} Let $\beta$ be a Jordan arc in $(\Sp, V(K))$ with $\inter(\beta) \subset \Sp \setminus \bigcup_{\alpha\in E(K)} \overline{U_\alpha}$ or a Jordan curve in $(\Sp, V(K))$ with $\beta \subset \Sp \setminus \bigcup_{\alpha\in E(K)} \overline{U_\alpha}$. Then $f(\beta)\sim \varphi(\beta)$ rel.\ $V(K)$.
    \end{enumerate}
\end{proposition}

The graph $K^\pm$ as above is called the \emph{blow-up} of $K$ under $f$.

\begin{rem}\label{rem: blow-up-triples-unique}
    Using the discussion in Section~\ref{subsubsec: Lifting arcs}, one can show that, if $|\Crit(f)| > 2$, then the  triples $(\alpha^+, \alpha^-, U_\alpha)$, $\alpha\in E(K)$, are uniquely determined by the map $f$ and the graph~$K$. The case $|\Crit(f)| = 2$ is exceptional: In this case, $f$ is combinatorially equivalent to the power map $z\mapsto z^d$ with $d:= \deg(f)\geq 2$, and $E(G)$ consists of exactly $d-1$ (multiple) edges joining the two points in $V(G)=\Crit(f)$.  It is then easy to see that for each planar embedded graph $K$ that is isotopic to $G$ rel.\ $V(G)$ there are (essentially) $d$ distinct choices of the desired triples $(\alpha^+, \alpha^-, U_\alpha)$, $\alpha\in E(K)$. More precisely, for each fixed $\alpha_0\in E(K)$ there are precisely $d$ connected components of $\Sp\setminus f^{-1}(\alpha_0)$, and each such component $U$ induces a collection of triples $(\alpha^+, \alpha^-, U_\alpha)$, $\alpha\in E(K)$, that satisfy conditions \ref{item: blow-up-tripple-i}-\ref{item: blow-up-tripple-iii} with $U_{\alpha_0}=U$.
\end{rem}

\begin{proof}[Proof of Proposition \ref{prop: blow-up-triples-properties}]
    
    First, suppose that $K$ is the planar embedded graph $\varphi(G)$ with the vertex set $V(K)=V(G)$. Note that since $(G,\varphi)$ is admissible, Proposition \ref{prop: Graph isotopic criterion} implies that $K$ is isotopic to $G$ rel.\ $V(G)$. Let $\alpha$ be an edge of $K=\varphi(G)$. Then $\alpha=\varphi(e)$ for some $e\in E(G)$. We may now set $\alpha^+ := \partial D_e^+$, $\alpha^- := \partial D_e^-$,  and $U_\alpha := \inter(D_e)$, where $D_e$, $\partial D_e^+$, and $\partial D_e^-$ are as chosen in the construction of $f$ by blowing up the pair $(G,\varphi)$; see Section~\ref{subsec: The blow-up operation}. One can easily check that all the conditions \ref{item: blow-up-tripple-i}-\ref{item: blow-up-tripple-iii} and \ref{item: blow-up-mapping-i}-\ref{item: blow-up-mapping-iv} are satisfied. In particular, condition \ref{item: blow-up-mapping-iv} follows from \eqref{eq: blow-up_def} and \ref{B1}-\ref{B5}, because $(\varphi\circ h_t)| \beta$, $t\in \I$, provides a non-ambient isotopy rel.\ $V(K)$ between $\varphi(\beta)=(\varphi\circ h_0)(\beta)$ and $f(\beta)=(\varphi\circ h_1)(\beta)$. The general case, when $K$ is an arbitrary planar embedded graph isotopic to $G$, then follows from this by the isotopy lifting property for Thurston maps (see Proposition~\ref{prop: isotopy-lifting}). We leave it to the reader to fill in the details.
\end{proof}

\begin{rem}\label{rem: blow-up-triples-properties}
    An analog of Proposition \ref{prop: blow-up-triples-properties} remains true in the case when a critically fixed Thurston map $f \colon \Sp \to \Sp$ is obtained by blowing up a (not necessarily admissible) pair $(G, \varphi)$. 
    Namely, the statement still holds if we take $K$ to be a planar embedded graph in $\Sp$ that is isotopic to $\varphi(G)$ rel.\ $V(G)$ and replace condition \ref{item: blow-up-tripple-i} with the following one:
     \begin{enumerate}[label=\normalfont{(D\arabic*')}]
        \item \label{item: new-blow-up-tripple-i}$
        \alpha^+$ and $\alpha^-$ are distinct lifts of $\alpha$ with $\partial \alpha^+=\partial \alpha^-$.
     \end{enumerate}
\end{rem}

We are now ready to state the following ``blow-up criterion'' for critically fixed Thurston maps. 

\begin{proposition}\label{prop: Blow-up triples}
    Let $f\colon\Sp\to\Sp$ be a critically fixed Thurston map and $K$ be a planar embedded graph in $\Sp$ with $V(K) = \Crit(f)$ and  $|E(K)| = \deg(f) - 1$. Suppose that for each $\alpha \in E(K)$ there is a triple $(\alpha^+, \alpha^-, U_\alpha)$ satisfying conditions \ref{item: blow-up-tripple-i}-\ref{item: blow-up-tripple-iii}.

    Then $f$ is obtained by blowing up an admissible pair $(G, \varphi)$, where $G$ is a planar embedded graph isotopic to $K$ rel.\ $\Crit(f)$. Furthermore, $f$ is isotopic to a critically fixed Thurston map obtained by blowing up the (admissible) pair $(K,\varphi)$.

\end{proposition}

\begin{proof} Suppose $f\colon\Sp\to\Sp$ is a critically fixed Thurston map and $K$ is a planar embedded graph in $\Sp$ satisfying the conditions in the statement. We will view the set $K^\pm:= \bigcup_{\alpha\in E(K)}(\alpha^+\cup \alpha^-)$ as a planar embedded graph with $V(K^\pm)=V(K)=\Crit(f)$. Then $K^\pm$ is a subgraph of the complete preimage~$f^{-1}(K)$. (Note that, unless $\deg(f) = 2$, $K^{\pm}$ is a proper subgraph of $f^{-1}(K)$.)

\begin{claim1}
$f(U_\alpha) \supset \Sp \setminus \alpha$ for each $\alpha\in E(K)$.
\end{claim1}

Indeed, fix a point $q\in U_\alpha\setminus f^{-1}(\alpha)$ and suppose $p\in \Sp \setminus \alpha$ is arbitrary. We may connect $p$ and $f(q)$ by a Jordan arc $\beta$ inside $\Sp \setminus \alpha$. Then there is a lift $\widetilde \beta$ of $\beta$ under $f$ connecting the point $q$ and a point $\widetilde p\in f^{-1}(p)$ (see, for example, \cite[Lemma~A.18]{THEBook}). Note that $\widetilde \beta$ must stay inside $U_\alpha$, and thus $\widetilde p \in U_\alpha$. Claim 1 follows.

\medskip

Note that Claim 1 implies that $\deg(f,q)\geq \deg_K(q)+1$ for all $q\in V(K)=\Crit(f)$.

\begin{claim2}
Suppose $H$ is a component of $f^{-1}(K)\setminus K^\pm$. Then $H\subset U_\alpha$ for some $\alpha\in E(K)$.
\end{claim2}

This is an easy counting argument. Let $p\in H$ be arbitrary, and $q:=f(p)\in K$. Then either $q\in \inter(\alpha_q)$ for some $\alpha_q\in E(K)$ or $q\in V(K)$. In the first case, Claim 1 implies that $q$ has at least one preimage in each $U_\alpha$ for $\alpha \in E(K)\setminus \{\alpha_q\}$. At the same time, $q$ has two preimages in $\partial U_{\alpha_q}=\alpha_q^+\cup \alpha_q^-\subset K^\pm$. Since $|E(K)|=\deg(f)-1$, the point $p$ must be in one of the regions $U_\alpha$ with $\alpha \neq \alpha_q$. 
In the latter case, $q$ is a fixed critical point with $\deg(f,q)\geq \deg_K(q)+1$. At the same time, Claim 1 implies that $q$ has at least one preimage in every $U_\alpha$ for which $\alpha\in E(K)$ is not incident to~$q$. Again, since $|E(K)|=\deg(f)-1$, the point $p$ must be in one of these regions $U_\alpha$. It follows that $p\in U_{\alpha_0}$ for some $\alpha_0\in E(K)$ in both cases. Since $H\ni p$ is connected and $\partial U_{\alpha_0}\subset K^\pm$, we conclude that $H\subset U_{\alpha_0}$.

\medskip

The proof of Claim 2 implies that $\deg(f,q) = \deg_K(q)+1$ for each $q\in \Crit(f)$. In particular, the graph $K$ has no isolated vertices. Furthermore, for every $\alpha\in E(K)$ we have $f(U_\alpha)\subset \Sp\setminus \alpha$, and thus $f(U_\alpha)=\Sp\setminus \alpha$ by Claim 1.

\begin{claim3} The map $f$ sends each face $\widetilde W$ of $K^\pm$ homeomorphically onto its image. More precisely, $f$ satisfies conditions \ref{item: blow-up-mapping-i}-\ref{item: blow-up-mapping-iii}. 
\end{claim3}

The proof is again based on a counting argument. First suppose that $\widetilde W$ is a face of $K^\pm$ distinct from each $U_\alpha$, $\alpha\in E(K)$. It easily follows from the Euler formula that there are exactly $|F(K)|$ such faces. 
Claim 2 implies that $\widetilde W$ is also a face of $f^{-1}(K)$. Thus, $f(\widetilde W)$ is a face of $K$ (with $\partial f(\widetilde{W})=f(\partial\widetilde{W})$) and $f|\widetilde W\colon \widetilde W \to f(\widetilde W)$ is a covering map. Since $|E(K)|=\deg(f)-1$, Claim~1 implies that $f$ maps $\Sp \setminus \bigcup_{\alpha\in E(K)} \overline{U_\alpha}$ injectively into $\Sp \setminus K$. 
So $f$ has to satisfy \ref{item: blow-up-mapping-ii} and \ref{item: blow-up-mapping-iii}.

Now suppose that $\widetilde W = U_\alpha$ for some $\alpha\in E(K)$. Since $f(U_\alpha)=\Sp\setminus \alpha$ and $f(\partial U_\alpha)=\alpha$, we conclude that $U_\alpha$ is a component of $f^{-1}(\Sp\setminus \alpha)$. Note that $U_\alpha\cap \Crit(f) = \emptyset$. Thus, by the Riemann-Hurwitz formula, $\deg(f|U_\alpha)=1$, and so $f$ satisfies \ref{item: blow-up-mapping-i}. 
Claim 3 follows.

\medskip

For every $\alpha\in E(K)$, let us now choose a Jordan arc $e_\alpha$ with $\inter(e_\alpha)\subset U_\alpha$ and $\partial e_\alpha = \partial \alpha$. Let $G$ be the planar embedded graph with the vertex set $V(G)=\Crit(f)$ and the edge set $E(G)=\{e_\alpha\colon \alpha\in E(K)\}$. Lemma \ref{lem: Buser_isotopy_arcs} and Proposition \ref{prop: Graph isotopic criterion} imply that $G$ is isotopic to $K$ rel.\ $\Crit(f)$. We will prove that $f$ is obtained by blowing up a pair $(G,\varphi)$, where $\varphi\in \Homeo^+(\Sp,\Crit(f))$.

For each edge $e=e_\alpha \in E(G)$, we set $D_e:=\overline{U_\alpha}$, $\partial D_e^+ := \alpha^+$, $\partial D_e^- := \alpha^-$, and $f_e:= f|\overline{U_\alpha}$. We also choose open Jordan regions $W_e \supset D_e \setminus \partial e$ satisfying conditions~\ref{A1}-\ref{A3}. Next, we consider a homotopy $h\colon \Sp\times \I \to \Sp$ that satisfies properties \ref{B1}-\ref{B5} together with the following extra condition:
\begin{enumerate}[label=($*$)]
    \item\label{B6} for every $e\in E(G)$ and arbitrary $x\in \partial D_e^+$ and $y\in \partial D_e^-$, $h_1(x)=h_1(y)$ if and only if $f(x)=f(y)$. 
\end{enumerate}

Let $\varphi\colon \Sp\setminus G \to \Sp\setminus K$ be the map defined by $\varphi(p):=f\circ h_1^{-1}(p)$ for each $p \in \Sp \setminus G$. Note that condition \ref{B4} and Claim 3 imply that $\varphi$ is a homeomorphism. 

\begin{claim4}
The map $\varphi\colon \Sp\setminus G \to \Sp\setminus K$ extends to a homeomorphism $\varphi\in\Homeo^+(\Sp,\Crit(f))$ so that $\varphi(e_\alpha) = \alpha$ for all $\alpha\in E(K)$.
\end{claim4}

We only give an outline of the argument and leave some details to the reader. 
First, we claim that the inverse map $\varphi^{-1}\colon \Sp \setminus K \to \Sp \setminus G$ may be extended to a continuous bijection $\varphi^{-1}\colon \Sp \to \Sp$. Indeed, let $p\in K$ be arbitrary. Then $p\in \alpha$ for some $\alpha\in E(K)$. We set $\varphi^{-1}(p):=h_1(q)$ for $q\in f^{-1}(p)\cap  (\alpha^+\cup \alpha^-)$. Condition \ref{B6} ensures that $\varphi^{-1}(p)$ is well-defined, and condition \ref{B5} ensures that $\varphi^{-1}|\alpha\colon \alpha \to e_\alpha$ is a homeomorphism for every $\alpha\in E(K)$. Thus $\varphi^{-1}\colon \Sp \to \Sp$ is a bijection. The continuity of $\varphi^{-1}$ can now be easily deduced from the mapping properties of $h_1$ and $f$. This implies that $\varphi^{-1}\colon \Sp\to\Sp$ is a homeomorphism. The rest follows from the construction of $\varphi$. 

\medskip

It is now straightforward to check that $f$ is obtained by blowing up the pair $(G, \varphi)$ with the choices above. It also follows from Proposition \ref{prop: Isotopy and blow-up} that $f$ is isotopic to a critically fixed Thurston map obtained by blowing up the pair $(K, \varphi)$. Finally, since $\alpha\sim e_\alpha$ rel.\ $\Crit(f)$ for each $\alpha\in E(K)$, we have $\varphi(\alpha) \sim \varphi(e_\alpha)=\alpha$ rel.\ $\Crit(f)$. Hence the pairs $(G,\varphi)$ and $(K,\varphi)$ are admissible. This completes the proof of Proposition \ref{prop: Blow-up triples}.
\end{proof}

Following the proof of Proposition \ref{prop: Blow-up triples}, one can also show the next statement. 
\begin{proposition}\label{prop: Blow-up triples-general}
    Let $f\colon\Sp\to\Sp$ be a critically fixed Thurston map, and let $K$ be a planar embedded graph in $\Sp$ with $|E(K)| = \deg(f) - 1$ and $f(v)=v$ for all $v\in V(K)$. Suppose that for each $\alpha \in E(K)$ there is a triple $(\alpha^+, \alpha^-, U_\alpha)$ satisfying conditions \ref{item: new-blow-up-tripple-i}, \ref{item: blow-up-tripple-ii}, and \ref{item: blow-up-tripple-iii}.  
    
    Let $K^+$ be the planar embedded graph with the vertex set $V(K^+)=V(K)$ and the edge set $E(K^+)=\{\alpha^+\colon \alpha\in E(K)\}$. Then $f$ is obtained by blowing a (not necessarily admissible) pair $(G, \varphi)$, where $G$ is a planar embedded graph isotopic to $K^+$ rel.\ $V(K)$. Furthermore, if $g$ is a critically fixed Thurston map obtained by blowing up the pair $(K^+,\varphi)$, then the marked Thurston maps $(f, V(G))$ and $(g, V(K^+))$ are isotopic.
\end{proposition}

\subsubsection{Arc lifting}\label{subsubsec: Lifting arcs}

In the following, $f\colon \Sp \to \Sp$ is a critically fixed Thurston map obtained by blowing up an admissible pair $(G, \varphi)$ in $\Sp$. We also suppose that the triples $(e^+,e^-,U_e)$, $e\in E(G)$, 
are as provided by Proposition \ref{prop: blow-up-triples-properties} for $K=G$. 
Our goal here is to prove several facts about lifts of Jordan arcs in $(\Sp, \Crit(f))$ under the map~$f$. First, we introduce the following notion.

\begin{definition}
     Let $f \colon \Sp \to \Sp$ be a critically fixed Thurston map and $\alpha$ be a Jordan arc in $(\Sp, \Crit(f))$. Suppose $\alpha$ has exactly $k$ distinct lifts under $f$ that are isotopic to $\alpha$ rel.\ $\Crit(f)$. Then $k$ is called the \emph{blow-up degree} of $\alpha$ under $f$ 
     and denoted by $\deg(f, \alpha)$. 
If $\deg(f, \alpha) \geq 2$, then we say that the Jordan arc~$\alpha$ \emph{blows up} under $f$.
\end{definition}

One can easily see that if two Jordan arcs $\alpha$ and $\alpha'$ in $(\Sp, \Crit(f))$ are isotopic rel.\ $\Crit(f)$, then their blow-up degrees under $f$ coincide, i.e., $\deg(f, \alpha) = \deg(f, \alpha')$.

\begin{lemma}\label{lem: Fixed and blow-up arcs}
    Let $f\colon \Sp \to \Sp$ be a critically fixed Thurston map obtained by blowing up an admissible pair $(G, \varphi)$. Suppose $\alpha$ is a Jordan arc in $(\Sp, \Crit(f))$ such that $\deg(f, \alpha) > 0$. Then $i_{C(f)}(\alpha, e) = 0$ for each edge $e \in E(G)$.
\end{lemma}

\begin{proof}
    Without loss of generality, we may assume that $\alpha$ and $e\in E(G)$ are in minimal position rel.\ $\Crit(f)$. Suppose that $\widetilde{\alpha}$ is a lift of $\alpha$ under $f$ that is isotopic to $\alpha$ rel.\ $\Crit(f)$. 
    Then by~\ref{item: blow-up-tripple-i} we have
    \begin{align*} 
    i_{C(f)}(\alpha, e) &= |\alpha \cap \inter(e)| = |\widetilde{\alpha} \cap f^{-1}(\inter(e))| \\
    & \geq  |\widetilde{\alpha}\cap \inter(e^+)| +  |\widetilde{\alpha}\cap \inter(e^-)|  \\
    &\geq i_{C(f)}(\widetilde{\alpha}, e^+) + i_{C(f)}(\widetilde{\alpha}, e^-) = 2\, i_{C(f)}(\alpha, e). 
    \end{align*}
    It follows that $i_{C(f)}(\alpha, e) = 0$ for each edge $e$ of $G$, as desired.
\end{proof}

\begin{lemma}\label{lem: Blow-up arcs}
    Let $f\colon \Sp \to \Sp$ be a critically fixed Thurston map obtained by blowing up an admissible pair $(G, \varphi)$. Then a Jordan arc $\alpha$ in $(\Sp,\Crit(f))$ blows up under $f$ if and only if $\alpha$ is isotopic to an edge of $G$ rel.\ $\Crit(f)$.
\end{lemma}

\begin{proof} 
    Each edge $e$ of $G$ blows up under $f$, because $e^+$ and $e^-$ are isotopic to $e$ rel.\ $\Crit(f)$. Thus, if $\alpha$ is a Jordan arc isotopic to $e$ rel.\ $\Crit(f)$, then it blows up as well. 
    
    Now suppose that $\alpha$ is a Jordan arc in $(\Sp, \Crit(f))$ that blows up under $f$. It follows from Lemma~\ref{lem: Fixed and blow-up arcs} that we may assume that $\alpha$ intersects the graph $G$ only in its vertices. Then by \ref{item: blow-up-mapping-i} each of the $\deg(f)-1$ regions $\overline{U_e}$, $e\in E(G)$, contains exactly one lift $\widetilde \alpha_e$ of $\alpha$ under~$f$. Since $\alpha$ blows up, 
    one of these lifts $\widetilde \alpha_e$ has to be isotopic to $\alpha$ rel.\ $\Crit(f)$. This is possible only if $\partial \widetilde \alpha_e = \partial e$, which implies that $e^+ \sim \widetilde \alpha_e$ rel.\ $\Crit(f)$ by Lemma \ref{lem: Buser_isotopy_arcs}. Hence $e \sim e^+\sim \widetilde \alpha_e \sim \alpha$ rel.\ $\Crit(f)$. This finishes the proof of the lemma.  
   \end{proof}

\begin{lemma}\label{lem: Absence of intersections}
    Let $f\colon \Sp \to \Sp$ be a critically fixed Thurston map obtained by blowing up an admissible pair $(G, \varphi)$ and $e$ be an edge of $G$. Suppose $\beta $ is a Jordan arc in $\Sp$ with $\partial \beta = \partial e$ (and possibly with critical points of $f$ in its interior) that satisfies $i_{C(f)}(\beta, e) = 0$. Then $\beta$ has a lift $\widetilde \beta$ under $f$ that is isotopic to $e$ rel.\ $\Crit(f)$.
\end{lemma}

\begin{proof}
    Without loss of generality, we may assume that  $\beta \cap e = \partial e$. Property \ref{item: blow-up-mapping-i} implies that there is a lift $\widetilde \beta$ of $\beta$ under $f$ with $\inter(\widetilde \beta) \subset U_e$ and $\partial \widetilde \beta = \partial e$. Again we have $\widetilde \beta \sim e^+ \sim e$ rel.\ $\Crit(f)$. This completes the proof.
\end{proof}

The following lemma provides a quantitative version of Lemma \ref{lem: Blow-up arcs}. (Recall that $m_G(\alpha)$ denotes the total number of edges of $G$ that are isotopic to $\alpha\in E(G)$ rel.\ $V(G)$.)

\begin{lemma}\label{lem: Same blow-up degrees}
    Let $f\colon \Sp \to \Sp$ be a critically fixed Thurston map obtained by blowing up an admissible pair $(G, \varphi)$. Then for each $\alpha \in E(G)$ we have $\deg(f, \alpha) = m_G(\alpha) + 1$. 
    
    More generally, for every Jordan arc $\beta$ in $(\Sp, \Crit(f))$ there are exactly $\max(0, \deg(f, \beta) - 1)$ edges in $G$ that are isotopic to $\beta$ rel.\ $\Crit(f)$.
\end{lemma}

\begin{proof}
Properties~\ref{item: blow-up-tripple-i}, \ref{item: blow-up-mapping-i}, and \ref{item: blow-up-mapping-iii} imply that each $\alpha \in E(G)$ has the following lifts under $f$: $\alpha^+$, $\alpha^-$, and a unique lift $\widetilde \alpha_e$ with $\inter(\widetilde \alpha_e) \subset U_e$ for every $e\in E(G)\setminus\{\alpha\}$. Note that $\partial \widetilde\alpha_e = \partial e$ if and only if $\partial \alpha=\partial e$. It follows that $\widetilde \alpha_e$ is isotopic to $\alpha$ rel.\ $\Crit(f)$ if and only if $e$ is isotopic to $\alpha$ rel.\ $\Crit(f)$, and therefore $\deg(\alpha, f) = m_G(\alpha) + 1$. 

To show the second statement, suppose $\beta$ is an arbitrary Jordan arc in $(\Sp, \Crit(f))$. If $\beta$ is isotopic to an edge $\alpha\in E(G)$ rel.\ $\Crit(f)$, then $\deg(f, \beta)=\deg(f, \alpha)=m_G(\alpha) + 1\geq 2$ by the discussion above, and the statement follows. Otherwise, $\beta$ does not blow up under $f$ by Lemma~\ref{lem: Blow-up arcs}, and thus $\deg(f,\beta)\in \{0,1\}$,  which also implies the desired statement.
\end{proof}

\subsubsection{Combinatorial and isotopy equivalence}\label{subsubsec: adm_equivalence} 
It is straightforward to check that the notion of equivalence for admissible pairs from the introduction (see Definition \ref{def: admissibility_and_equivalence}) is equivalent to the following one.

\begin{definition}\label{def: admissibility_new}
          Let $(G, \varphi)$ and $(G', \varphi')$ be two admissible pairs in topological $2$-spheres $\Sp$ and $\widehat{\Sp}$, respectively. We say that $(G, \varphi)$ and $(G', \varphi')$ are \emph{equivalent} if there exists an orientation-preserving homeomorphism $\psi\colon \Sp \to \widehat{\Sp}$ such that the pairs $(G', \varphi')$ and $\big(\psi(G), \psi \circ \varphi \circ \psi^{-1}\big)$ are isotopic. 
\end{definition}

The next proposition shows that the notions of combinatorial equivalence and isotopy for critically fixed Thurston maps obtained by blowing up admissible pairs agree with the notions of equivalence and isotopy for admissible pairs.

\begin{proposition}\label{prop: admis_equiv}
    Let $f\colon \Sp \to \Sp$ and $f'\colon \widehat{\Sp} \to \widehat{\Sp}$ be critically fixed Thurston maps obtained by blowing up admissible pairs $(G, \varphi)$ and $(G', \varphi')$ in $\Sp$ and $\widehat{\Sp}$, respectively.
    Then the following statements are true:
    \begin{enumerate}[label=\normalfont{(\roman*)}]
        \item \label{item: case comb equiv} $f$ is combinatorially equivalent to $f'$ if and only if $(G, \varphi)$ is equivalent to $(G', \varphi')$;
        \item \label{item: case isotopy} Suppose $\widehat {\Sp} = \Sp$. Then the maps $f$ and $f'$ are isotopic if and only if the pairs $(G, \varphi)$ and $(G', \varphi')$ are isotopic.
    \end{enumerate}    
\end{proposition}

\begin{proof} Suppose $f\colon \Sp\to \Sp$ and $f'\colon \widehat{\Sp} \to \widehat{\Sp}$ are as in the statement.

\ref{item: case comb equiv} First note that $f$ is combinatorially equivalent to $f'$ if and only if there exists an orientation-preserving homeomorphism $\psi \colon \Sp \to \widehat{\Sp}$ such that $g := \psi \circ f \circ \psi^{-1}$ and $f'$ are isotopic. The proof of Proposition \ref{prop: comb_equiv_blowups} implies that the map $g$ is obtained by blowing up the pair $(\widehat G, \widehat \varphi)$, where $\widehat G$ is the planar embedded graph $\psi(G)$ in $\widehat{\Sp}$ with the the vertex set $V(\widehat G) = \psi(V(G)) = V(G')$ and $\widehat \varphi = \psi \circ \varphi \circ \psi^{-1}$. Moreover, the pair $(\widehat G, \widehat \varphi)$ is admissible in $\widehat{\Sp}$ and equivalent to $(G,\varphi)$ (see Remark \ref{rem: comb_equiv_and_adm_pairs}\ref{item: comb_equiv_and_adm_pairs_1}).  These observations allow us to reduce \ref{item: case comb equiv} to \ref{item: case isotopy}.

\ref{item: case isotopy}   Now suppose that $\widehat{\Sp} = \Sp$. Proposition \ref{prop: Isotopy and blow-up} immediately gives the ``if''-direction. 

For the converse direction, assume that the maps $f$ and $f'$ are isotopic. By Proposition~\ref{prop: isotopy-lifting} every edge $\alpha \in E(G)$ blows up under $f'$. Lemma \ref{lem: Same blow-up degrees} then implies that $\alpha$ is isotopic to an edge $\alpha'$ of $G'$ and $m_{G}(\alpha)=m_{G'}(\alpha')$. 
Since $|E(G)| = \deg(f) - 1 = \deg(f') - 1 = |E(G')|$ by Proposition \ref{prop: Blow-up properties}\ref{item: degree-of-blow-up}, it follows from Proposition~\ref{prop: Graph isotopic criterion} that $G$ and $G'$ are isotopic rel.\ $V(G)$.

It is now left to prove that $\varphi$ and $\varphi'$ are isotopic rel.\ $V(G)$. Let $g$ be a critically fixed Thurston map obtained by blowing up the pair $(G, \varphi')$.  Since $f$ and $f'$ are isotopic, Propositions \ref{prop: Isotopy and blow-up} implies that $f$ and $g$ are isotopic as well. Therefore, we may write $f = g \circ \psi$ for some homeomorphism $\psi\in \Homeo_0^+(\Sp, V(G))$. Let $(\alpha^+,\alpha^-,U_\alpha)$, $\alpha\in E(G)$, 
be the triples provided by Proposition~\ref{prop: blow-up-triples-properties} for the map $f$ and the graph $G$. Then $(\psi(\alpha^+),\psi(\alpha^-),\psi(U_\alpha))$, $\alpha\in E(G)$, are the respective triples for the map $g$ and the graph $G$. In particular, $G^\pm_g = \psi(G^\pm_{f})$, where $G^\pm_f = \bigcup_{\alpha\in E(G)} (\alpha^+\cup \alpha^-)$ and $G^\pm_{g}$ are the blow-ups of $G$ under $f$ and~$g$, respectively. Choose a connected graph $H \subset \Sp \setminus \bigcup_{\alpha\in E(G)} (\overline U_\alpha \setminus \partial \alpha)$ with $V(H) = V(G)$. Then, \ref{item: blow-up-mapping-iv} implies that $(g\circ \psi)(e)= f(e)\sim \varphi(e)$ rel.\ $V(G)$ for every $e\in E(H)$. At the same time, since  $\inter(\psi(e))\subset \Sp \setminus \bigcup_{\alpha\in E(G)} \psi(\overline U_\alpha)$, we get $(g\circ \psi)(e)= g(\psi(e))\sim \varphi'(\psi(e)) \sim \varphi'(e)$ rel.\ $V(G)$. Corollary \ref{cor: Homeo rigidity} now implies that the homeomorphisms $\varphi$ and $\varphi'$ are isotopic rel.\ $V(G)$, which finishes the proof of \ref{item: case isotopy}. 
\end{proof}

\begin{rem}\label{rem: adm-pair-for-connected}
Suppose we are in the setting of Proposition \ref{prop: admis_equiv}, and the graph $G$ is connected. By admissibility, $\varphi\in \Homeo^+_0(\Sp, V(G))$ (see Corollary~\ref{cor: Homeo rigidity}). It follows that in this case the maps $f$ and $f'$ are combinatorially equivalent if and only if the graphs $G$ and $G'$ are isomorphic. Similarly, when $\widehat{\Sp} = \Sp$, the maps $f$ and $f'$ are isotopic if and only if the graphs $G$ and $G'$ are isotopic rel.\ $V(G)$.
\end{rem}

\subsection{Rational case}\label{subsec: Rational case}

Let $f\colon \widehat{\C}\to \widehat{\C}$ be a critically fixed rational map and $c \in \Crit(f)$ be a critical point of $f$ (in the following, we always assume that $\deg(f) \geq 2$). Note that $c$ is a superattracting fixed point of $f$, and thus all points in a neighbourhood of $c$ converge to $c$ under iteration. 
The \emph{basin of attraction} of~$c$ is defined to be the set $$B_c := \{z\in \widehat{\C}\colon \lim_{n\to\infty} f^n(z) =c\}.$$
The connected component of $B_c$ 
containing the point $c$ is called the \emph{immediate basin} of $c$ and denoted by $\Omega_c$. 
It follows from \cite[Theorem 9.3]{Milnor_Book} that $\Omega_c$ is a simply connected open set and there exists a biholomorphic map $\tau_c \colon \D \to \Omega_c$ such that
$$
    (\tau_c^{-1} \circ f \circ \tau_c)(z) = z^{d_c},
$$
where $d_c := \deg(f, c)$. Furthermore, the map $\tau_c$ extends to a continuous and surjective map $\tau_c\colon \overline{\D} \to \overline{\Omega_c}$, which provides a semi-conjugacy between the power map $z \mapsto z^{d_c}$ on $\overline{\D}$ and the map $f$ on $\overline{\Omega_c}$.

The \emph{internal ray of angle $\theta \in \R/\Z$} in the immediate basin $\Omega_c$ is the image of the radial arc $r(\theta) := \{t e^{2\pi i\theta}\colon t \in \I\}$ under the map $\tau_c$. The point $\tau_c(e^{2\pi i\theta}) \in \partial \Omega_c$ is called the \emph{landing point} of this ray. 
Note that the internal ray of angle $\theta$ is fixed under $f$ (i.e., $f(r(\theta))=r(\theta)$) if and only if $\theta \equiv \frac{j}{d_c - 1}\,\,\, \mod\, \Z$ for some $j \in \{0, \dots, d_c - 2\}$. The landing points of such rays are repelling fixed points of the map $f$.

The \emph{Tischler graph} of a critically fixed rational map $f$ is the planar embedded graph $\Tisch(f)$
whose edge set consists of the fixed internal rays taken in the immediate basins $\Omega_c$ for all $c \in \Crit(f)$ and whose vertex set consists of the endpoints of all these rays (which are the landing points of the rays together with the critical points of $f$). That is, as a subset of~$\widehat{\C}$, $\Tisch(f)$ is the union of all fixed internal rays described in the previous paragraph.

Let us denote by $\Fix(f)$ the set of all fixed points of a (critically fixed) rational map~$f$. Recall that the holomorphic fixed point formula implies that $|\Fix(f)|=\deg(f)+1$, if counted with multiplicity. 

 The Tischler graph has the following properties, see \cite[Theorem 1 and Corollary 6]{H_Tischler} and \cite[Lemma 3]{Pilgrim_crit_fixed}.

\begin{proposition}
    Let $f$ be a critically fixed rational map. 
    Then the following statements are true:
    \begin{enumerate}[label=\normalfont{(\roman*)}]
        \item The vertex set of $\Tisch(f)$ consists of all fixed points of $f$. In particular, \newline $|V(\Tisch(f))|= |\Fix(f)| = \deg(f)+1$.
        
        \item $\Tisch(f)$ is a bipartite graph: each edge of $\Tisch(f)$ connects a superattracting fixed point and a repelling fixed point of $f$. 
        
        \item $\Tisch(f)$ is connected.
        
        \item The boundary $\partial W$ of each face $W$ of $\Tisch(f)$ is either a quadrilateral or a bigon with a sticker inside (see Figure \ref{fig: Faces of Tischler graph}).

        \item $\Tisch(f)$ has exactly $\deg(f)-1$ faces.
    \end{enumerate}
\end{proposition}

The proposition above justifies the next definition. 

\begin{definition}\label{def: rational charge graph}
     Let $f$ be a critically fixed rational map. 
     For each face $W$ of the Tischler graph $\Tisch(f)$, choose a Jordan arc $e(W)$ joining the (only) two critical points of $f$ on $\partial W$ so that $\inter(e(W)) \subset W$ (see Figure \ref{fig: Faces of Tischler graph} for an illustration). Let $G$ be the planar embedded graph with the vertex set $\Crit(f)$ and the edge set $\{e(W)\colon W \in F(\Tisch(f))\}$. Then any planar embedded graph isotopic to $G$ rel.\ $\Crit(f)$ is called the \emph{charge graph} of $f$ and denoted by $\Charge(f)$.
\end{definition}

\begin{figure}[t]
    \centering           \includegraphics[width = 10cm]{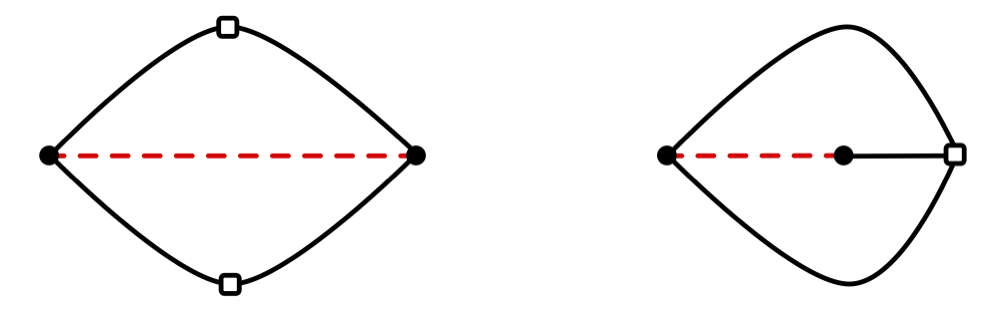}

    \caption{Constructing an edge $e(W)$ of $\Charge(f)$ inside a face $W$ of $\Tisch(f)$ if $\partial W$ is a quadrilateral (left) and if $\partial W$ a bigon with a sticker inside (right). The boundary of $W$ consists of black edges; the critical points of $f$ on $\partial W$ are represented by black dots; the repelling fixed points of $f$ on $\partial W$ are represented by white squares; and the edge $e(W)$ is represented by a dashed red line.}
    
    \label{fig: Faces of Tischler graph}
\end{figure}

We point out that the Tischler graph $\Tisch(f)$ of a critically fixed rational map $f$ 
is uniquely defined, while the charge graph $\Charge(f)$ is defined only up to isotopy rel.\ $\Crit(f)$. We also note that $\Charge(f)$ has exactly $\deg(f)-1$ edges and it is always connected (see the remark after \cite[Lemma 8]{H_Tischler}).

The next statement relates critically fixed rational maps and their charge graphs via the blow-up operation. (It immediately follows from Proposition~\ref{prop: Isotopy and blow-up} and \cite[Proposition~7]{H_Tischler}.)

\begin{proposition}\label{prop: charge graph rational case}%[{\cite[Proposition 7]{H_Tisch}}]
    Let $f$ be a critically fixed rational map 
    and $g$ be a critically fixed Thurston map obtained by blowing up the pair $(\Charge(f), \id_{\widehat{\C}})$. Then the maps $f$ and $g$ are isotopic.
\end{proposition}

\begin{rem}\label{rem: isotopy_rel_fixed_points}
In fact, Proposition~\ref{prop: Isotopy and blow-up} and \cite[Proposition~7]{H_Tischler} imply a slightly stronger result. Suppose $f$ is a critically fixed rational map and $G$ is the planar embedded graph in $\widehat{\C}$ constructed in Definition~\ref{def: rational charge graph} from the Tischler graph of $f$. Consider the planar embedded graph $G':=G \cup \Fix(f)$ with the vertex set $\Fix(f)$. We note that each face of $G$ contains exactly one point from $\Fix(f)\setminus\Crit(f)$; see \cite[Lemma 8]{H_Tischler}. Let $g'$ be a critically fixed Thurston map obtained by blowing up the pair $(G', \id_{\widehat{\C}}).$ Then the marked Thurston maps $(f, \Fix(f))$ and $(g', \Fix(f))$ are isotopic.
\end{rem}

The following converse to Proposition \ref{prop: charge graph rational case} easily follows from \cite[Corollary~3]{PT}.

\begin{proposition}[{\cite[Theorem 9]{Pilgrim_crit_fixed}}]\label{prop: rational if charge connected}
    Let $f$ be a critically fixed Thurston map obtained by blowing up a pair $(G, \id_{\Sp})$, where $G$ is a planar embedded graph in $\Sp$ without isolated vertices. Then $f$ is combinatorially equivalent to a rational map if and only if $G$ is connected. 
\end{proposition}

\begin{remark} 
The following stronger statement can be easily derived from the discussion in Remark \ref{rem: isotopy_rel_fixed_points}. Let $f$ be a critically fixed Thurston map obtained by blowing up a pair $(G, \id_{\Sp})$. Then the marked Thurston map $(f,V(G))$ is realized (by a marked rational map) if and only $G$ has exactly one non-trivial connected component $H$ and each face of $H$ contains at most one isolated vertex of~$G$. Here, a connected graph $H$ is called \emph{non-trivial} if $H$ has at least one edge.
\end{remark}

The family of critically fixed rational maps may be completely classified using their charge graphs. Namely, Theorem \ref{thm: class_criti_fix_rational} is an immediate corollary of the following result.

\begin{proposition}[{\cite[Section 5]{H_Tischler}}, compare Remark \ref{rem: adm-pair-for-connected}]\label{prop: Classification of critically fixed rational maps}
    Two critically fixed rational maps $f$ and $g$ are combinatorially equivalent if and only if their charge graphs $\Charge(f)$ and $\Charge(g)$ are isomorphic.
\end{proposition}

\begin{example}\label{ex: Rational square map}

    Let us consider the following rational map
    $$
        F_{\square}(z) = \frac{3z^5 - 20z}{5z^4 - 12}.
    $$
    One can easily check that $\Crit(F_{\square})=\{1 + i, -1 + i, -1 - i, 1 - i\}$ and $\Fix(F_{\square})=\Crit(F_{\square})\sqcup \{0,\infty\}$, that is, $F_{\square}$ is a critically fixed rational map. Furthermore, we have $\deg(F_{\square}) = 5$ and $\deg(F_{\square},c) = 3$ for every $c\in \Crit(F_{\square})$.
    
    The attracting basins and the Tischler graph of $F_{\square}$ are shown on the left in  Figure~\ref{fig: Rational square map}. Namely, the basins of attraction of the critical points $1 + i$, $-1 + i$, $-1 - i$, and $1 - i $ (indicated by the black dotes in the picture) are shown in blue, red, orange, and green colors, respectively. The white square corresponds to the repelling fixed point at $0$. The black arcs represent the edges of the Tischler graph of $F_\square$; note that $\Tisch(F_\square)$ has a vertex at $\infty$ connected to all critical points. The right part of Figure \ref{fig: Rational square map} illustrates the construction of the charge graph of $F_\square$. Here, the edges of $\Charge(F_\square)$ and $\Tisch(F_\square)$ are shown in solid and dashed black lines, respectively.

    \begin{figure}[t]
        \centering           
        
        \includegraphics[width=5cm]{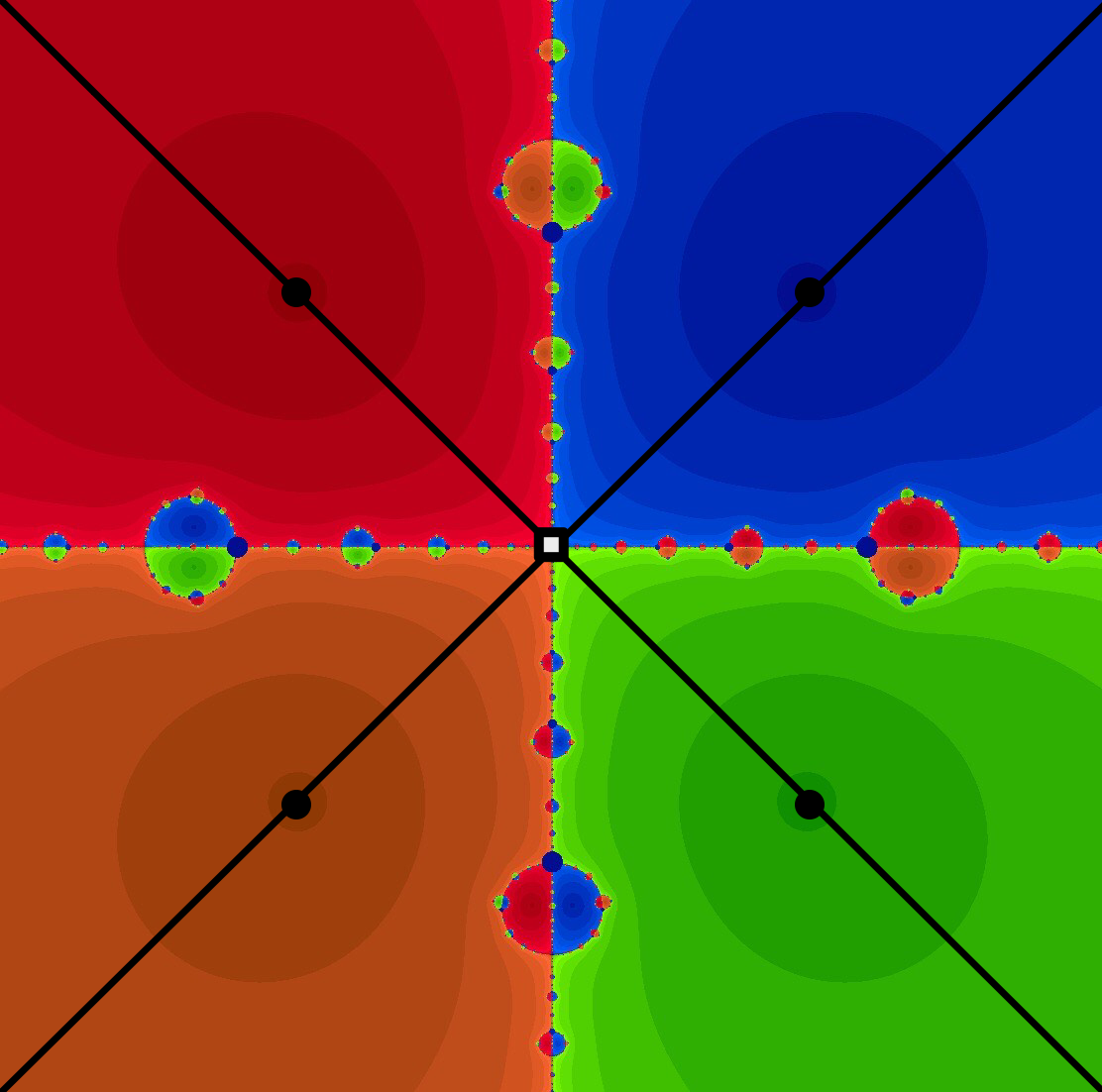} \hspace{4cm}
        \includegraphics[width=5cm]{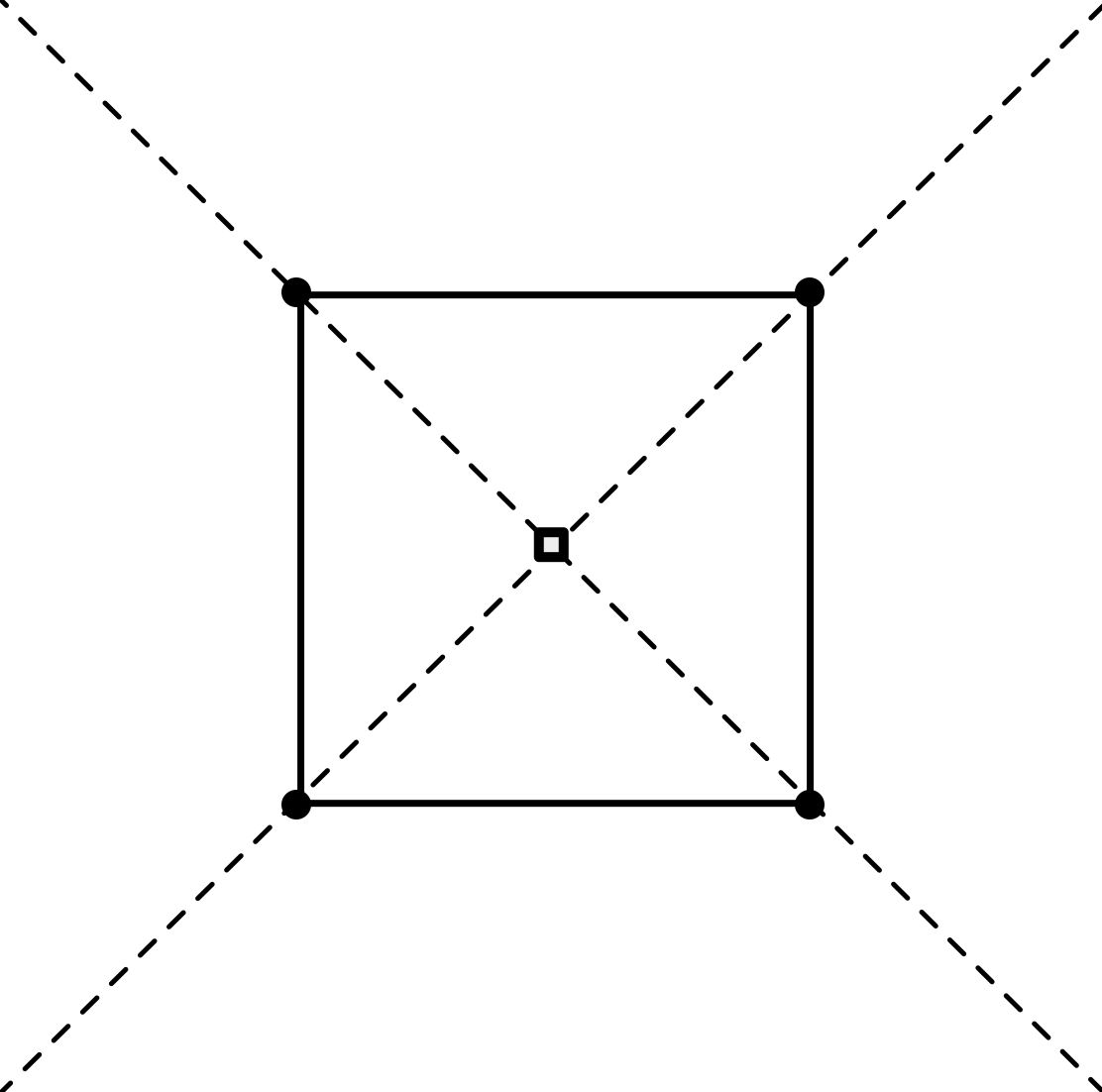}

        \caption{The Tischler graph of $F_{\square}$ (left) and the construction of the charge graph of $F_{\square}$ from $\Tisch(F_{\square})$ (right).}
        
        \label{fig: Rational square map}
    \end{figure}

    Note that $\Charge(F_{\square})$ is isomorphic to the square graph $G_\square$. Propositions  \ref{prop: admis_equiv}\ref{item: case comb equiv}  and \ref{prop: charge graph rational case} imply that the critically fixed Thurston map $f_{\square}$ introduced in Example \ref{ex: Topological square map} is combinatorially equivalent to the rational map $F_\square$. 
    We will use the map $f_{\square}$ in our further examples, but the same observations will hold for any Thurston map combinatorially equivalent to $f_{\square}$, in particular, for the rational map $F_{\square}$.
\end{example}

\subsection{Decomposition} \label{subsec: Decomposition}  Throughout this subsection we follow the notation and terminology introduced in Section~\ref{subsec: Decomposition theory}. Our goal is to prove the following result, which shows that completely invariant multicurves (and the induced decompositions) for critically fixed Thurston maps satisfy very restrictive conditions.

\begin{theorem}\label{thm: crit-fix-decomposition} Let $f$ be a critically fixed Thurston map and $\Gamma$ be a non-empty completely invariant multicurve. Suppose $\widehat{f}\colon \widehat{\mathscr{S}}_\Gamma\to\widehat{\mathscr{S}}_\Gamma$ is the corresponding map on the small spheres with respect to $\Gamma$. Then the following statements are true:
\begin{enumerate}[label=\normalfont{(\roman*)}]
    \item\label{item: decomp_thm_1} For each curve $\gamma \in \Gamma$ there is exactly one component $\gamma'$ of $f^{-1}(\gamma)$ that is isotopic to $\gamma$ rel.\ $\Crit(f)$. All other components $\delta'$ of $f^{-1}(\gamma)$ are null-homotopic with $\deg(f|\delta')=1$.
    \item\label{item: deccomp_thm_2} Every small sphere $\widehat S \in \widehat{\mathscr{S}}_\Gamma$ is fixed under $\widehat{f}$. Moreover, every point in $Q(\widehat S)$ is fixed under $\widehat f$.
\end{enumerate}
\end{theorem}

Before we proceed with the proof of this theorem, we first provide some auxiliary constructions and lemmas.

Let us assume that $f\colon\Sp\to\Sp$ is an arbitrary Thurston map and $\Gamma$ is a completely invariant multicurve. Consider an (abstract) graph $T_\Gamma$ with the vertex set $\mathscr{S}_\Gamma$ and the edge set $\Gamma$, where we connect two distinct components $S_1, S_2 \in \mathscr{S}_\Gamma$ by an edge $\gamma\in \Gamma$ if and only if $\gamma$ is a boundary curve in each of them. 
It easily follows that $T_\Gamma$ is connected. In fact, $T_\Gamma$ must be a tree. Indeed, the removal of any edge disconnects $T_\Gamma$, because each curve $\gamma\in \Gamma$ disconnects the sphere $\Sp$. We will denote by $T_{f^{-1}(\Gamma)}$ the corresponding tree for $f^{-1}(\Gamma)$.

If $S$ is a component in $\mathscr{S}_\Gamma$, we denote by $\widehat{S}$ the corresponding small sphere in $\widehat{\mathscr{S}}_\Gamma$ and by $Q(\widehat{S})$ the corresponding marked set, and similarly for the components in  $\mathscr{S}_{f^{-1}(\Gamma)}$. Recall that $f$ maps each component $S'\in \mathscr{S}_{f^{-1}(\Gamma)}$ onto a component $f(S')\in \mathscr{S}_\Gamma$, which induces a branched covering map $f_*\colon \big(\widehat{S'},Q(\widehat{S'})\big)  \to \big(\widehat{f(S')},Q(\widehat{f(S')})\big)$ between the associated small spheres (that respects the marked points). Note that $f$ sends adjacent vertices in $T_{f^{-1}(\Gamma)}$ to adjacent vertices in $T_\Gamma$.  
Recall also that, since $\Gamma$ is completely invariant, for every component $S \in \mathscr{S}_\Gamma$ there is a unique component $i(S)\in \mathscr{S}_{f^{-1}(\Gamma)}$ such that $i(S)\setminus \Post(f)$ is homotopic to $S\setminus\Post(f)$ in $\Sp \setminus \Post(f)$. This allows us to identify the corresponding small spheres $\widehat{S}\in \widehat{\mathscr{S}}_\Gamma$ and $\widehat{i(S)}\in \widehat{\mathscr{S}}_{f^{-1}(\Gamma)}$ via a homeomorphism $i^*\colon \big(\widehat{S}, Q(\widehat{S})\big) \to \big(\widehat{i(S)}, Q(\widehat{i(S)})\big)$. 
Then the map $\widehat{f}\colon \widehat{\mathscr{S}}_\Gamma \to \widehat{\mathscr{S}}_\Gamma$ is defined as the composition $f_*\circ i^*$.

We will now introduce two special subtrees of $T_{f^{-1}(\Gamma)}$. The first one, which we denote by $T^{ess}_{f^{-1}(\Gamma)}$, is the (unique) spanning subtree of the vertex set $\{i(S)\colon S\in \mathscr{S}_{\Gamma}\}$ in $T_{f^{-1}(\Gamma)}$. It is easy to see that the edges of $T^{ess}_{f^{-1}(\Gamma)}$ are exactly all the essential curves in $f^{-1}(\Gamma)$. In fact, the following claim is true (the proof is straightforward from the definitions; see also Figure~\ref{fig: trees}).

\begin{lemma}\label{lem: ess_tree_structure}
The tree $T^{ess}_{f^{-1}(\Gamma)}$ is obtained from $T_{\Gamma}$ by edge subdivision: if two components $S_1,S_2 \in \mathscr{S}_\Gamma$ are connected in $T_\Gamma$ by an edge $\gamma\in \Gamma$ then the components $i(S_1), i(S_2)$ are connected in $T^{ess}_{f^{-1}(\Gamma)}$ by a simple path consisting of all edges $\delta' \in f^{-1}(\Gamma)$ that are isotopic to $\gamma$ rel.\ $\Post(f)$.
\end{lemma}

To define the second subtree of $T_{f^{-1}(\Gamma)}$,  let us consider the set
\[\mathscr{S}^{\bullet}_{f^{-1}(\Gamma)}:=\{S'\in \mathscr{S}_{f^{-1}(\Gamma)}\colon S'\cap \Post(f) \neq \emptyset\},\]
which represents the vertices of $T_{f^{-1}(\Gamma)}$ that contain the postcritical points of $f$.  
We denote by $T^{\bullet}_{f^{-1}(\Gamma)}$ the (unique) spanning subtree of $\mathscr{S}^{\bullet}_{f^{-1}(\Gamma)}$ in $T_{f^{-1}(\Gamma)}$. The next lemma describes the structure of $T^{\bullet}_{f^{-1}(\Gamma)}$.  Again, the proof follows easily from the definitions and is left to the reader (see also Figure~\ref{fig: trees}).

\begin{lemma}\label{lem: post_tree_structure} 
The following statements are true:
\begin{enumerate}[label=\normalfont{(\roman*)}]

\item\label{item: post_tree_1} The edges of $T^{\bullet}_{f^{-1}(\Gamma)}$ are all the curves in $f^{-1}(\Gamma)$ that are not null-homotopic.  In particular,  $T^{ess}_{f^{-1}(\Gamma)}$ is a subtree of $T^{\bullet}_{f^{-1}(\Gamma)}$.

\item\label{item: post_tree_2} Let $S'\in \mathscr{S}^{\bullet}_{f^{-1}(\Gamma)}$.  Suppose $p$ is a postcritical point in $S'$ and $S\in \mathscr{S}_{\Gamma}$ is the component containing $p$.  Then either $i(S)=S'$ and $S'$ is a vertex of $T^{ess}_{f^{-1}(\Gamma)}$,  or $i(S)\neq S'$ and $S'\in V(T^{\bullet}_{f^{-1}(\Gamma)}) \setminus V(T^{ess}_{f^{-1}(\Gamma)})$.  In the latter case,  $S'$ is a leaf of $T^{\bullet}_{f^{-1}(\Gamma)}$ with $S'\cap \Post(f) = \{p\}$.  Furthermore,  $S'$ and $i(S)$ are connected in  $T^{\bullet}_{f^{-1}(\Gamma)}$ by a simple path consisting of all $p$-peripheral curves $\delta' \in f^{-1}(\Gamma)$.
\end{enumerate}
\end{lemma}

Here a curve $\delta'\in f^{-1}(\Gamma)$ is called \emph{$p$-peripheral} if for a component $U$ of $\Sp\setminus \delta'$ we have $U\cap \Post(f) = \{p\}$.  Note that any two $p$-peripheral curves are isotopic rel.\ $\Post(f)$.

\begin{figure}[t]
        \centering
        \begin{overpic}[width=\textwidth]{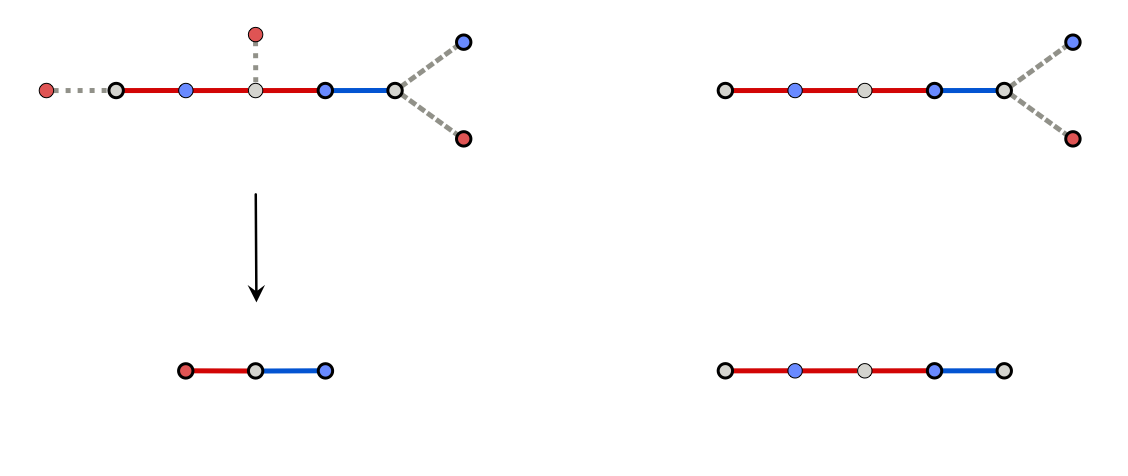}
        \put(14,11){$S_1$}   
        \put(21.5,11){$S_2$}   
        \put(28.5,11){$S_3$}  
        \put(19,7){$\alpha$}   
        \put(25,7){$\beta$}         
        \put(7,36){$i(S_1)$} 
        \put(25,36){$i(S_2)$}  
        \put(32,36){$i(S_3)$}
        \put(60,36){$i(S_1)$} 
        \put(79,36){$i(S_2)$}  
        \put(86,36){$i(S_3)$}
        \put(60,11){$i(S_1)$} 
        \put(80,11){$i(S_2)$}  
        \put(89,11){$i(S_3)$}
        \put(21.5,3){$T_\Gamma$}  
        \put(20,20){$f$}
        \put(20,28){$T_{f^{-1}(\Gamma)}$}
        \put(74,28){$T^{\bullet}_{f^{-1}(\Gamma)}$}
        \put(74,3){$T^{ess}_{f^{-1}(\Gamma)}$}
        \end{overpic}
        \caption{
        Left: The trees $T_\Gamma$ and $T_{f^{-1}(\Gamma)}$ for the Thurston map $f\colon \Sp\to \Sp$ and the completely invariant multicurve $\Gamma=\{\alpha,\beta\}$ from Figure~\ref{fig: decomposition}. Right: The corresponding subtrees $T^{ess}_{f^{-1}(\Gamma)}$ and $T^{\bullet}_{f^{-1}(\Gamma)}$ of $T_{f^{-1}(\Gamma)}$.}
        \label{fig: trees}
\end{figure}

\begin{example*}
    Let us consider the Thurston map $f\colon \Sp\to \Sp$ and the completely invariant multicurve $\Gamma=\{\alpha,\beta\}$ from Figure~\ref{fig: decomposition} with $\mathscr{S}_\Gamma=\{S_1,S_2,S_3\}$. We illustrate the respective trees $T_\Gamma$ and $T_{f^{-1}(\Gamma)}$ on the left in Figure~\ref{fig: trees}; the corresponding subtrees $T^{ess}_{f^{-1}(\Gamma)}$ and $T^{\bullet}_{f^{-1}(\Gamma)}$ of $T_{f^{-1}(\Gamma)}$ are shown on the right in the same figure. Here, each vertex of the tree $T_{f^{-1}(\Gamma)}$ is mapped by $f$ to the vertex of $T_\Gamma$ of the same color. The solid, densely dashed, and sparsely dashed edges represent the essential, peripheral, and null-homotopic curves in $f^{-1}(\Gamma)$, respectively. In particular, the three red solid edges in $T_{f^{-1}(\Gamma)}$ are isotopic to $\alpha$, the unique solid blue edge is isotopic to $\beta$, while the four dashed gray edges are non-essential. Finally, the vertices of $T_{f^{-1}(\Gamma)}$ with thicker boundary circles correspond to the components of $\mathscr{S}_{f^{-1}(\Gamma)}$ that contain the postcritical points of $f$.
\end{example*}

We now turn to the proof of Theorem \ref{thm: crit-fix-decomposition}.

\begin{proof}[Proof of Theorem \ref{thm: crit-fix-decomposition}]
    Suppose $f$ and $\Gamma$ are as in the statement. Then $\Post(f)=\Crit(f)$. We will say that a vertex of $T_\Gamma$ or of $T^{\bullet}_{f^{-1}(\Gamma)}$ is \emph{critical} if it contains a critical point. 
 
 Let 
\begin{equation*}%\label{eq: path_T_prime}
P':=\big(S'_0, \delta'_1, S'_1,\dots, \delta_n', S'_n \big)
\end{equation*}
be a simple path in $T^{\bullet}_{f^{-1}(\Gamma)}$. Here, each $S_j'$ is a vertex of the tree $T_{f^{-1}(\Gamma)}^\bullet$, and each $\delta'_j \in  E(T^{\bullet}_{f^{-1}(\Gamma)})$ is a non null-homotopic curve in $f^{-1}(\Gamma)$ by Lemma~\ref{lem: post_tree_structure}\ref{item: post_tree_1}. Then 
\begin{equation*}%\label{eq: image_path}
f(P'):=\big(f(S'_0), f(\delta'_1), f(S'_1),\dots, f(\delta'_n), f(S'_n)\big)
\end{equation*}
is a walk in $T_\Gamma$. 

Now suppose that $P'$ is a maximal (i.e., non-extendable) simple path in $T^{\bullet}_{f^{-1}(\Gamma)}$ such that $f(P')$ is a simple path in $T_\Gamma$. Clearly, $P'$ has positive length. We will refer to $P'$ as a \emph{maximal injective path} in $T^{\bullet}_{f^{-1}(\Gamma)}$. 

\begin{claim1}
The start and end vertices of $P'$ are critical.
\end{claim1}

It is sufficient to show that $P'$ starts at a critical vertex. If $S'_0$ is a leaf of $T^{\bullet}_{f^{-1}(\Gamma)}$, then $S'_0$ immediately has to be critical, because $T^{\bullet}_{f^{-1}(\Gamma)}$ is the spanning subtree of $\mathscr{S}^{\bullet}_{f^{-1}(\Gamma)}$ in $T_{f^{-1}(\Gamma)}$. 

So, now we assume that $S'_0$ is not a leaf of $T^{\bullet}_{f^{-1}(\Gamma)}$. We argue by contradiction and suppose that $S'_0 \cap \Crit(f) = \emptyset$. Set $\gamma:=f(\delta'_1)$, and let $E(S'_0)$ be the set of all edges that are incident to $S'_0$ in $T^{\bullet}_{f^{-1}(\Gamma)}$. By maximality of $P'$, each edge $\delta'\in E(S'_0)$ satisfies $f(\delta')=\gamma$. Furthermore, every critical point of $f$ is separated from $S'_0$ by one of these edges $\delta'$. Let us denote by $\widetilde S'$ the unique component in $\mathscr{S}_{f^{-1}(\{\gamma\})}$ that contains $S'_0$. Then $f(\widetilde S')\in \mathscr{S}_{\{\gamma\}}$ is an open Jordan region. Note that each edge $\delta'\in E(S'_0)$ is a boundary curve of $\widetilde S'$. It follows that $\chi( \widetilde S') \leq 0$ and $\widetilde S' \cap \Crit(f) = \emptyset$. Hence $f|\widetilde S'\colon \widetilde S'\to  f(\widetilde S')$ is a covering map, which is impossible. This contradiction implies that $S'_0\cap \Crit(f) \neq \emptyset$, which completes the proof of the claim. 

\medskip

Claim 1 implies that $S'_0$ contains a critical point $c_0$ and $S'_n$ contains a critical point $c_n$ of~$f$. Let $S_0$ and $S_n$ be the vertices of $T_\Gamma$ that contain $c_0$ and $c_n$, respectively. By Lemma \ref{lem: post_tree_structure}\ref{item: post_tree_2}, either $S'_0=i(S_0)$ and $S'_0$ is a vertex of $T^{ess}_{f^{-1}(\Gamma)}$, or $S'_0\neq i(S_0)$ and $S'_0$ is a leaf of $T^{\bullet}_{f^{-1}(\Gamma)}$ that is connected to $i(S_0)$ by a simple path consisting of all $c_0$-peripheral curves in $f^{-1}(\Gamma)$. In either case, we obtain that $i(S_0)$ is the first vertex of $T^{ess}_{f^{-1}(\Gamma)}$ on the path $P'$. Similarly,  $i(S_n)$ is the last vertex of $T^{ess}_{f^{-1}(\Gamma)}$ on the path $P'$.

Since $f$ is critically fixed, $f(S'_0) = S_0$ and $f(S'_n)=S_n$. Thus, $f(P')$ is a simple path that connects $S_0$ and $S_n$ in $T_\Gamma$. At the same time, it follows from Lemma \ref{lem: ess_tree_structure} that the subpath of $P'$ that connects $i(S_0)$ and $i(S_n)$ is not shorter than $f(P')$. In fact, this subpath has to pass through all the vertices
$i(f(S'_0))=i(S_0), i(f(S'_1)),\dots, i(f(S'_{n-1})),i(f(S'_n))=i(S_n)$
and in this particular order. These facts together imply the following two claims.

\begin{claim2} 
For each $j=0,\dots, n$, we have $i(f(S'_j)) = S'_j$. In particular, all vertices of $P'$ are in $T^{ess}_{f^{-1}(\Gamma)}$.
\end{claim2}

\begin{claim3} 
For each $j=1,\dots, n$, the curve $\delta'_j$ is essential and isotopic to $f(\delta'_j)$ rel.\ $\Crit(f)$. Furthermore, $\delta'_j$ is the only curve in $f^{-1}(\Gamma)$ that is isotopic to $f(\delta'_j)$ rel.\ $\Crit(f)$.
\end{claim3}

\medskip

Let us now fix an arbitrary curve $\gamma\in \Gamma$. Since $\Gamma$ is completely invariant, there must be a component $\gamma'$ of $f^{-1}(\Gamma)$ that is isotopic to $\gamma$ rel.\ $\Crit(f)$. The curve $\gamma'$ corresponds to an an edge of $T^{\bullet}_{f^{-1}(\Gamma)}$, and thus it is contained in some maximal injective path in $T^{\bullet}_{f^{-1}(\Gamma)}$. It follows from Claim 3 that $\gamma'$ satisfies $f(\gamma')=\gamma$, that is, $\gamma'$ is a component of $f^{-1}(\gamma)$. Furthermore, $\gamma'$ is the only curve in $f^{-1}(\Gamma)$ that is isotopic to $\gamma$ rel.\ $\Crit(f)$. 

Now suppose that $\delta'$ is an arbitrary component of $f^{-1}(\gamma)$. If $\delta'$ is not null-homotopic, it is contained in some maximal injective path in $T^{\bullet}_{f^{-1}(\Gamma)}$. Claim 3 implies that $\delta'$ is isotopic to $f(\delta')=\gamma$ rel.\ $\Crit(f)$, and thus it must be the curve $\gamma'$. If $\delta'$ is null-homotopic, then there is a component $U$ of $\Sp\setminus \delta'$ such that $U\cap \Crit(f) = \emptyset$. 
It follows from the Riemann-Hurwitz formula that $f$ sends $U$ homeomorphically onto a component of $\Sp\setminus \gamma$, and hence $\deg(f|\delta')=1$. This proves part \ref{item: decomp_thm_1} of the theorem.  

To prove part \ref{item: deccomp_thm_2}, consider an arbitrary small sphere $\widehat{S}\in \widehat{\mathscr{S}}_\Gamma$. Let $S$ be the corresponding component in $\mathscr{S}_\Gamma$. Then $i(S)$ is a vertex of $T^{\bullet}_{f^{-1}(\Gamma)}$, and thus it is contained in some maximal injective path in $T^{\bullet}_{f^{-1}(\Gamma)}$. By Claim 2 we have $i(f(i(S))) = i(S)$. It follows that $f(i(S))=S$, which means that $\widehat{f}(\widehat{S})=\widehat{S}$, as required. Finally, since the marked set $Q(\widehat S)$ corresponds to the points in $S\cap \Crit(f)$ and the components of $\partial S$, every marked point is fixed under $\widehat f$. This completes the proof of the theorem.
\end{proof}

We record the following immediate corollary of Theorem \ref{thm: crit-fix-decomposition}.

\begin{corollary}
    Let $f$ be a critically fixed Thurston map and $\Gamma$ be a non-empty completely invariant multicurve. Then $T^{\bullet}_{f^{-1}(\Gamma)}= T^{ess}_{f^{-1}(\Gamma)}$ and $f$ provides an isomorphism between $T^{\bullet}_{f^{-1}(\Gamma)}$ and $T_{\Gamma}$.
\end{corollary}

The next corollary states that it is sufficient to check the absence of Levy fixed curves to conclude that a given critically fixed Thurston map is realized; compare Proposition \ref{main: obstr-levy-fixed}. 

\begin{corollary}\label{cor: crit_fix_obstructions}
Let $f$ be a critically fixed Thurston map. Then every Thurston obstruction of $f$ contains a Levy fixed curve. In particular, $f$ is realized by a rational map if and only if $f$ does not have a Levy fixed curve.
\end{corollary}

\begin{proof}
    The proof easily follows from Theorems \ref{thm: Thurston_theorem} and \ref{thm: crit-fix-decomposition}, since every Thurston obstruction contains a simple one. (Recall that if $|\Crit(f)|>2$,  then $f$ has a hyperbolic orbifold; otherwise, $|\Crit(f)|=2$ and $f$ is combinatorially equivalent to the power map $z\mapsto z^{\deg(f)}$.)
\end{proof}

\subsection{Canonical Thurston obstruction}\label{subsec: crit-fix-can-obstr} The following corollary of Theorem \ref{thm: crit-fix-decomposition} describes properties of canonical Thurston obstructions for critically fixed Thurston maps. In particular, it shows that every critically fixed Thurston map can be canonically decomposed into homeomorphisms and critically fixed rational maps.

\begin{corollary}\label{cor: can_crit_fix_decomposition}
    Let $f$ be a critically fixed Thurston map and $\Gamma:=\Gamma_{\operatorname{Th}}$ be the canonical Thurston obstruction of $f$. Then the following statements are true:
\begin{enumerate}[label=\normalfont{(\roman*)}]
    \item\label{item: can_levy_1} For every curve $\gamma\in \Gamma$ there is exactly one component $\gamma'$ of $f^{-1}(\gamma)$ that is isotopic to $\gamma$ and satisfies $\deg(f|\gamma')=1$. All other components $\delta'$ of $f^{-1}(\gamma)$ are null-homotopic and also satisfy $\deg(f|\delta')=1$. In particular, each curve $\gamma\in \Gamma$ is a Levy fixed curve.
    \item\label{item: can_levy_2} Every small sphere $\widehat{S}\in\widehat{\mathscr{S}}_\Gamma$ is fixed under $\widehat f$. Moreover, every point in $Q(\widehat S)$ is fixed under $\widehat f$.
\item\label{item: can_levy_3} If $S\in \mathscr{S}_\Gamma$ satisfies $S\cap \Crit(f)=\emptyset$, then 
$\widehat f\colon \big(\widehat{S}, Q(\widehat{S})\big)\to \big(\widehat{S}, Q(\widehat{S})\big)$ is a homeomorphism.
     \item\label{item: can_levy_4} If $S\in \mathscr{S}_\Gamma$ satisfies $S\cap \Crit(f)\neq\emptyset$, then 
     $\widehat f\colon \big(\widehat{S}, Q(\widehat{S})\big)\to \big(\widehat{S}, Q(\widehat{S})\big)$  is realized by a marked critically fixed rational map of degree $d(\widehat{S})=1+\frac{1}{2}\sum_{c\in S\cap \Crit(f)} (\deg(f,c)-1)$.
     \item\label{item: can_levy_5} If distinct components $S_1, S_2\in \mathscr{S}_\Gamma$ satisfy $\partial S_1\cap \partial S_2 \neq \emptyset$, then either $S_1\cap \Crit(f) \neq \emptyset$ or $S_2\cap \Crit(f) \neq \emptyset$.
\end{enumerate}    
\end{corollary}

\begin{proof} Suppose $f$ and $\Gamma$ are as in the statement.

   \ref{item: can_levy_1}-\ref{item: can_levy_4} The proof is immediate from Theorems \ref{thm: can_obstr} and \ref{thm: crit-fix-decomposition}, since the canonical Thurston obstruction $\Gamma$ is simple. The formula for $d(\widehat{S})$ in \ref{item: can_levy_4} follows from the Riemann-Hurwitz formula.
   
   \ref{item: can_levy_5} Suppose $\partial S_1\cap \partial S_2 \neq \emptyset$ for some distinct $S_1, S_2\in \mathscr{S}_\Gamma$. This means that $\gamma=\partial S_1\cap \partial S_2$ for some curve $\gamma\in \Gamma$. 
Set $\Gamma':=\Gamma\setminus\{\gamma\}$ and $S'_\gamma:=S_1\cup\gamma\cup S_2$. Then 
\[\mathscr{S}_{\Gamma'}= \big(\mathscr{S}_\Gamma\setminus \{S_1, S_2\}\big) \sqcup \{ S'_\gamma\}.\]

By \ref{item: can_levy_1}, the multicurve $\Gamma':=\Gamma\setminus\{\gamma\}$ is a completely invariant Thurston obstruction for~$f$. It follows that every $\widehat{S'}\in \widehat{\mathscr{S}}_{\Gamma'}$ is fixed under the induced map $\widehat f'\colon \widehat{\mathscr{S}}_{\Gamma'} \to \widehat{\mathscr{S}}_{\Gamma'}$ on the small spheres with respect to $\Gamma'$. Furthermore, when $\widehat{S'}\neq \widehat {S'_\gamma}$, the corresponding small sphere map 
is a homeomorphism or a marked Thurston map that is realized by a marked critically fixed rational map. Now if $S_1\cap \Crit(f) = S_2\cap \Crit(f) = \emptyset$, then the small sphere map $\widehat f'\colon \big(\widehat{S'_\gamma}, Q(\widehat{S'_\gamma})\big) \to \big(\widehat{S'_\gamma}, Q(\widehat{S'_\gamma})\big)$ is a homeomorphism.  Hence the multicurve $\Gamma'$ satisfies the conditions from Theorem \ref{thm: can_obstr}, which contradicts the minimality of $\Gamma$.
\end{proof}

    We will now describe how to determine the canonical Thurston obstruction for a critically fixed Thurston map obtained by blowing up an admissible pair. First, we introduce some terminology. 
    
    Let $G$ be a planar embedded  graph in $\Sp$. Suppose $H$ is a connected component of $G$ and $U$ is one of the faces of $H$. Since $U$ is simply connected, we may fix a homeomorphism $\eta\colon \D \to U$. Let $0<\varepsilon < 1$. We say that a Jordan curve $\gamma \subset U$ is an \emph{$\varepsilon$-boundary} of $G$ (with respect to $U$) if $\eta^{-1}(\gamma) \subset\{z\colon 1-\varepsilon < |z| < 1\}$ and $\eta^{-1}(\gamma)$ separates $0$ from $\partial \D$. Note that the isotopy class of $\gamma$ rel.\ $V(G)$ is fixed for all sufficiently small $\varepsilon$ and is independent of the choice of~$\eta$. In the following, whenever we talk about $\varepsilon$-boundaries, we always assume that $\varepsilon$ is sufficiently small.

\begin{theorem}\label{thm: can-obstr-crit-fix}
    Let $f\colon \Sp\to\Sp$ be a critically fixed Thurston map obtained by blowing an admissible pair $(G,\varphi)$. Set $\Gamma$ to be the multicurve obtained by taking all the essential $\varepsilon$-boundaries of $G$ and identifying the isotopic ones. Then $\Gamma$ is the canonical Thurston obstruction for $f$.
\end{theorem}    

Note that the multicurve $\Gamma$ is empty if and only the graph $G$ is connected. It follows that the map $f$ is realized by a rational map if and only if $G$ is connected; compare Proposition~\ref{prop: rational if charge connected}.

\begin{proof}
    We only give an outline of the argument and leave some straightforward details to the reader.

    Let $f$ be a critically fixed Thurston map obtained by blowing an admissible pair $(G,\varphi)$. In particular, we fix a choice of $W_e$, $D_e$, $f_e$, and $h$ as in Section \ref{subsec: The blow-up operation}. Without loss of generality, we may assume that $\varphi(e) = e$ for every edge $e\in E(G)$; see Propositions \ref{prop: Graph isotopic criterion} and \ref{prop: Isotopy and blow-up}. By (the proof of) Proposition \ref{prop: blow-up-triples-properties}, the triples $(\partial D_e^+, \partial D_e^-, \inter(D_e))$, $e\in E(G)$, satisfy all the conditions \ref{item: blow-up-tripple-i}-\ref{item: blow-up-tripple-iii} and \ref{item: blow-up-mapping-i}-\ref{item: blow-up-mapping-iv}. Then  $G^\pm=\bigcup_{e\in E(G)} (\partial D_e^+\cup \partial D_e^-)$ is the corresponding blow-up of $G$ under the map $f$.

    \begin{claim}
    Let $\gamma$ be an essential  $\varepsilon$-boundary of $G$. Then there is exactly one component $\gamma'$ of $f^{-1}(\gamma)$ that is isotopic to $\gamma$ rel.\ $\Crit(f)$ and satisfies $\deg(f|\gamma')=1$. All other components~$\delta'$ of $f^{-1}(\gamma)$ are null-homotopic and satisfy $\deg(f|\delta')=1$. In particular, $\gamma$ is a Levy fixed curve of $f$.
    \end{claim}

    Suppose $\gamma$ is an essential $\varepsilon$-boundary of $G$ with respect to $U$. Then $\gamma \subset W \subset U$, where $W$ is a multiply connected face of $G$. Let $\widetilde W$ be the face of $G^\pm$ such that $f|\widetilde{W}\colon \widetilde{W} \to W$ is a homeomorphism (see Proposition~\ref{prop: blow-up-triples-properties}). Then $\widetilde W$ contains a unique component $\gamma'$ of $f^{-1}(\gamma)$. Clearly, $\deg(f|\gamma')=1$. By \ref{item: blow-up-mapping-i} and \ref{item: blow-up-mapping-iii}, any other component $\delta'$ of $f^{-1}(\gamma)$ belongs to $\inter(D_e)$ for some $e\in E(G)$. 
    Therefore, $\delta'$ is null-homotopic and satisfies $\deg(f|\delta')=1$.

    To prove the claim, it remains to show that $\gamma'$ is isotopic to $\gamma$ rel.\ $\Crit(f)$. Since $\varphi(G)=G$, the Jordan curve $\varphi^{-1}(\gamma)$ is also an $\varepsilon$-boundary of $G$ with respect to $U$. In particular, $\varphi^{-1}(\gamma)\sim \gamma$ rel.\ $\Crit(f)$. By \eqref{eq: blow-up_def}, $h_t|\gamma'$, $t\in \I$, provides a non-ambient isotopy rel.\ $\Crit(f)$ between $h_0(\gamma')=\gamma'$ and $h_1(\gamma')=\varphi^{-1}\circ (\varphi\circ h_1)(\gamma')=\varphi^{-1}(f(\gamma'))=\varphi^{-1}(\gamma)$. It follows that $\gamma'\sim \varphi^{-1}(\gamma)\sim \gamma$ rel.~$\Crit(f)$.

 \smallskip   

    Let $\Gamma$ be the multicurve as in the statement. If $\Gamma$ is empty, then $G$ must be connected, and thus $\varphi$ is isotopic to $\id_{\Sp}$ rel.\ $V(G)$ by Corollary~\ref{cor: Homeo rigidity}. Hence $f$ is realized by a rational map by Propositions~\ref{prop: Isotopy and blow-up} and \ref{prop: rational if charge connected}. The statement follows in this case.
    
    So, now we assume  that $\Gamma$ is non-empty. The claim above implies that the multicurve $\Gamma$ is completely invariant. Suppose $\widehat S\in \widehat{\mathscr{S}}_\Gamma$ is a small sphere with respect to $\Gamma$, and $S$ is the respective component in $\mathscr{S}_\Gamma$. Recall that $\widehat S$
    is marked by a finite set $Q(\widehat S)$ corresponding to the points in $S\cap \Crit(f)$ and the components of $\partial S$. By  Theorem~\ref{thm: crit-fix-decomposition}\ref{item: deccomp_thm_2}, the small sphere $\widehat S$ is fixed under $\widehat f$. 

    For short, let us denote the small sphere map $\widehat f\colon \big(\widehat{S}, Q(\widehat{S})\big)\to \big(\widehat{S}, Q(\widehat{S})\big)$ by $\widehat f|\widehat S$.  If $S\cap \Crit(f) = \emptyset$, then $\widehat f|\widehat S$ is a homeomorphism. Otherwise, $S$ contains a unique connected component $G_S$ of $G$, which we view as a subset of $\widehat{S} \setminus Q(\widehat{S})$. Note that 
    \[|E(G_S)|=\frac{1}{2}\sum_{c\in S\cap \Crit(f)} \deg_{G}(c) = \frac{1}{2}\sum_{c\in S\cap \Crit(f)} (\deg(f,c)-1) = \deg(\widehat f|\widehat S)-1,\]
    where the last equality follows from the Riemann-Hurwitz formula.
    
    By the definition of $\widehat f$, the triples $(\partial D_e^+, \partial D_e^-, \inter(D_e))$, $e\in E(G_S)$, induce triples that satisfy \ref{item: blow-up-tripple-i}-\ref{item: blow-up-tripple-iii} for the small sphere map $\widehat f|\widehat S$ and the graph $G_S\subset \widehat S$. Moreover, each face $U$ of $G_S$ contains at most one marked point in $Q(\widehat S)$. It now follows from Proposition \ref{prop: Blow-up triples} and Remark \ref{rem: isotopy_rel_fixed_points} that the marked small sphere map $(\widehat f|\widehat S, Q(\widehat S))$ is realized (by a marked critically fixed rational map). 
    
    We now check that the multicurve $\Gamma$ satisfies the criterion in Theorem~\ref{thm: can_obstr} (and thus $\Gamma$ is the canonical Thurston obstruction of~$f$). Indeed, the induced small sphere maps satisfy the necessary requirements by the discussion above. To check minimality, assume that $\Gamma'\subsetneq \Gamma$ and $\gamma \in \Gamma\setminus \Gamma'$. Let $S'_\gamma\in \mathscr S_{\Gamma'}$ be the component containing $\gamma$. Note that $\gamma$ is an essential curve in the respective marked small sphere $\widehat{S'_\gamma}$. Moreover, $S'_\gamma \cap \Crit(f) \neq \emptyset$. The claim now implies that $\gamma$ is a Levy fixed curve for the marked small sphere map associated with $\widehat{ S'_\gamma}$. Therefore, $\Gamma'$ does not meet the requirements from Theorem~\ref{thm: can_obstr}, and $\Gamma$ must be the canonical Thurston obstruction for $f$.
\end{proof}

\subsection{The charge graph and classification} \label{subsec: The charge graph and its applications}

The goal of this subsection is to extend the notion of the charge graph of a critically fixed rational map to the more general setting of critically fixed Thurston maps. Moreover, we will provide a classification of critically fixed Thurston maps in terms of admissible pairs.

Let $f\colon\Sp \to \Sp$ be a critically fixed Thurston map and $\Gamma:=\Gamma_{\operatorname{Th}}$ be the canonical Thurston obstruction for $f$. We may decompose  $\widehat{\mathscr{S}}_{\Gamma}$ as the disjoint union $\widehat{\mathscr{S}}_{\Gamma,\operatorname{Rat}}\sqcup \widehat{\mathscr{S}}_{\Gamma,\operatorname{Homeo}}$, where $\widehat{\mathscr{S}}_{\Gamma,\operatorname{Rat}}$ consists of all small spheres containing the critical points of $f$, and $\widehat{\mathscr{S}}_{\Gamma,\operatorname{Homeo}}$ consists of all small spheres without these critical points.

Suppose $\widehat S \in \widehat{\mathscr{S}}_{\Gamma,\operatorname{Rat}}$ is a small sphere, $Q(\widehat S)$ is the corresponding marked set, and $S$ is the respective component in $\mathscr{S}_\Gamma$. By Corollary \ref{cor: can_crit_fix_decomposition}, $\widehat S$ is fixed under~$\widehat f\colon \widehat{\mathscr{S}}_{\Gamma}\to\widehat{\mathscr{S}}_{\Gamma}$ and the small sphere map $\widehat f|\widehat S \colon \big(\widehat S, Q(\widehat S)\big) \to \big(\widehat S, Q(\widehat S)\big) $ is a (marked) critically fixed Thurston map. Moreover, $(\widehat f|\widehat S, Q(\widehat S))$ is combinatorially equivalent to a critically fixed rational map $(F_{\widehat S}, Q(F_{\widehat S}))$ with $\Crit(F_{\widehat S})\subset Q(F_{\widehat S})\subset \Fix(F_{\widehat S})$. By Proposition \ref{prop: comb_equiv_blowups} and Remarks \ref{rem: comb_equiv_and_adm_pairs}\ref{item: comb_equiv_and_adm_pairs_0} and~\ref{rem: isotopy_rel_fixed_points}, the map $\widehat f|\widehat S$ is obtained by blowing up an admissible pair $(G'_{\widehat S}, \varphi_{\widehat S})$, where $G'_{\widehat S}$ is a planar embedded graph in $\widehat S$ with $V(G'_{\widehat S})=Q(\widehat S)$ and $\varphi_{\widehat S}\in \Homeo^+_0(\widehat{S},Q(\widehat S))$. Furthermore, the graph $G'_{\widehat S}$ has the following properties:
\begin{itemize}
    \item $G'_{\widehat S}$ has a unique non-trivial connected component $G_{\widehat S}$;
    \item $G_{\widehat S}$ is  isomorphic to $\Charge(F_{\widehat S})$;
    \item $V(G_{\widehat S})=\Crit(\widehat f| \widehat S) = \Crit(f)\cap S$;
    \item $G'_{\widehat S}= G_{\widehat S} \sqcup \big(Q(\widehat S) \setminus V(G_{\widehat S})\big)$.
\end{itemize}

It follows that we may naturally view every graph $G_{\widehat S}$, $\widehat S\in  \widehat{\mathscr{S}}_{\Gamma,\operatorname{Rat}}$, as a subset of the respective component $S \subset \Sp$. Then the union
\begin{equation}\label{eq: charge_graph}
\Charge(f):=\bigsqcup_{\widehat S\in  \widehat{\mathscr{S}}_{\Gamma,\operatorname{Rat}}} G_{\widehat{S}}
\end{equation}
 is a planar embedded graph in $\Sp$ with the vertex set $\Crit(f)$. These observations lead us to the following definition. 
     
\begin{definition}\label{def:charge-graph}
Let $f$ be a critically fixed Thurston map. We define the \emph{charge graph} of~$f$ to be the planar embedded graph $\Charge(f)$ with the vertex set $\Crit(f)$ constructed as above.
\end{definition}

Note that the charge graph is defined only up to isotopy rel.\ $\Crit(f)$. In fact, its isotopy class is uniquely characterized by the following result (see Proposition \ref{prop: admis_equiv}\ref{item: case isotopy}).

\begin{proposition}\label{prop: charge-graph}
    Let $f\colon\Sp\to\Sp$ be a critically fixed Thurston map. Then $f$ is isotopic to a critically fixed Thurston map obtained by blowing up an admissible pair $(\Charge(f), \varphi_f)$ in $\Sp$, where $\varphi_f\in \Homeo^+(\Sp,\Crit(f))$. 
\end{proposition}

\begin{proof} Suppose $f\colon \Sp \to \Sp$ is a critically fixed Thurston map, $G:=\Charge(f)$ is the charge graph of~$f$, and $\Gamma:=\Gamma_{\operatorname{Th}}$ is the canonical Thurston obstruction of $f$.

Following the notation above, let  $\widehat S\in \widehat{\mathscr{S}}_{\Gamma,\operatorname{Rat}}$ be a small sphere and $S$ be the respective component in $\mathscr{S}_{\Gamma}$. By construction,  $S\subset \Sp$ contains exactly one (non-trivial) connected component $G_{\widehat S}$ of the charge graph $G$. 
Combining Proposition~\ref{prop: Blow-up properties}, Corollary \ref{cor: can_crit_fix_decomposition}\ref{item: can_levy_4}, and the Riemann-Hurwitz formula, we obtain
\begin{align*}
|E(G)|&=\sum_{\widehat S\in \widehat{\mathscr{S}}_{\Gamma,\operatorname{Rat}}} |E(G_{\widehat S})| = \sum_{\widehat S\in \widehat{\mathscr{S}}_{\Gamma,\operatorname{Rat}}} \big(\deg(\widehat f|\widehat S) -1\big) =\\
&=\sum_{\widehat S\in \widehat{\mathscr{S}}_{\Gamma,\operatorname{Rat}}} \frac{1}{2}\sum_{c\in S\cap \Crit(f)} \big(\deg(f,c)-1\big)=\frac{1}{2}\sum_{c\in \Crit(f)} \big(\deg(f,c)-1\big) = \deg(f)-1.
\end{align*}

We claim that each edge $\alpha$ of $G$ admits a triple $(\alpha^+,\alpha^-, U_\alpha$) that satisfies conditions \ref{item: blow-up-tripple-i}-\ref{item: blow-up-tripple-iii}. Indeed, if $\alpha \in E(G_{\widehat{S}})$, then the small sphere map $\widehat f|\widehat S \colon \big(\widehat S, Q(\widehat S)\big) \to \big(\widehat S, Q(\widehat S)\big) $  induces such a triple in $\Sp$ (apply the argument in the proof of Proposition~\ref{prop: blow-up-triples-properties} to $f|\widehat S$ and note that $\overline{U_\alpha} \subset S\subset \Sp$ by the construction). 
The proposition now follows from Proposition~\ref{prop: Blow-up triples}.
\end{proof}

We are finally ready to prove Main Theorem \ref{thm_intro: classification} from the introduction. Let $\mathrm{CritFixMaps}$ be the set of  combinatorial equivalence classes of critically fixed Thurston maps $f\colon \Sp \to \Sp$, and $\mathrm{AdmPairs}$ be the set of equivalence classes of admissible pairs $(G, \varphi)$, where $G$ is a planar embedded graph in $\Sp$ and $\varphi$ is a homeomorphism in $\Homeo^+(\Sp,V(G))$. Proposition \ref{prop: admis_equiv}\ref{item: case comb equiv} implies that the blow-up operation induces a well-defined injective map $$\mathrm{BlowUp}\colon \mathrm{AdmPairs} \to \mathrm{CritFixMaps}$$ given by 
    $$
        [(G,\varphi)]\mapsto \left[f_{(G,\varphi)}\right],
    $$
where $f_{(G,\varphi)}\colon \Sp \to \Sp$ denotes a critically fixed Thurston map obtained by blowing up an admissible pair $(G,\varphi)$ in $\Sp$, and $[\cdot]$ denotes the equivalence class of an element as usual. Moreover, Proposition \ref{prop: charge-graph} implies that this map is surjective. Hence, we get the following result, which establishes the first part of Main Theorem~\ref{thm_intro: classification}.

\begin{theorem}\label{thm: Classification}
    The map $\mathrm{BlowUp}\colon \mathrm{AdmPairs} \to \mathrm{CritFixMaps}$ induced by the blow-up operation is a bijection.    
\end{theorem}

Similarly to the above, Propositions \ref{prop: Blow-up properties}, \ref{prop: admis_equiv}\ref{item: case isotopy}, and \ref{prop: charge-graph} imply the second part of Main Theorem~\ref{thm_intro: classification}.

\begin{theorem}
\label{thm: class-isotopy} Fix a marked sphere $(\Sp, Z)$ with $|Z|\geq 2$ and an integer $d\geq 2$. Then the blow-up operation induces a canonical bijection between the isotopy classes of critically fixed Thurston maps $f\colon \Sp\to\Sp$ with $\Crit(f)=Z$ and $\deg(f)=d$ and the isotopy classes of admissible pairs $(G,\varphi)$ with $V(G)=Z$ and $|E(G)|=d-1$.
\end{theorem}

\section{The Lifting Algorithm}\label{sec: Lifting algorithm}

In this section, we develop an algorithm that for a given critically fixed Thurston map $f$ finds its charge graph $G_f:=\Charge(f)$. This also allows us to reconstruct an admissible pair $(G_f,\varphi_f)$ that is associated with $f$ by Theorem~\ref{thm: class-isotopy}. Namely, we provide a combinatorial description of the homeomorphism $\varphi_f$ using the charge graph $G_f$ and the original map~$f$. 
Knowing the graph $G_f$, we can decide whether $f$ is obstructed or not, depending on the connectivity of $G_f$. Furthermore, if $f$ is realized (i.e., if $G_f$ is connected), 
the charge graph~$G_f$ determines both the combinatorial equivalence class and the isotopy class of $f$ (see Remark~\ref{rem: adm-pair-for-connected}). Otherwise, we find the canonical obstruction of $f$ by taking the essential $\varepsilon$-boundaries of $G_f$ (Theorem~\ref{thm: can-obstr-crit-fix}). 

\subsection{The pullback relation on trees}\label{subsec: The pullback operation}

In the following, we suppose that $f\colon \Sp \to \Sp$ is a Thurston map. A planar embedded tree $T$ in $\Sp$ is called \emph{admissible} (for $f$) if it satisfies the following two properties:
\begin{itemize}%[label = \normalfont{(\roman*)}]
    \item $\Post(f) \subset V(T)$; 
    \item $\deg_{T}(v)\geq 3$ for all $v\in V(T)\setminus \Post(f)$.
\end{itemize}
In particular, every leaf of an admissible tree is a postcritical point of $f$.

\begin{lemma}\label{lem: Combinatorial complexity}
    Let $f\colon \Sp \to \Sp$ be a Thurston map and $T$ be an admissible planar embedded tree for $f$. 
    Then $|V(T)| \leq 2 |\Post(f)| - 2$.
\end{lemma}

\begin{proof}
    By definition, each non-postcritical vertex of $T$ has degree greater than $2$. Hence the following inequality holds:

    $$
        \sum\limits_{v \in V(T)} \deg_{T} (v) \geq |\Post(f)| + 3\big(|V(T)| - |\Post(f)|\big).
    $$
    At the same time, since $T$ is a tree, we have
    $$
        \sum\limits_{v \in V(T)} \deg_{T} (v) = 2 |E(T)| = 2|V(T)| - 2.
    $$
    It follows that $$2|V(T)| - 2 \geq |\Post(f)| + 3\big(|V(T)| - |\Post(f)|\big),$$ and thus $|V(T)| \leq 2|\Post(f)|-2$, as desired.
\end{proof}

We denote by $\AdmTrees(f)$ the set of all admissible planar embedded trees for $f$. The map $f$ induces a natural relation on $\AdmTrees(f)$ as follows; compare \cite{TimorinTrees, LifitingTrees, Pilgrim_Pullback}. 

Let $T\in \AdmTrees(f)$ be an admissible tree. By Lemma~\ref{lem:preimage-connected}, the complete preimage $f^{-1}(T)$ is a connected graph with $\Post(f) \subset V(f^{-1}(T))$. Hence, we can choose a spanning subtree of the postcritical set $\Post(f)$ in the graph $f^{-1}(T)$, that is, a minimal under inclusion subtree $\widetilde T$ of $f^{-1}(T)$ such that $\Post(f) \subset V(\widetilde T)$. Note that each non-postcritical vertex of $\widetilde T$ must then have degree at least $2$. We can now further ``simplify'' the tree $\widetilde T$ by forgetting all non-postcritical vertices of degree $2$ (if there are any). More formally, we consider the tree $T'$ embedded in $\Sp$ with the same realization as $\widetilde T$ but with the vertex set $V(T')$ given by
\[V(T')=\Post(f) \cup \{v\in V(\widetilde T)\colon \ \deg_{\widetilde T}(v) \geq 3\}.\]
By construction, $T'$ is an admissible tree for $f$. 

\begin{definition}\label{def:pullback-operation}
    The planar embedded tree $T'$ constructed as above is called  a \emph{pullback} of the tree~$T$ under the map $f$. We use the notation $\xleftarrow{f}$ for the induced \emph{pullback relation} on $\AdmTrees(f)$, i.e., we write $T\xleftarrow{f}T'$ if $T'$ is a pullback of the tree $T$. We also denote by $\Pi_f(T)$ the set of all such pullbacks.
\end{definition}

\begin{remark}
    The isotopy lifting property for Thurston maps (see Proposition \ref{prop: isotopy-lifting}) implies that the pullback relation $\xleftarrow{f}$ descends to the quotient of $\AdmTrees(f)$ by isotopies rel.\ $\Post(f)$.
\end{remark}
    
    We emphasize that a pullback of the tree $T\in \AdmTrees(f)$ is not uniquely determined, since the choice of the spanning subtree $\widetilde T$ may not be unique.   
    Nevertheless, we may iterate the pullback relation $\xleftarrow{f}$ starting with an arbitrary admissible tree $T_0$ and obtain an infinite sequence $\{T_n\}_{n \geq 0} \subset \AdmTrees(f)$ that satisfies 
    \[T_0\,\xleftarrow{f}\, T_1\, \xleftarrow{f}\, T_2\xleftarrow{f}\dotsb.\]
    In other words, we get a sequence $\{T_n\}_{n \geq 0}$ of admissible planar embedded trees, where $T_{n+1} \in \Pi_f(T_{n})$ for all $n\ge 0$. 
Lemma \ref{lem: Combinatorial complexity} implies that the ``combinatorial complexity'' of the trees $T_n$ is uniformly bounded. In the next subsection, we will show that the ``topological complexity'' of these trees eventually stabilizes in the case when $f$ is critically fixed.

    \begin{figure}[t]
        \centering          
        \begin{overpic}[width=12cm]{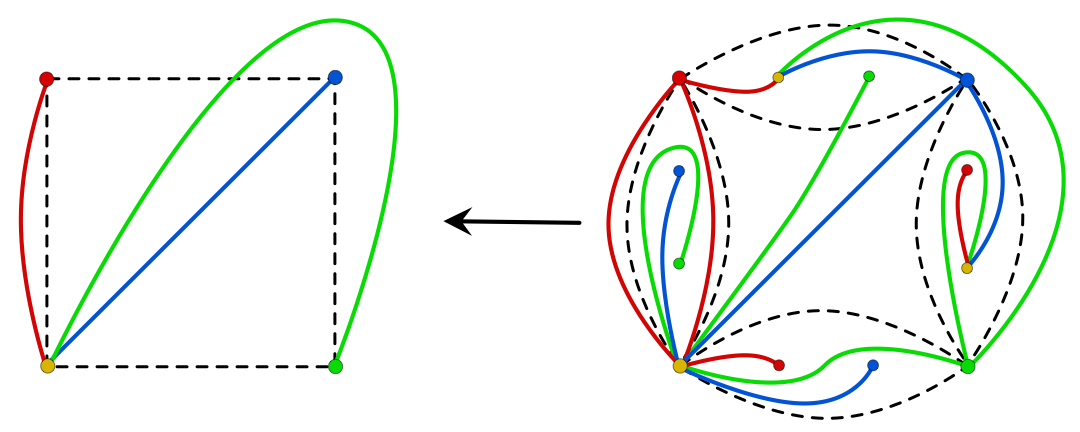}
        \put(47,22){$f_\square$}
        \put(-4,22){$T$}        
        \put(102,22){$f_\square^{-1}(T)$}
        \end{overpic}
        \caption{Taking the complete preimage of an admissible tree $T$ with $V(T)=\Crit(f_\square)$
        under the critically fixed Thurston map $f_\square$.}

        \label{fig: Tree and complete preimage}
    \end{figure}
    \begin{figure}[t]
        \centering           \includegraphics[width=15cm]{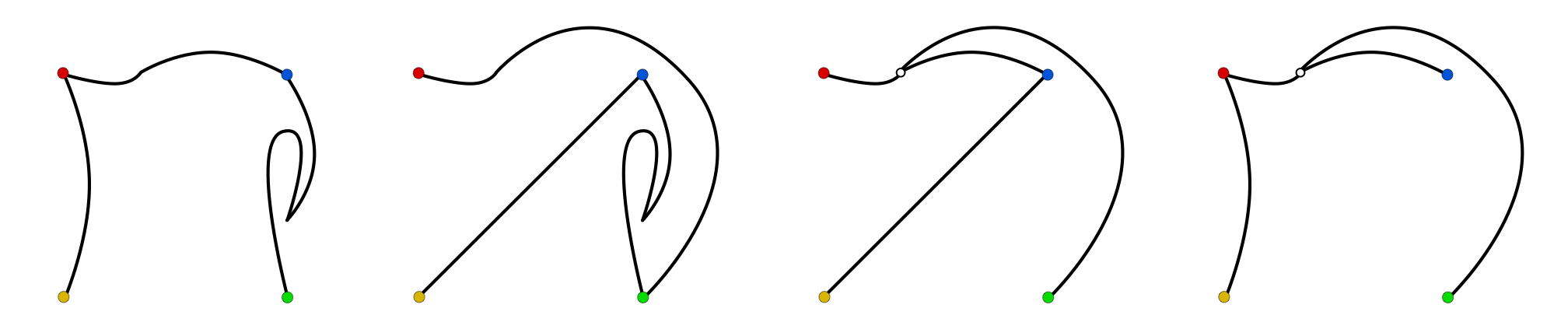}
    
        \caption{
        Examples of pullbacks of the admissible tree $T$ from Figure \ref{fig: Tree and complete preimage} under the map~$f_\square$.
        }

        \label{fig: Pullbacks}
    \end{figure}

\begin{example}\label{ex: Pullback example}

    Let us consider the critically fixed Thurston map $f_{\square}$ from Example \ref{ex: Topological square map} and an admissible planar embedded tree $T$ with $V(T)=\Post(f_\square)=\Crit(f_\square)$ shown on the left in Figure~\ref{fig: Tree and complete preimage}. The tree $T$ has three edges colored red, green, and blue. The complete preimage $f_\square^{-1}(T)$ is illustrated on the right in Figure \ref{fig: Tree and complete preimage}. The map $f_\square$ sends each edge and vertex of $f_\square^{-1}(T)$ to the edge and vertex of $T$ of the same color. The graphs in dashed lines indicate the charge graph of $f_\square$ (on the left) and its blow-up (on the right).

    Figure \ref{fig: Pullbacks} illustrates some of the possible pullbacks of the tree $T$ under the map $f_{\square}$. Note that the first two examples are trees with vertices only in the critical points of $f_{\square}$, while the last two examples are trees with an extra non-critical vertex of degree $3$ indicated in white.

\end{example}

\subsection{Topological contraction of the pullback relation}\label{subsec: Topological contraction of the pullback operation}

    To control the topological complexity of the (iterated) pullbacks of an admissible tree under a critically fixed Thurston map~$f$, we will use intersection numbers rel.\ $\Crit(f)$. 

\begin{definition}
     Let $f\colon \Sp \to \Sp$ be a critically fixed Thurston map and $G$ be a planar embedded graph in $\Sp$ with $\Crit(f) \subset V(G)$. We define the \emph{norm} $\|G\|_f$ of $G$ with respect to the map $f$ as
     $$
        \|G\|_f := \max_{\alpha \in E(\Charge(f))} i_{C(f)}(G,\alpha).
     $$
\end{definition}

    Our goal is to show the following result. 
    
\begin{theorem}\label{thm: Pullback contraction}
    Let $f\colon \Sp \to \Sp$ be a critically fixed Thurston map and $\{T_n\}_{n\geq 0} \subset \AdmTrees(f)$ be a sequence of admissible trees that satisfies\[T_0\,\xleftarrow{f}\, T_1\, \xleftarrow{f}\, T_2\xleftarrow{f}\dotsb.\] 
    Then $\|T_n\|_f = 0$ for $n \geq \|T_0\|_f$.    
\end{theorem}

In other words, after at most $\|T_0\|_f$ iterations of the pullback relation we obtain a tree that, up to isotopy rel.\ $\Crit(f)$, intersects the charge graph of $f$ only in critical points (to conclude this from the theorem, use 
Proposition \ref{prop: isotop_to_0_intersections}). This establishes Main Theorem \ref{thm_intro: contraction}.

\begin{remark}
    Let $f \colon \Sp \to \Sp$ be a critically fixed Thurston map and $k$ be a non-negative integer. It then follows from Lemma~\ref{lem: Combinatorial complexity} that the number of isotopy classes rel.\ $\Crit(f)$ of admissible planar embedded trees $T$ (for $f$) with $\|T\|_{f} \leq k$ is finite if and only if $\Charge(f)$ is connected, i.e., $f$ is realized by a rational map.
\end{remark}

\begin{example*}
    Let us consider the map $f_{\square}$ and the admissible tree $T$ from Example \ref{ex: Pullback example}. One can check that $\|T\|_f = 1$, but every pullback $T'$ of $T$ under $f$ satisfies $\|T'\|_f = 0$. In fact, we have $\| f^{-1}(T)\|_f=0$ (see  Figures~\ref{fig: Tree and complete preimage} and \ref{fig: Pullbacks}).
\end{example*}

Theorem \ref{thm: Pullback contraction} will easily follow from the next statement.

\begin{proposition}\label{prop: Intersections decreasing}
    Let $f\colon \Sp \to \Sp$ be a critically fixed Thurston map and $T$ be a planar embedded tree in $\Sp$. Then for each 
    edge $\alpha\in E(\Charge(f))$ we have:

    \begin{enumerate}[label = \normalfont{(\roman*)}]
    \item $i_{C(f)}(f^{-1}(T), \alpha) \leq i_{C(f)}(T, \alpha)$; \label{item: first_ineq}
    
    \item $i_{C(f)}(f^{-1}(T), \alpha) < i_{C(f)}(T, \alpha)$, whenever $i_{C(f)}(T, \alpha) > 0$. \label{item: second_ineq}
\end{enumerate}
     
\end{proposition}

\begin{proof} Let $f$, $T$, and $\alpha$ be as in the statement. Since the charge graph of $f$ is defined up to isotopy, we may assume without loss of generality that the tree $T$ and the edge $\alpha\in E(\Charge(f))$ are in minimal position rel.\ $\Crit(f)$. By Proposition \ref{prop: blow-up-triples-properties}, there are distinct lifts $\alpha^+, \alpha^-$ of $\alpha$ under $f$ and a component $U_\alpha$ of $\Sp\setminus (\alpha^+ \cup \alpha^-)$ such that $\alpha^+,\alpha^-$ are isotopic to $\alpha$ rel.\ $\Crit(f)$ and $f|U_\alpha\colon U_\alpha \to \Sp \setminus \alpha$ is a homeomorphism.

We first show that inequality \ref{item: first_ineq} holds.  Since $f|\alpha^+\colon \alpha^+ \to \alpha$ is a homeomorphism, we have 
\[
    |(f^{-1}(T) \cap \alpha^+) \setminus \Crit(f)| = |(T \cap \alpha) \setminus \Crit(f)| =
    i_{C(f)}(T, \alpha).
\]    
At the same time, since $\alpha^+ \sim \alpha$ rel.\ $\Crit(f)$, we have
\[    i_{C(f)}(f^{-1}(T), \alpha) =     i_{C(f)}(f^{-1}(T), \alpha^+) \leq
|(f^{-1}(T) \cap \alpha^+) \setminus \Crit(f)|,\]
which implies the desired inequality.

Now suppose that $i_{C(f)}(T,\alpha)>0$. By the argument above, to prove \ref{item: second_ineq} it is sufficient to show that 
\begin{equation}\label{eq: e+_decrease}
  i_{C(f)}(f^{-1}(T), \alpha^+) <
  |(f^{-1}(T) \cap \alpha^+) \setminus \Crit(f)|
\end{equation}
or
\begin{equation}\label{eq: e-_decrease}
   i_{C(f)}(f^{-1}(T), \alpha^-) <
|(f^{-1}(T) \cap \alpha^-) \setminus \Crit(f)|.
\end{equation}
In other words, our goal is to show that some intersection between $f^{-1}(T)$ and either $\alpha^+$ or $\alpha^-$ can be removed by an isotopy rel.\ $\Crit(f)$.

\begin{figure}[t]
    \centering      
     \begin{overpic}[width=14cm]{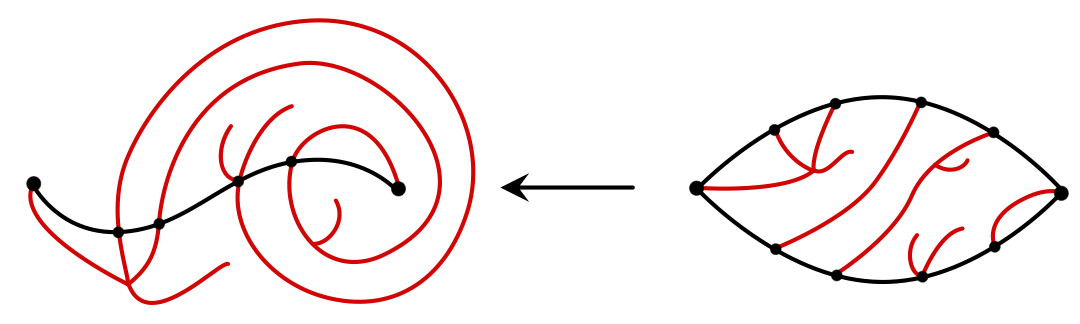}
        \put(52,14.5){$f_\alpha$}
    %    \put(-4,22){$T$}  
        \put(8,10){$p_1$}             
       \put(15.8,7.5){$p_2$}  
%old        \put(19,14){$p_3$}  
        \put(18,13.1){$p_3$}  

        \put(28,13){$p_4$}  
        \put(68.5,19){$p_1^+$}             
        \put(74.5,22){$p_2^+$}  
        \put(86,22){$p_3^+$}  
        \put(93,19){$p_4^+$}  
        \put(67.7,4.5){$p_1^-$}             
        \put(73.5,2){$p_2^-$}  
        \put(86,2){$p_3^-$}  
        \put(93,5){$p_4^-$}          
    %    \put(102,22){$f_\square^{-1}(T)$}
        \end{overpic}
    
    \caption{Left: a planar embedded tree $T$ (in red) intersecting an edge $\alpha\in E(\Charge(f))$ (in black). Right: the preimage $f_\alpha^{-1}(T)$ (in red) and $\partial U_\alpha= \alpha^+\cup \alpha^-$ (in black).}
        
    \label{fig: Graph inside region}
        
\end{figure}

\begin{figure}[t]
    \centering   
     \begin{overpic}[width=14cm]{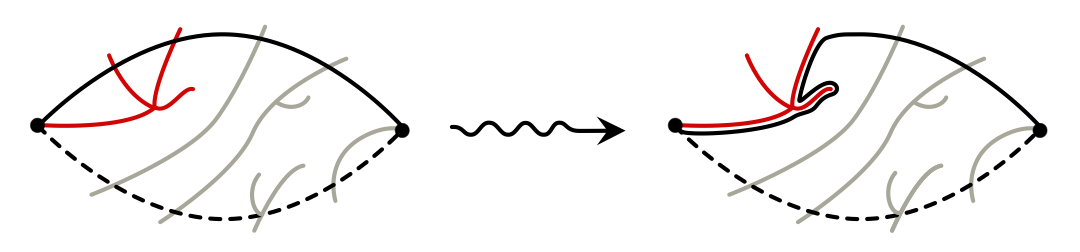}
%        \put(44.5,14){isotopy}
%        \put(42,7){rel.\ $\Crit(f)$}    
        \put(35,16){$\alpha^+$}   
        \put(94,16){$\widetilde{\alpha}$}
        \put(7,19){\textcolor{red}{$H$}}           
        \put(66,19){\textcolor{red}{$H$}}   
    \end{overpic}
    
    \caption{Modifying the Jordan arc $\alpha^+$ into a Jordan arc $\widetilde{\alpha}$ by an isotopy rel.\ $\Crit(f)$ to reduce intersections with $f^{-1}(T)$.}
        
    \label{fig: Removing intersections}
        
\end{figure}

Let $P = \{p_1, p_2, \dots, p_m\}$, $m\geq 1$, be the set of intersection points between $\inter(\alpha)$ and $T$ listed from one endpoint of $\alpha$ to another. 
If we set $P^+:=\inter(\alpha^+)\cap f^{-1}(T)$ and $P^-:=\inter(\alpha^-)\cap f^{-1}(T)$, we can write $P^+ = \{p^+_1, p^+_2, \dots, p^+_m\}$ and $P^- = \{p^-_1, p^-_2, \dots, p^-_m\}$, where $f(p_j^+) = f(p_j^-) = p_j$ for each $j = 1, \dots, m$.

 To simplify the notation, set $f_\alpha := f|\overline{U_\alpha}\colon \overline{U_\alpha} \to \Sp$. We may view the preimage $f_\alpha^{-1}(T)$ as a planar embedded graph with the vertex set $f_\alpha^{-1}(V(T)) \cup P^+ \cup P^-$. Then every connected component of $f_\alpha^{-1}(T)$ is a planar embedded tree with at least one vertex on $\partial U_\alpha=\alpha^+\cup \alpha^-$; see Figure \ref{fig: Graph inside region} for an illustration.

\begin{claim}
There exists a connected component $H$ of $f_\alpha^{-1}(T)$ such that $H\cap \inter(\alpha^+) = \emptyset$ or $H\cap \inter(\alpha^-) = \emptyset$.
\end{claim}

Indeed, let $H^+$ and $H^-$ be the components of $f_\alpha^{-1}(T)$ such that $p_1^+ \in H^+$ and $p_1^- \in H^-$. Note that $H^+\cap H^- = \emptyset$, because otherwise $H^+=H^-$ and $f_\alpha(H^+) \subset T$ contains a cycle, which leads to a contradiction. Now if either $H^+ \cap \inter(\alpha^-) = \emptyset$ or $H^- \cap \inter(\alpha^+) = \emptyset$, then we are done. Hence we may assume that $p_j^-\in H^+$ and $p_s^+\in H^-$ for some $j,s\in \{2,\dots,m\}$. Since $H^+\supset\{p_1^+,p_j^-\}$ is connected and $\partial U_\alpha \supset P^+\cup P^-$ is a Jordan curve, it follows that $p_1^-$ and $p_s^+$ belong to distinct connected components of $\overline{U_\alpha} \setminus H^+$. At the same time, $H^-$ connects $p_1^-$ and $p_s^+$, and thus $H^+\cap H^- \neq \emptyset$, which is a contradiction. The claim follows.

\medskip

Without loss of generality we may assume that there is a connected component $H$ of $f_\alpha^{-1}(T)$ with $H\cap \inter(\alpha^-) = \emptyset$. But since $\overline{U_\alpha}\cap \Crit(f) = \partial \alpha$, the Jordan arc $\alpha^+$ can be isotoped rel.\ $\Crit(f)$ into a Jordan arc $\widetilde{\alpha}\subset \overline{U_\alpha}$ so that  
\[ 
  |(f^{-1}(T) \cap \widetilde{\alpha}) \setminus \Crit(f)| <   |(f^{-1}(T) \cap \alpha^+) \setminus \Crit(f)|;
\]
see Figure \ref{fig: Removing intersections} for an illustration. Thus \eqref{eq: e+_decrease} holds, which completes the proof of part \ref{item: second_ineq}.
\end{proof}

    The first part of Proposition \ref{prop: Intersections decreasing} remains true if $T$ is an arbitrary planar embedded graph in $\Sp$ (in fact, the same proof applies). However, as the next example shows, the assumption that $T$ is a tree is crucial for the second part of the statement. 
     The example also explains why extracting a subtree out of the complete preimage instead of taking the complete preimage itself (see Definition~\ref{def:pullback-operation}) is essential for controlling the topological complexity of the (iterated) pullbacks. 

\begin{example*}
    
    Consider the left part of Figure \ref{fig: Counterexample}. Let $G$ be the planar embedded graph shown in black dashed lines and $H$ be the planar embedded graph shown in colored  solid lines. Suppose $f$ is a critically fixed Thurston map obtained by blowing up the pair $(G, \id_{\Sp})$. The colored lines on the right in Figure \ref{fig: Counterexample} indicate a subgraph $\widetilde H$ of $f^{-1}(H)$. Let $H'$ be the planar embedded graph in $\Sp$ with the same realization as $\widetilde H$ but with vertex set $V(H')=V(H)=\Crit(f)$. Note that $H'$ is isotopic to $H$ rel.\ $\Crit(f)$, and thus $i_{C(f)}(f^{-1}(H), \alpha)= i_{C(f)}(H, \alpha)$ for each edge $\alpha\in E(G)$. This obviously implies that $\|f^{-n}(H)\|_f = \|H\|_f = 2$ for all $n \geq 0$.

    \begin{figure}[t]
        \centering           
        \includegraphics[width=14cm]{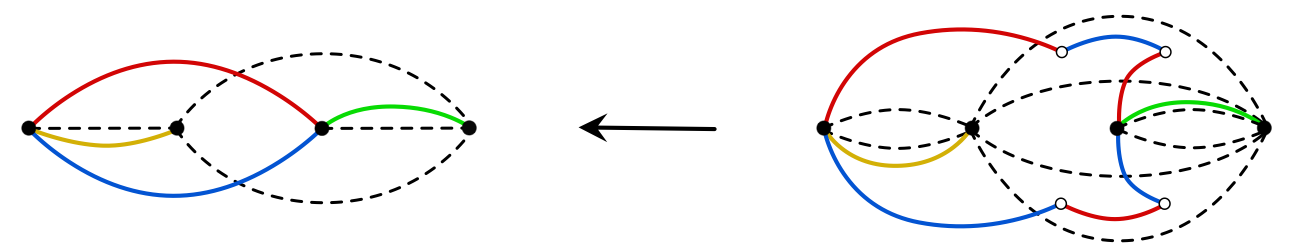}
    
        \caption{Left: a graph $G$ in black dashed  lines and a graph $H$ in colored solid lines. Right: the blow-up of $G$ in black dashed  lines and a subgraph $\widetilde H$ of $f^{-1}(H)$ in colored solid lines. The colors indicate onto which edges of $H$ the edges of $\widetilde H$ are mapped by $f$. 
        }
        
        \label{fig: Counterexample}
        
    \end{figure}
    
\end{example*}

\begin{proof}[Proof of Theorem \ref{thm: Pullback contraction}]
    Let $f$ and $\{T_n\}_{n\geq 0}$ be as in the statement. Proposition \ref{prop: Intersections decreasing} implies that for all $n \geq 0$ we have the inequality $\|f^{-1}(T_n)\|_f \leq \|T_n\|_f$, which is strict unless $\|T_n\|_f = 0$. 
    At the same time, since $T_{n+1}$ is a subset of $f^{-1}(T_n)$, we have  $\|T_{n + 1}\|_f \leq \|f^{-1}(T_n)\|_f$ for each $n \geq 0$. Thus the sequence $\|T_0\|_f, \|T_1\|_f, \|T_2\|_f,\dots$ strictly decreases until we reach $0$. This finishes the proof of the theorem. 
\end{proof}

\subsection{The Lifting Algorithm}\label{subsec: Lifting algorithm} We are finally ready to describe an algorithm that finds the charge graph of a given critically fixed Thurston map $f \colon \Sp \to \Sp$.

Let $T_0$ be an admissible planar embedded tree for $f$, 
and suppose that $T_{n+1} \in \Pi_f(T_{n})$ is a pullback of $T_{n}$ for every $n\geq 0$. By construction, each $T_n$ is a planar embedded tree in $\Sp$ with $V(T_n)\supset \Crit(f)$. 

Let $\beta$ be a simple path in $T_n$ with endpoints in $\Crit(f)$, which we view as a Jordan arc in $\Sp$ joining the corresponding endpoints. Note that $\beta$ may have critical points in its interior. A lift $\widetilde \beta$ of $\beta$ under $f$ is called \emph{critical} if $\widetilde \beta$ is a Jordan arc in $(\Sp,\Crit(f))$, that is, $\partial\widetilde\beta \subset \Crit(f)$ and $\inter(\widetilde\beta) \cap \Crit(f) = \emptyset$. 

By Lemma~\ref{lem: Same blow-up degrees}, there are exactly $\max(0, \deg(f, \widetilde\beta) - 1))$ edges in $\Charge(f)$ that are isotopic to a critical lift $\widetilde\beta$. At the same time, if $i_{C(f)}(\beta,\alpha)=0$ for some edge $\alpha\in E(\Charge(f))$ with $\partial \alpha = \partial \beta$, then there is a critical lift $\widetilde \beta$ of $\beta$ under $f$ such that $\widetilde \beta \sim \alpha$ rel.\ $\Crit(f)$ (see Lemma~\ref{lem: Absence of intersections}). Combining these facts with Theorem \ref{thm: Pullback contraction}, we propose an algorithm for the reconstruction of the charge graph of $f$; see Algorithm~\ref{alg: Lifting algorithm}.

\begin{algorithm}[b]
    \caption{Lifting Algorithm} \label{alg: Lifting algorithm}
%   \SetKwInOut{KwIn}{Input}
%    \SetKwInOut{KwOut}{Output}
\begin{flushleft}
  % \hspace*{\algorithmicindent}  
  \textbf{Input}: a critically fixed Thurston map $f \colon \Sp \to \Sp$\\
 %   \KwIn{critically fixed Thurston map $f \colon \Sp \to \Sp$.}
%    \KwOut{charge graph $\Charge(f)$ of the map $f$.}
 %   \hspace*{\algorithmicindent}    
 \textbf{Output}: the charge graph $\Charge(f)=(\Crit(f),\mathcal{E})$ of the map $f$\\
 \end{flushleft}
    \begin{algorithmic}[1]
        \State set $\mathcal{E}:= \emptyset$ and $n:=0$
        \State choose any planar embedded tree $T_0\in \AdmTrees(f)$ 
        \While{$|\mathcal{E}| < \deg(f) - 1$} \label{Step: While-start}
            \For {every simple path $\beta$ in $T_n$ with endpoints in $\Crit(f)$}
                \For {every critical lift $\widetilde{\beta}$ of $\beta$ under $f$ with $\deg(f,\widetilde{\beta}) >1$} 
                    \If {there is no arc in $\mathcal{E}$ isotopic to $\widetilde{\beta}$ rel.\ $\Crit(f)$} 
                    \State add $\deg(f, \widetilde{\beta}) - 1$ Jordan arcs isotopic to $\widetilde{\beta}$ rel.\ $\Crit(f)$ to $\mathcal{E}$ 
                    \State (so that all arcs in $\mathcal{E}$ have pairwise disjoint interiors)
                    \EndIf
                \EndFor
            \EndFor
            \State\label{Step: Pullback} take $T_{n+1} \in \Pi_f(T_n)$
            \State $n \gets n+1$
        \EndWhile\label{Step: While-end}
        \State \Return $(\Crit(f), \mathcal{E})$
    \end{algorithmic} 
\end{algorithm}

\begin{theorem} \label{thm: speed of alg}
    Let $f\colon \Sp \to \Sp$ be a critically fixed Thurston map and $T_0$ be an initial admissible tree for $f$. 
    Then Algorithm~\ref{alg: Lifting algorithm} stops after taking at most $\|T_0\|_f+1$ iterations of the while cycle (Lines \ref{Step: While-start}--\ref{Step: While-end}) and returns the charge graph of $f$. 
\end{theorem}

\begin{proof}
Theorem~\ref{thm: Pullback contraction} implies that after taking iteratively at most $\|T_0\|_f$ pullbacks starting with the given admissible tree $T_0$, we obtain a planar embedded tree $T'$ with $\Crit(f) \subset V(T')$ and $\|T'\|_f = 0$. Note that for every edge $\alpha \in E(\Charge(f))$, there exists a simple path $\beta$ in $T'$ such that $i_{C(f)}(\beta, \alpha) = 0$ and $\partial \beta = \partial \alpha$. This observation and the discussion before this theorem imply that Algorithm~\ref{alg: Lifting algorithm} terminates and returns the charge graph of $f$ after at most $\|T_0\|_{f} + 1$ iterations.
\end{proof}

        \begin{figure}[t]
            \centering    
            \begin{overpic}[width=\textwidth]%, trim= 0 -30 0 -30]
            {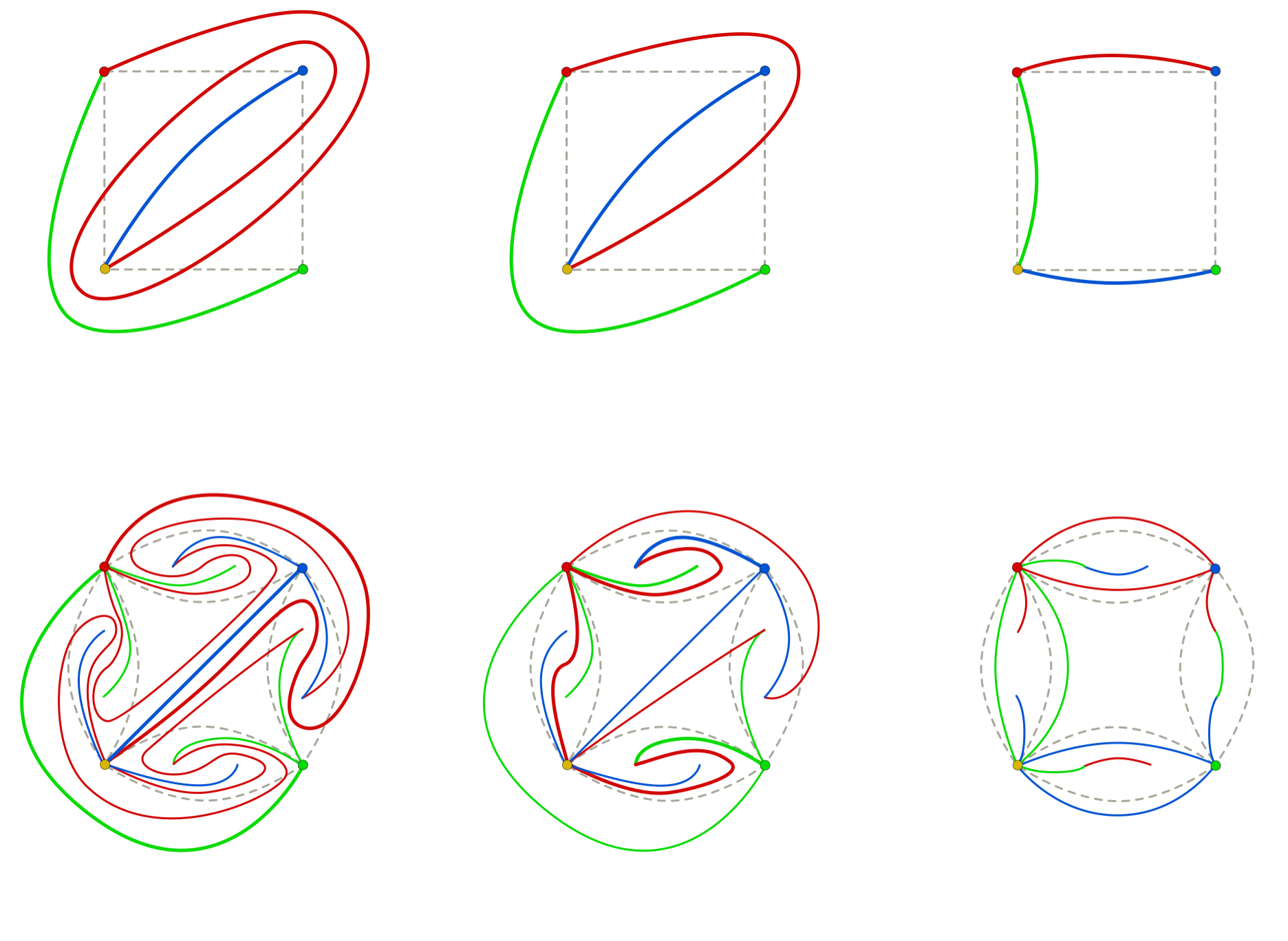}
         \put(5.5,68){$A$}
         \put(24,68){$B$} 
         \put(24,50){$C$} 
         \put(5.5,50){$D$} 
         \put(41.5,68){$A$}
         \put(60,68){$B$} 
         \put(60,50){$C$} 
         \put(41.5,50){$D$} 
         \put(76.5,68){$A$}
         \put(95,68){$B$} 
         \put(95,50){$C$} 
         \put(76.5,50){$D$} 
         \put(15,43){$T_0$}       
         \put(50,43){$T_1$}           
         \put(85,43){$T_2$}   
         \put(13,3){$f_\square^{-1}(T_0)$}       
         \put(48,3){$f_\square^{-1}(T_1)$}           
         \put(83,3){$f_\square^{-1}(T_2)$}  
         \put(5.5,29.5){$A$}
         \put(24,29.5){$B$} 
         \put(24,11.5){$C$} 
         \put(5.5,11.5){$D$} 
         \put(41.5,29.5){$A$}
         \put(60,29.5){$B$} 
         \put(60,11.5){$C$} 
         \put(41.5,11.5){$D$} 
         \put(76.5,29.5){$A$}
         \put(95,29.5){$B$} 
         \put(95,11.5){$C$} 
         \put(76.5,11.5){$D$}  
            \end{overpic}
        
            \caption{Iterations of the Lifting Algorithm applied to the critically fixed Thurston map $f_\square$. 
            }
            
            \label{fig: algo_example}
            
        \end{figure}  

\begin{remark}
    Note that the execution of Algorithm~\ref{alg: Lifting algorithm} significantly depends on the choice of the initial tree $T_0$ and the choices of the pullbacks $T_{n+1}\in \Pi_f(T_n)$ for $n\geq 0$. Regardless of this, the trees $T_n$ have uniformly bounded combinatorial complexity. More precisely, Lemma~\ref{lem: Combinatorial complexity} implies that $|V(T_n)| \leq 2 |\Crit(f)| - 2$, so $|E(T_n)| \leq 2|\Crit(f)| - 3$ for all $n \geq 0$. 
\end{remark}

\begin{example} 
    To illustrate Algorithm~\ref{alg: Lifting algorithm}, we apply it to the critically fixed Thurston map $f_{\square}$ with the charge graph $G_\square$ from Example~\ref{ex: Topological square map}. 

    We start with the admissible tree $T_0$ shown in color on the top left in Figure~\ref{fig: algo_example}. The two other colored trees in the top part of the figure illustrate (up to isotopy rel.\ $\Crit(f_\square)$) the specific choices of the pullbacks $T_1\in \Pi_f(T_0)$ and $T_2\in \Pi_f(T_1)$ we made while executing Line~\ref{Step: Pullback} of Algorithm~\ref{alg: Lifting algorithm}. The colored graphs in the bottom part of the figure show the corresponding complete preimages. Here, as usual, the map $f_\square$ sends each edge of $f_\square^{-1}(T_n)$ to the edge of $T_n$ of the same color. Slightly thicker colored edges in $f_\square^{-1}(T_0)$ and $f_\square^{-1}(T_1)$ indicate the choices of the pullbacks $T_1$ and $T_2$, respectively. The graphs in dashed lines illustrate the charge graph $G_\square$ of $f_\square$ (on the top pictures) and its blow-up $G^\pm_\square$ (on the bottom pictures).

    We now describe how the edge set $\mathcal{E}$ changes during each iteration of the algorithm; see Figure~\ref{fig: algo_example} as a reference. Here, we labeled the points in $C(f_\square) = V(G_\square)=V(G^\pm_\square)$ by $A$, $B$, $C$, and $D$, so that $E(G_\square)=\{AB,BC,CD,DA\}$. 
        \begin{enumerate}
        \item 
        On the first iteration ($n = 0$), we discover the edges $AB$ and $CD$ of the charge graph $G_{\square}$. Namely, the simple path $(A, D, B)$, respectively $(C, A, D)$, in the tree $T_0$ has a critical lift with the blow-up degree $2$ that is isotopic rel.\ $C(f_\square)$ to the edge $AB$, respectively $CD$, of $G_\square$. In other words, after the first iteration, the edge set $\mathcal{E}$ consists of two Jordan arcs that are isotopic rel.\ $C(f_{\square})$ to $AB$ and $CD$, respectively.

        \item On the second iteration ($n = 1$), we discover the edge $DA$ of $G_\square$ using the simple path $(D, A)$ in the tree $T_1$. Note that the simple paths $(A, D, B)$ and $(C, A, D)$ in $T_1$ also have critical lifts that blow up under $f_{\square}$, but they provide us only with the edges $AB$ and $CD$ of $G_\square$ obtained on the previous iteration.

        \item On the third iteration ($n = 2$), we discover the remaining edge $BC$ of $G_{\square}$ using the simple path $(B, A, D, C)$ in the tree $T_2$. After this, we have $|\mathcal{E}| = 4 = \deg(f_\square) - 1$, and the algorithm terminates.

    \end{enumerate}
Note that $\|T_0\|_f = 2$, thus the bound in Theorem \ref{thm: speed of alg} is sharp.

\end{example}

\subsection{Recognizing the combinatorial model}\label{subsec: recover-homeo} In the following, let $f\colon \Sp \to \Sp$ be a critically fixed Thurston map. 
By Theorem~\ref{thm: class-isotopy}, there is a canonical admissible pair $(G, \varphi)$ that corresponds to the map $f$. More precisely, there is a unique (up to isotopy) admissible pair $(G, \varphi)$ such that $f$ is isotopic to a critically fixed Thurston map obtained by blowing up the pair $(G,\varphi)$. The graph $G$ is the charge graph of $f$, which may be recovered using the Lifting Algorithm.  We close this section by discussing how we can combinatorially encode the homeomorphism $\varphi$ (up to isotopy rel.\ $\Crit(f)$) using the graph $G$ and the original map $f$. 

By Proposition \ref{prop: blow-up-triples-properties}, we may find triples $(\alpha^+, \alpha^-, U_\alpha)$, $\alpha\in E(G)$, that satisfy conditions \ref{item: blow-up-tripple-i}-\ref{item: blow-up-tripple-iii} and \ref{item: blow-up-mapping-i}-\ref{item: blow-up-mapping-iv} 
with respect to $f$. (Recall that these triples are uniquely determined by the map $f$ and and the graph $G$ whenever $|\Crit(f)| > 2$.) Fix a connected planar embedded graph $H$ in $\Sp$ with $V(H)= V(G) = \Crit(f)$ and $H\cap \overline{U_\alpha} \in \{\alpha^+,\alpha^-\}$ for each $\alpha\in E(G)$. Conditions \ref{item: blow-up-tripple-i} and \ref{item: blow-up-mapping-iii} imply that $f$ sends $H$ homeomorphically onto its image. The map $f|H\colon H\to f(H)$ may be continuously extended to a homeomorphism $\varphi_f\in \Homeo^+(\Sp, V(H))$. This easily follows from \cite[Proposition~3.4.3]{H_Thesis} since $f$ preserves the cyclic order of edges incident to every vertex of $H$. Alternatively, since $f|H\colon H\to f(H)$ is a graph map, we may use the criterion \cite[Proposition 6.4]{BFH_Class} to check that $f|H$ admits a regular extension (see Section \ref{subsec: covers and graphs}), which then must be a homeomorphism. Note also that the homeomorphism $\varphi_f$ is unique up to isotopy rel.\ $V(H)$ (see Proposition \ref{prop: Graph rigidity}).

The following statement implies that the pair $(G, \varphi_f)$ corresponds to the map $f$. In other words, we may recover the canonical combinatorial model for $f$. 

\begin{proposition}\label{prop: homeo from graph}
    The pair $(G, \varphi_f)$ is admissible. Moreover, the homeomorphisms $\varphi_f$ and $\varphi$ are isotopic rel.\ $V(G)$.
\end{proposition}
\begin{proof}
    It is sufficient to show that $\varphi_f$ and $\varphi$ are isotopic rel.\ $V(G)$. Let $g\colon \Sp \to \Sp$ be a critically fixed Thurston map obtained by blowing up the admissible pair $(G, \varphi)$. Then $f=g\circ \psi$ for some $\psi \in \Homeo_0^+(\Sp, \Crit(f))$. 
    It follows that the triples $(\psi(\alpha^+),\psi(\alpha^-),\psi(U_\alpha))$, $\alpha\in E(G)$, satisfy all the conditions \ref{item: blow-up-tripple-i}-\ref{item: blow-up-tripple-iii} and \ref{item: blow-up-mapping-i}-\ref{item: blow-up-mapping-iv} for the map $g$ and the graph~$G$. In particular, $G^\pm_g = \psi(G^\pm_{f})$, where $G^\pm_f = \bigcup_{\alpha\in E(G)} (\alpha^+\cup \alpha^-)$ and $G^\pm_{g}$ are the blow-ups of $G$ under $f$ and~$g$, respectively.
    
    We claim that $\varphi_f(e)$ and $\varphi(e)$ are isotopic rel.\ $V(G)$ for all $e\in E(H)$. 
    
    \begin{case1}
    Suppose $e\in \{\alpha^+,\alpha^-\}$ for some $\alpha\in E(G)$. Then $\varphi_f(e)=f(e)=\alpha$. At the same time,  $\varphi(e)\sim\varphi(\alpha)$ rel.\ $V(G)$, because $e\sim \alpha$ rel.\ $V(G)$ by \ref{item: blow-up-tripple-i}, and $\varphi(\alpha)\sim\alpha$ rel.\ $V(G)$, because the pair $(G, \varphi)$ is admissible. Thus $\varphi_f(e)\sim \varphi(e)$ rel.\ $V(G)$.
    \end{case1}

    \begin{case2}
    Now suppose $e\notin \{\alpha^+,\alpha^-\}$ for every $\alpha\in E(G)$. By construction, $\inter(\psi(e))\subset \Sp \setminus\bigcup_{\alpha\in E(G)} \psi(\overline{U_\alpha})$.
    Thus we get 
    $(g\circ\psi)(e)=g(\psi(e))\sim \varphi(\psi(e)) \sim \varphi(e)$ rel.\ $V(G)$, where the first isotopy equivalence follows from \ref{item: blow-up-mapping-iv} and the second one from $\psi\in \Homeo_0^+(\Sp, \Crit(f))$. At the same time, $(g\circ\psi) (e)=f(e)= \varphi_f(e)$. Therefore, $\varphi_f(e)\sim \varphi(e)$ rel.\ $V(G)$ as claimed. 
    \end{case2}
    
The proposition now follows from Corollary \ref{cor: Homeo rigidity}.
\end{proof}

\section{The Twisting Problem}\label{sec: Twisting problems}

Let $f\colon \Sp \to \Sp$ be a Thurston map and $\varphi \in \Homeo^+(\Sp, \Post(f))$ be a homeomorphism. Consider the map $g := \varphi \circ f\colon \Sp\to \Sp$, called the \emph{twist of $f$ by $\varphi$} (or simply a \emph{twisted map}). Note that $g$ is a branched covering map on $\Sp$ with $\deg(g) = \deg(f)$ and $\Crit(g)=\Crit(f)$. Moreover, $f$ and $g$ have the same dynamics on the critical set, so $g$ is a Thurston map. In particular, if $f$ is critically fixed, then the twisted map $g$ is critically fixed as well. By the \emph{Twisting Problem} we mean the problem of determining the combinatorial equivalence class of the twisted map $g=\varphi\circ f$, knowing the maps $f$ and $\varphi$. (Here, we implicitly assume that there is a combinatorial classification available for the family of twisted maps.)

Since we are interested in $g$ only up to combinatorial equivalence, we may consider $f$ and $\varphi$ up to isotopy rel.\ $\Post(f)$. In particular, we may treat $\varphi$ as an element of $\PMCG(\Sp, \Post(f))$. It is known that $\PMCG(\Sp, \Post(f))$ is generated by finitely many Dehn twists; see, for example,  \cite[Theorems 4.9 and 4.11]{FarbMargalit}. Thus, understanding the case when $\varphi = T_\gamma^n$, where $n \in \Z$ and $T_\gamma$ is the Dehn twist about an essential Jordan curve $\gamma$ on $(\Sp, \Post(f))$, has been of particular interest. 

The Twisting Problem for polynomial maps has been  sufficiently well studied in the last decade \cite{BarNekr_Twist, KelseyLodge, LifitingTrees, HighDegreeRabbitTwisting, CubicTwisting}. However, in the non-polynomial case, the problem has been previously considered only for rational maps of low degree with four postcritical points \cite{Lodge_Boundary, KelseyLodge}. 

In this section, we address the Twisting Problem for the family of critically fixed Thurston maps. We start though by briefly discussing the principal previous work in the polynomial case; see \cite{LifitingTrees} for a more thorough overview.

The first instance of the Twisting Problem was posed by John Hubbard in the 1980s. Namely, let us consider the quadratic polynomials of the form $z^2+c$ for which the critical point $0$ is $3$-periodic. There are exactly three such polynomials, called the \emph{rabbit polynomial} $p_R$ (with $c\approx -0.1225 + 0.7448i$), the \emph{co-rabbit polynomial} $p_C$ (with $c\approx -0.1225 - 0.7448i$),  and the \emph{airplane polynomial} $p_A$ (with $c\approx -1.7548$). Take the rabbit polynomial $p_R$ and postcompose it with (an iterate of) the Dehn twist $T_\gamma$ about a Jordan curve $\gamma$ surrounding the postcritical points $p_R(0)$ and $p^2_R(0)$ and staying in the upper half-plane in $\C$. The Levy--Bernstein theorem \cite[Theorem~10.3.9]{HubbardBook2} implies that the twisted map $T_\gamma^n \circ p_R$ is realized by a rational map for every $n\in \Z$, and thus it must be combinatorially equivalent to precisely one of the polynomials $p_R$, $p_C$, and $p_A$. The problem of finding the combinatorial equivalence class of $T_\gamma^n \circ p_R$ (as a function of $n\in \Z$) is called \emph{Hubbard's twisted rabbit problem}.

Laurent Bartholdi and Volodymyr Nekrashevych solved the twisted rabbit problem in \cite{BarNekr_Twist} using a novel algebraic machinery provided by \emph{iterated monodromy groups}. Quite unexpectedly, the answer depends on the $4$-adic expansion of the power $n$ of the twist.
A similar algebraic approach can be applied to the Twisting Problem for some other polynomial and non-polynomial rational maps. For example, twisting questions for 
$z^2+i$ and $\displaystyle \frac{3z^2}{2z^3+1}$ were considered in \cite{BarNekr_Twist} and  \cite{Lodge_Boundary}, respectively. We remark that in these cases the twisted maps may be obstructed. 

Recently, an alternative approach to the Twisting Problem for postcritically-finite polynomials was proposed by  James Belk, Justin Lanier, Dan Margalit, and Rebecca Winarski \cite{LifitingTrees}. Similarly to the work of Bartholdi and Nekrashevych, the solution is algorithmic, but it is based on ideas from combinatorial topology to a greater extent; see \cite{LifitingTrees} for a comparison of the two methods.

\medskip

The combinatorial classification and the Lifting Algorithm we developed for critically fixed Thurston maps in the preceding sections naturally suggest an algorithmic solution to the Twisting Problem in this setting.

\begin{figure}[t]    
\centering    
\begin{overpic}[width=\textwidth]%, trim= 0 -30 0 -30]
            {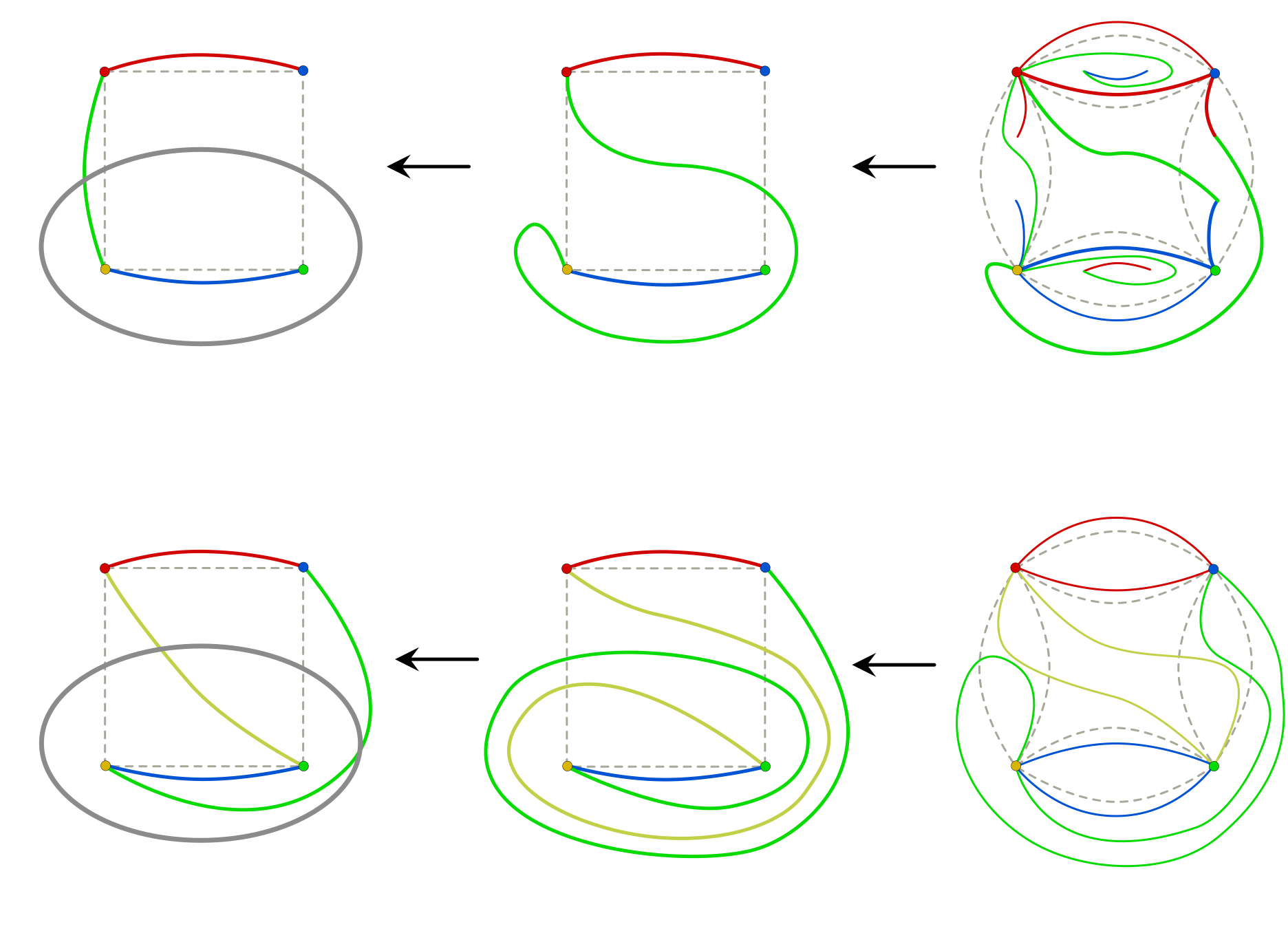}
       \put(5.5,68){$A$}
         \put(24,68){$B$} 
         \put(24,50){$C$} 
         \put(5.5,50){$D$} 
         \put(41.5,68){$A$}
         \put(60,68){$B$} 
         \put(60,50){$C$} 
         \put(41.5,50){$D$} 
         \put(76.5,68){$A$}
         \put(94.5,68){$B$} 
         \put(94.5,50){$C$} 
         \put(76.5,50){$D$}

        \put(32.5,62){$T^{-1}_\gamma$} 
        \put(69,62){$f_\square$} 

        \put(2,57.5){$\gamma$} 
        \put(2,19){$\gamma$}

        \put(32.5,23.5){$T^{-1}_\gamma$} 
        \put(69,23.5){$f_\square$} 
        
         \put(15,43){$T_0$}       
         \put(48,43){$T_\gamma(T_0)$}           
     %    \put(79,43){$(f_\square^{-1}\circ T_\gamma)(T_0)$}
            \put(83.5,43){$g^{-1}(T_0)$} 
        
         \put(15,3){$G$}       
         \put(49,3){$T_\gamma(G)$}           
         \put(86,3){$G^\pm$}         
         
%         \put(13,3){$f_\square^{-1}(T_0)$}       
%         \put(48,3){$f_\square^{-1}(T_1)$}           
%         \put(83,3){$f_\square^{-1}(T_2)$}  
         \put(5.5,29.5){$A$}
         \put(24,29.5){$B$} 
         \put(23.5,11.5){$C$} 
         \put(5.5,11.5){$D$} 
         \put(41.5,29.5){$A$}
         \put(60,29.5){$B$} 
         \put(60,11.5){$C$} 
         \put(41.5,11.5){$D$} 
         \put(76.5,29.5){$A$}
         \put(94.5,29.5){$B$} 
         \put(94.5,11.5){$C$} 
         \put(76.5,11.5){$D$}              
\end{overpic}       \captionsetup{width=.9\linewidth}
    \caption{Finding the combinatorial equivalence class of the map $g=T^{-1}_\gamma\circ f_\square$.}
    \label{fig: twsisting example}
\end{figure}

\begin{example}\label{ex: Twisting by lifting}
Let us show how using Algorithm~\ref{alg: Lifting algorithm} one can find the combinatorial equivalence class of the twisted map $g := T_{\gamma}^{-1} \circ f_{\square}$, where $f_\square$ is the critically fixed Thurston map from Example \ref{ex: Topological square map} and $\gamma$ is the gray Jordan curve 
on the top left picture in Figure \ref{fig: twsisting example}. Here, the graph in dashed lines represents the charge graph $G_{\square}$ of $f_\square$. Furthermore, the points in $C(f_\square) = V(G_\square)$ are labeled by $A$, $B$, $C$, and $D$. 

We run Algorithm~\ref{alg: Lifting algorithm} for the map~$g$ starting with the admissible tree $T_0$ shown in color on the top left picture in Figure \ref{fig: twsisting example}. During the first iteration of the algorithm, it will find out that there are critical lifts of the simple paths $(A,B)$, $(C,D)$, $(A,D,C)$, and $(B,A,D)$ in the tree $T_0$ providing four different edges of the charge graph of the map $g$. Examples of such lifts are shown in slightly thicker colored lines on the top right picture in Figure \ref{fig: twsisting example}, which illustrates the full preimage $g^{-1}(T_0)$. Here, the graph in dashed lines represents the blow-up of $G_{\square}$ under $f_\square$. Furthermore, $f_\square$ sends each edge of $g^{-1}(T_0)$ to the edge of the same color in the tree $T_\gamma(T_0)$ on the top middle picture in Figure \ref{fig: twsisting example}. Similarly, $T_\gamma^{-1}$ sends each edge of $T_\gamma(T_0)$ to the edge of the same color in the tree $T_0$ on the top left picture. 

The bottom part of Figure \ref{fig: twsisting example} verifies that the found critical lifts indeed blow up under the twisted map~$g$. Namely, the picture on the left shows a graph $G$ composed of these lifts (up to isotopy rel.\ $\Crit(f_\square)$); the middle picture shows $T_\gamma(G)$; and the picture on the right illustrates the blow-up $G^\pm$ of $G$ under the map $g$.  

It follows that Algorithm~\ref{alg: Lifting algorithm} stops after the very first iteration. Note that the charge graph~$G_\square$ of $f$ and the charge graph $G$ of $g=T^{-1}_\gamma \circ f$ are connected and isomorphic to each other. Therefore, the maps $f_{\square}$ and $T_{\alpha}^{-1}\circ f_{\square}$ are combinatorially equivalent, even though they are not isotopic (see Remark~\ref{rem: adm-pair-for-connected}).

\end{example}

The main goal of this final section is to address some special instances of the Twisting Problem for the family of critically fixed Thurston maps $f$ obtained by blowing up admissible pairs $(G, \id_{\Sp})$. Note that up to isotopy this family includes all critically fixed rational maps (see Proposition~\ref{prop: charge graph rational case}). We will show that for some special Jordan curves $\gamma$ in $(\Sp, \Crit(f))$ the combinatorial equivalence class of the twisted map $T^n_\gamma\circ f$, $n\in\Z$, can be determined by applying a simple combinatorial operation to the charge graph $G$ of the initial map $f$.

\subsection{Simple transversals and their properties}\label{subsec: special curves}
Let $G$ be a planar embedded graph in~$\Sp$. We say that an essential Jordan curve $\gamma$ in $(\Sp, V(G))$ is a \emph{simple transversal with respect to $G$} if it satisfies the following two conditions:
\begin{enumerate}[label = \normalfont{(\roman*)}]
\item $i_{V(G)}(G, \gamma)= |G\cap \gamma|$, that is, $G$ and $\gamma$ are in minimal position rel.\ $V(G)$;
\item $|e\cap \gamma| \leq 1$ for each edge $e\in E(G)$.
\end{enumerate}
Note that if $G$ is connected, then the set of isotopy classes $[\gamma]$ rel.\ $V(G)$ for the simple transversals $\gamma$ with respect to $G$ is finite (essentially, $[\gamma]$ is determined by the order in which $\gamma$ crosses the edges of $G$). In fact, the converse is also true when $|V(G)|\geq 4$. (If $|V(G)|< 4$, then there are no simple transversals, as they are assumed to be essential.)

\begin{lemma}\label{prop: Preimage lemma}
    Let $f\colon \Sp \to \Sp$ be a critically fixed Thurston map obtained by blowing up a pair $(G, \id_{\Sp})$,  
    and let $\gamma$ be a simple transversal with respect to $G$. Then the following statements are true:
        
    \begin{enumerate}[label = \normalfont{(\roman*)}, leftmargin=*]
        \item\label{item: preimage-lemma-i} There exists a unique component $\gamma'$ of $f^{-1}(\gamma)$ that is isotopic to $\gamma$ rel.\ $V(G)$. 
        Moreover, $\deg(f|\gamma')=|G\cap \gamma|+1$.
       
         \item\label{item: preimage-lemma-ii} All other components $\delta'$ of $f^{-1}(\gamma)$ are null-homotopic in $(\Sp,V(G))$ and satisfy \linebreak $\deg(f|\delta')=1$.
    \end{enumerate}
\end{lemma}

\begin{proof}
    The lemma easily follows from the definition of the blow-up operation; see Definition~\ref{def: Blow up}. Indeed, let $f$ and $\gamma$ be as in the statement. In particular, we fix a choice of $W_e$, $D_e$, $f_e$, and $h$ as in Section \ref{subsec: The blow-up operation}. 
    
    We assume below that $m:=|G\cap \gamma|\geq 2$; the proof can be easily adapted for the remaining two cases.  Let $\beta_1,\beta_2,\dots, \beta_m$ be all the edges of $G$ that $\gamma$ intersects. We label these edges in the order they are met by $\gamma$ (for some chosen basepoint and orientation on~$\gamma$). Then the Jordan curve $\gamma$ can be broken into $m$ consecutive Jordan arcs $\gamma_1, \gamma_2, \dots, \gamma_m$ having endpoints $x_j \in \beta_j$ and $x_{j + 1} \in \beta_{j + 1}$ for each $j = 1, 2, \dots, m$. Here and further all indices are understood modulo $m$.
    
    By \eqref{eq: blow-up_def}, $f^{-1}(\gamma)\cap D_{\beta_j}$ is a Jordan arc $\gamma^\pm_j$ connecting two preimages $x_j^-\in \partial D_{\beta_j}^-$ and $x_j^+\in \partial D_{\beta_j}^+$ of $x_j$ under $f$. Moreover, up to relabeling, we may assume that 
    $x_j^+$ and $x^-_{j+1}$ are connected by a lift $\gamma'_j \subset \Sp \setminus \bigcup_{j = 1}^m \inter(D_j)$ of $\gamma_j$ under $f$. The concatenation of the arcs $\gamma_1^\pm,\,\gamma_1',\, \gamma_2^\pm,\, \gamma_2',\,\dots,\, \gamma_m^\pm,\, \gamma'_m$ is a Jordan curve $\gamma'$. Moreover, $f|\gamma'\colon \gamma' \to \gamma$ is a covering map of degree $m+1$. 
    
    We modify $h$ within $\bigcup_{j = 1}^m \inter(D_j)$ so that $h_1(\gamma^\pm_j)=x_j$ for all $j=1,\dots,m$. (Since this does not change the isotopy class of $(f,V(G))$, it does not affect the desired statement by Proposition~\ref{prop: isotopy-lifting}.) Then by \eqref{eq: blow-up_def} we have
    \[h_1(\gamma')=h_1\left(\bigcup_{j=1}^m \gamma'_j\right)=f\left(\bigcup_{j=1}^m \gamma'_j\right)=\bigcup_{j=1}^m \gamma_j=\gamma.\]
    It follows that $h_t|\gamma'$, $t\in \I$, provides a non-ambient isotopy rel.\ $V(G)$ between $\gamma'=h_0(\gamma')$ and $\gamma=h_1(\gamma')$. Hence $\gamma'$ and $\gamma$ are isotopic rel.\ $V(G)$.

    Finally, if $\delta'$ is a component of $f^{-1}(\gamma)$ that is different from $\gamma'$, then $\delta' \subset \inter(D_e)$ for some $e\in E(G)\setminus\{\beta_1,\beta_2,\dots,\beta_m\}$. Hence, $\delta'$ is null-homotopic in $(\Sp,V(G))$ and $\deg(f|\delta')=1$. This completes the proof of the lemma.
\end{proof}

\begin{remark}
We note that simple transversals with respect to the charge graph of a critically fixed rational map $f\colon \widehat{\C}\to\widehat{\C}$ correspond exactly to the essential curves in the \emph{global curve attractor} of~$f$; see \cite[Proposition 10]{H_Tischler} and Lemma \ref{prop: Preimage lemma}. In particular, for every Jordan curve $\gamma$ in $(\widehat{\C}, \Crit(f))$ there exists $n\in \N$ such that each component of $f^{-n}(\gamma)$ is either non-essential or  isotopic rel.\ $\Crit(f)$ to a simple transversal with respect to $\Charge(f)$. 
In fact, the analogous statement is true for critically fixed Thurston maps $f\colon \Sp \to \Sp$ obtained by blowing up a pair $(G, \id_{\Sp})$ and Jordan curves in $(\Sp, V(G))$. It follows that the global curve attractor of such a critically fixed Thurston map $f$ is finite if and only if $G$ is connected.
\end{remark}

\subsection{Graph rotation}\label{subsec: Combinatorial operation}
Our goal now is to introduce a simple combinatorial operation on planar embedded graphs that will allow us to describe the action of (special) Dehn twists on admissible pairs $(G, \id_{\Sp})$.

Consider the map $T_{n/m}\colon \partial \D \times \I \to \partial \D \times \I$, where $n \in \Z$ and  $m \in \N$, defined as
    $$
        T_{n/m}\left(e^{2\pi i \theta}, t\right) = \left(e^{2\pi i\left(\theta + t \frac{n}{m}\right)}, t\right).
    $$ 
The map $T_{n/m}$ is called the \emph{$n/m$-twist} on the cylinder $\partial \D \times \I$. Note that $T_{n/m}$ fixes the boundary circle $\partial \D \times \{0\}$ pointwise and ``rotates'' the boundary circle $\partial \D \times \{1\}$ counterclockwise by angle $2\pi\frac{n}{m}$. Furthermore, $T_{n/m}=(T_{1/m})^n$ for all $n\in\Z$ and $m\in \N$. Figure \ref{fig: Fractional twists} illustrates the action of the $1/4$-twist on the radial arcs $\{e^{2\pi i \theta}\} \times  \I$ in the cylinder $\partial \D \times \I$. Here, the cylinder $\partial \D \times \I$ is viewed as an annulus in $\C$ under the embedding $(e^{2\pi i\theta},t)\mapsto e^{2\pi i\theta}(t+1)$.

    \begin{figure}[t]
     \centering   
     \begin{overpic}[width=0.9\textwidth]{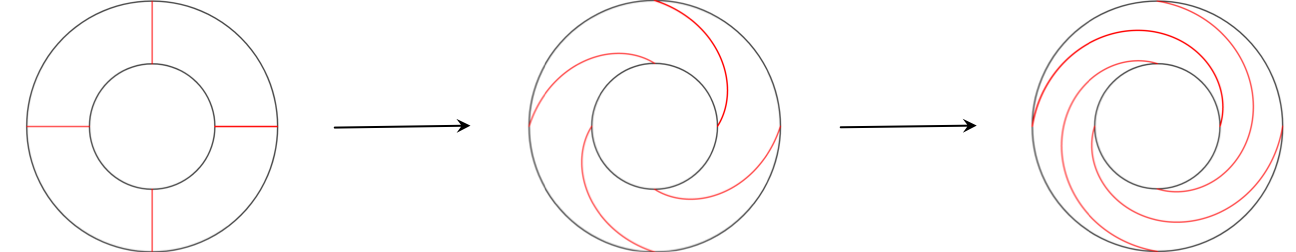}
        \put(28,11){$T_{1/4}$}
        \put(67,11){$T_{1/4}$}
     \end{overpic}
    \caption{The action of the ${1/4}$-twist on radial arcs.}
    \label{fig: Fractional twists}
    \end{figure}

Let $G$ be a planar embedded graph in~$\Sp$, and suppose that $\gamma$ is a simple transversal with respect to $G$ with $m:=|G\cap \gamma|\geq 1$. We denote by $\beta_1,\dots, \beta_m$ all the edges of $G$ that meet $\gamma$. Note that each of these edges intersects $\gamma$ transversely. Finally, we assume that the edges $\beta_1,\dots, \beta_m$ are labeled in the order they are met by $\gamma$ (for some chosen basepoint and orientation on $\gamma$).

Let us thicken the curve $\gamma$ to a (small) closed annulus $A\subset \Sp \setminus V(G)$ so that $A\setminus G$ has exactly $m$ connected components. Choose an orientation-preserving homeomorphism $\phi\colon \partial \D \times \I \to A$ so that $\phi(r_{j, m})= \beta_j \cap A$ for every $j=1,\dots, m$, where 
$r_{j,m}:=\displaystyle\{e^{2\pi i \frac{j-1}{m}}\}\times \I$ are radial arcs in the cylinder $\partial \D \times \I$ subdividing it into $m$ congruent pieces. 

Now consider the map $T_{n/m, \phi}\colon \Sp\to \Sp$ defined as
$$
    T_{n/m, \phi}(p)=\begin{cases}
        (\phi \circ T_{n/m} \circ \phi^{-1})(p) & \text{ if } p \in A\\
        p & \text{ if } p \in \Sp \setminus A.
    \end{cases}
$$    
We call $T_{n/m,\phi}$ the \emph{$n/m$-twist of $A$ with respect to $\phi$}. 

Note that $T_{n/m,\phi}$ is a homeomorphism of $\Sp$ if and only if $n/m \in \Z$. In fact, the map $T_{n/m,\phi}$ should be thought of as a ``fractional'' Dehn twist: when $q:=n/m \in \Z$, $T_{n/m,\phi}$ coincides (up to isotopy rel.\ $V(G)$) with the $q$-th iterate $T^{q}_\gamma$ of the Dehn twist $T_\gamma$ about the curve $\gamma$. For $n/m\notin \Z$, the map $T_{n/m,\phi}$ fixes one of the boundary curves of $A$ pointwise and ``rotates'' the other one to the left when viewed from the inside of $A$.  

Consider the image $G':= T_{n/m,\phi}(G)$ of the graph $G$ under the $n/m$-twist of $A$ with respect to $\phi$. It is easy to see that $G'$ may be viewed as a planar embedded graph in $\Sp$ with the vertex set $V(G)$. Note that then $T_{n/m,\phi}$ modifies only the edges $\beta_1,\dots, \beta_m$ of $G$ and keeps all other edges fixed. One can check that, up to isotopy rel.\ $V(G)$, the graph $G'$ is uniquely defined independently of the choice of $A$ and $\phi$. 

\begin{definition}\label{def: Graph rotation}
The planar embedded graph $G'$ constructed as above is called the \emph{$n$-rotation} of the graph $G$ about the curve $\gamma$.
\end{definition}

\begin{example}
\label{ex: Rotation example} 
    Figure \ref{fig: Rotation} illustrates the $1$-rotations of the square graph $G_{\square}$ about 
    two simple transversals $\gamma_1$ and $\gamma_2$ with respect to $G_\square$. Here, the pictures in the left column indicate the chosen annuli $A_{\gamma_1}$ and $A_{\gamma_2}$ around $\gamma_1$ and $\gamma_2$, respectively. The red arcs correspond to the intersections of the annuli with the graph. In the middle column, we see the images of $G_{\square}$ under the $1/2$-twist of $A_{\gamma_1}$ (top) and the $1/4$-twist of $A_{\gamma_2}$ (bottom), that is, the $1$-rotations of $G_\square$ about $\gamma_1$ and $\gamma_2$, respectively. The red arcs indicate the modifications of the edges. Finally, the right column shows the same graphs after simplification by isotopy rel.\ $V(G_\square)$.
    
    \begin{figure}[t]
     \centering   
     \begin{overpic}[width=0.9\textwidth]{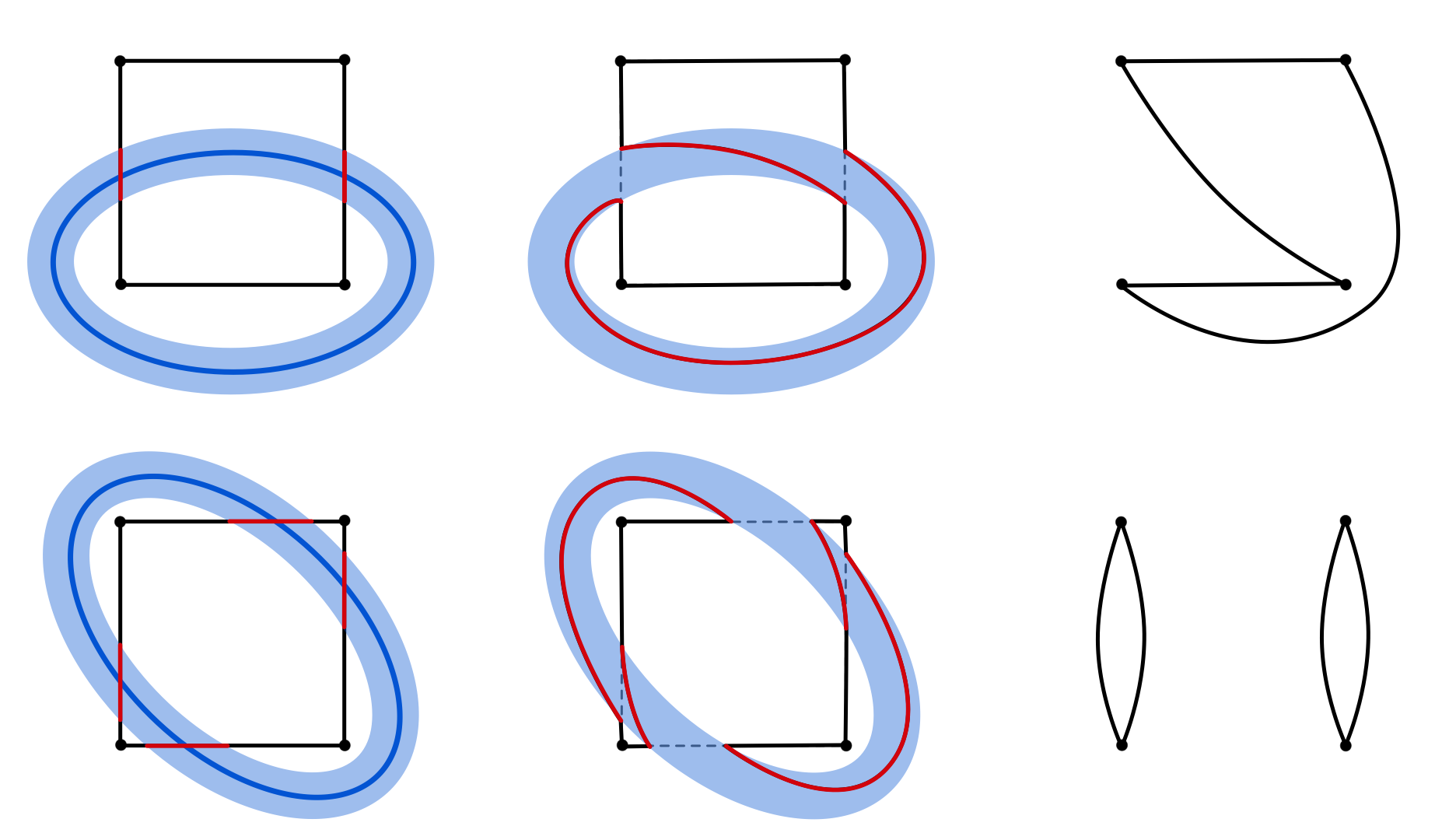}
        \put(-6,34){$\gamma_1\subset A_{\gamma_1}$}
        \put(-5,11){$\gamma_2 \subset A_{\gamma_2}$}
        \put(66,44){\huge$\sim$}
        \put(66,13){\huge$\sim$}
    \end{overpic}    
        \caption{The 1-rotations of the square graph $G_{\square}$ about the simple transversals 
        $\gamma_1$ (top) and $\gamma_2$ (bottom).}
        \label{fig: Rotation}
    \end{figure}

\end{example}

In analogy with the usual Dehn twists, we have the following statement. 

\begin{proposition}\label{prop: rotations up to isotopy}
    Let $G$ and $G'$ be two planar embedded graphs in $\Sp$ with a common vertex set~$V$, and let $\gamma$ and $\gamma'$ be simple transversals with respect to $G$ and $G'$, respectively, with $|G\cap \gamma|\geq 1$ and $|G'\cap \gamma'|\geq 1$.
    
  If $G$ is isotopic to $G'$ and $\gamma$ is isotopic to $\gamma'$ rel.\ $V$, then the $n$-rotations of $G$ about $\gamma$ and of $G'$ about $\gamma'$ are isotopic rel.\ $V$ for all $n \in \Z$.
\end{proposition}

\subsection{Twists about simple transversals}\label{subsec: Main theorem C}
We now look at a special instance of the Twisting Problem. Namely, we consider a critically fixed Thurston map $f$ obtained by blowing up an admissible pair $(G, \id_{\Sp})$ and (iterates of) the Dehn twist $T_\gamma$ about a simple transversal $\gamma$ with respect to~$G$. We are going to describe the combinatorial equivalence classes of the twisted maps $T^n_\gamma \circ f$, $n\in \Z$, using the graph rotation operation introduced in Section \ref{subsec: Combinatorial operation}.

\begin{proposition}\label{prop: Key twisting proposition}
    Let $f\colon \Sp \to \Sp$ be a critically fixed Thurston map obtained by blowing up an admissible pair $(G, \id_{\Sp})$ and $\gamma$ be a simple transversal  with respect to $G$ with $|G\cap \gamma| \geq 1$. 
    Then the twisted map $T_\gamma^{-1} \circ f$ is isotopic to a critically fixed Thurston map obtained by blowing up the admissible pair $(H, \id_{\Sp})$, where $H$ is the $1$-rotation of $G$ about the curve $\gamma$. In particular, $H$ is the charge graph of $T_\gamma^{-1} \circ f$. 
\end{proposition}

    \begin{proof}
    Suppose $f$ and $\gamma$ are as in the statement. In particular, we fix a choice of $W_e$, $D_e$, $f_e$, and $h$ as in Section \ref{subsec: The blow-up operation}. Then $V(G)=\Crit(f)$ and $\deg(f)=|E(G)|+1$. Set $m := |G\cap \gamma| = i_{V(G)}(G,\gamma)$, and let $A$ be a (small) closed annulus in $\Sp\setminus V(G)$ obtained by thickening the curve $\gamma$ so that $A\setminus G$ has exactly $m$ components.
    
    Let us denote by $H$ the $1$-rotation of $G$ about the curve $\gamma$ realised by the $1/m$-twist $T_{1/m,\phi}$ of $A$ with respect to some orientation-preserving homeomorphism $\phi\colon \partial \D \times \I \to A$ as in Section~\ref{subsec: Combinatorial operation}. Without loss of generality, we may also assume that the Dehn twist $T_\gamma$ is defined with respect to the same homeomorphism $\phi$, so that $T_\gamma$ is the identity on $\Sp \setminus \inter(A)$. We are going to show that each edge $\alpha \in E(H)$ of $H$ has a triple $(\alpha^+, \alpha^-, U_\alpha)$ satisfying conditions \ref{item: blow-up-tripple-i}-\ref{item: blow-up-tripple-iii} with respect to the map $g:=T_\gamma^{-1} \circ f$. It would then follow from Proposition \ref{prop: Blow-up triples} and Proposition \ref{prop: admis_equiv}\ref{item: case isotopy} that $H$ is the charge graph of $g$. 
    
    By adjusting $D_e$, $f_e$, and $h$ (which does not change the isotopy class of $f$), we may assume the following:
\begin{itemize}
    \item for every $e\in E(G)$, we have $D_e\cap A =\emptyset$, whenever $e\cap A = \emptyset$ (i.e., when $e\cap \gamma=\emptyset)$;
    \item  $A$ is a component of $f^{-1}(A)$, so that $f|A\colon A \to A$ is a self-covering of degree $m + 1$ (compare Lemma~\ref{prop: Preimage lemma}).
\end{itemize}

    We are going to set up some notation now. Suppose $X$ and $Y$ are the two components of $\Sp\setminus\inter(A)$, where $X$ is the component with $\partial X = \phi(\partial\D \times \{0\})$ and $Y$ is the component with $\partial Y = \phi(\partial\D \times \{1\})$. Then $T_{1/m,\phi}$ fixes $\partial X$ pointwise, and it rotates $\partial Y$ to the left when viewed from the inside of $A$. 
    
    Let $E^A(G):=\{\beta_1,\dots, \beta_m\}$ be the set of all edges of $G$ that intersect $A$. These edges subdivide the annulus $A$ into $m$ closed components $A_1,\dots,A_m$. We label these edges and components so that they are met in the order 
    \[\beta_1,\,A_1,\,\beta_2,\,A_2,\,\dots ,\,\beta_m,\, A_m\]
    when we walk around the Jordan domain $X$ in the counter-clockwise direction. For a reference, see the top middle picture in Figure \ref{fig: twisting_proof}. Here, the graph $G$ and the annulus $A$ are shown in black and blue colors, respectively.

    \begin{figure}[t]
     \centering   
     \begin{overpic}[width=\textwidth]{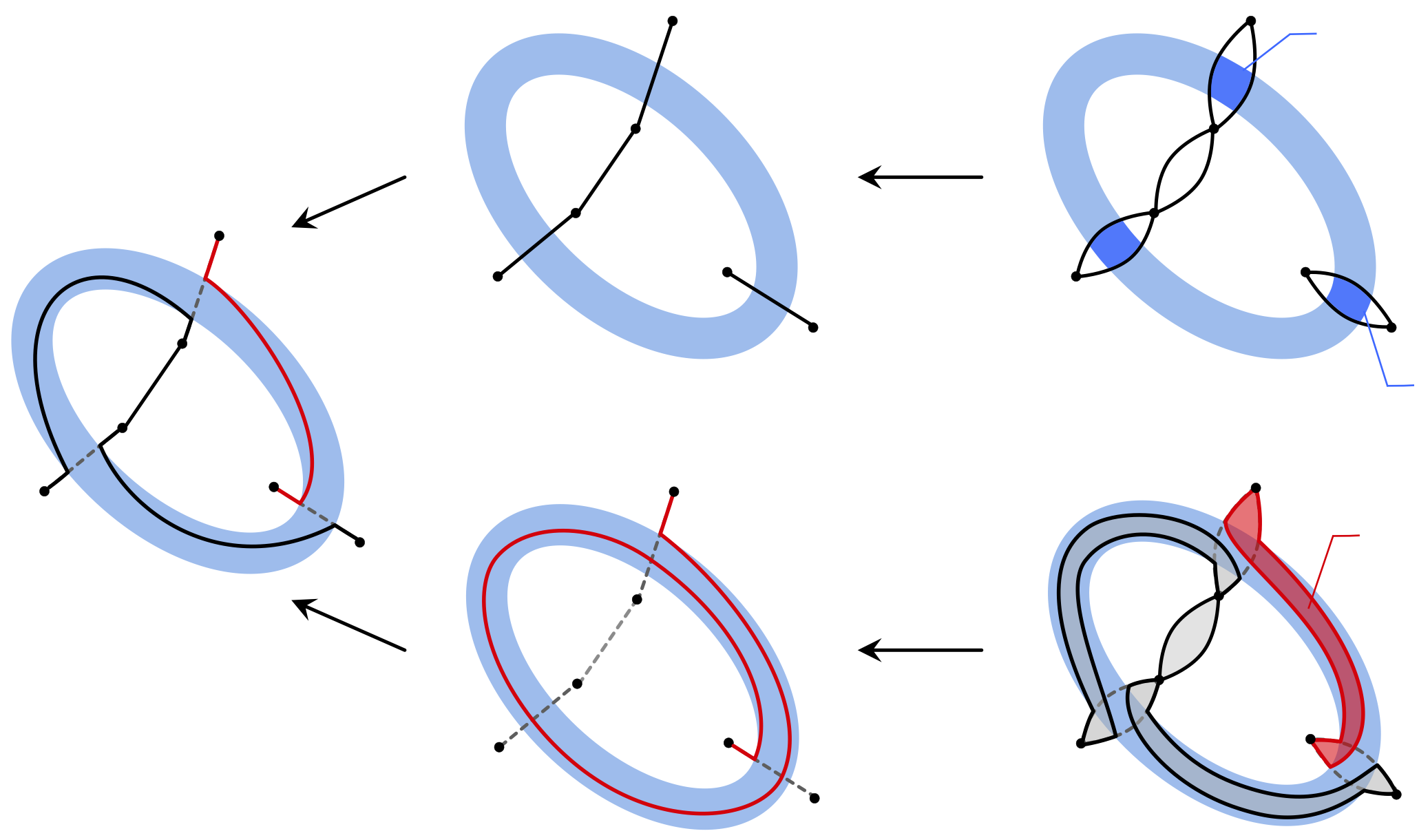}
        \put(51,47){\textcolor{blue}{$A_{j}$}}
  %      \put(48,38){$\beta_{j,X}$}
  %      \put(53,40){$\beta_{j,A}$}
  %      \put(55,34){$\beta_{j,Y}$}        
%        \put(53.5,39){$\beta_{j}$}     
%        \put(41,54){$\beta_{j+1}$}  
%old    \put(56,33.5){$\beta_{j}$}          
        \put(52,36){$\beta_{j}$}  \put(41,54){$\beta_{j+1}$}
        
%old    \put(48,57){$\beta_{j+1}$}  

      \put(91.8,46.5){\textcolor{blue}{$A'_{j}$}}
        
%old        \put(89,55.5){$\beta^-_{j+1}$}  
%old        \put(82.5,57.5){$\beta^+_{j+1}$}  
%old        \put(95,34){$\beta^-_{j}$}  
%old        \put(98,38){$\beta^+_{j}$}
        
%old        \put(85.1,53){\textcolor{blue}{$A^\pm_{j+1}$}}  

%old        \put(93.5,37.5){\textcolor{blue}{$A^\pm_{j}$}}     

         \put(88,51.5){$\beta^-_{j+1}$}  
        \put(80.5,54.2){$\beta^+_{j+1}$}  
        \put(92,35){$\beta^-_{j}$}  
        \put(94.5,40.5){$\beta^+_{j}$}

        \put(91.2,57.2){\textcolor{blue}{$A^\pm_{j+1}$}}  

        \put(98,32.5){\textcolor{blue}{$A^\pm_{j}$}}

        \put(21,33){\textcolor{red}{$\alpha_{j}$}}     
        \put(52,17){\textcolor{red}{$T_\gamma(\alpha_{j})$}}    
  %      \put(94,13){\textcolor{red}{$U_{\alpha_j}$}}   
        \put(94.3,22){\textcolor{red}{$U_{\alpha_j}$}}          \put(95,14){\textcolor{red}{$\alpha_j^-$}}   
        \put(89,13){\textcolor{red}{$\alpha_j^+$}}   
        
 %       \put(89.5,22){\textcolor{red}{$\alpha^-_{j}$}}  
 %       \put(84,24){\textcolor{red}{$\alpha^+_{j}$}}

        \put(37,48){$X$}
        \put(37,35){$Y$}        
        \put(65,48){$f$}
        \put(65,11){$f$}
        \put(20,47){$T_{1/m,\phi}$}
        \put(22.5,12){$T^{-1}_\gamma$}     
    \end{overpic}    
        \caption{Proof of Proposition \ref{prop: Key twisting proposition}.}
        \label{fig: twisting_proof}
    \end{figure}    
    
    The arcs $\beta^+_j:=\partial D^+_{\beta_j}$ and $\beta^-_j:=\partial D^-_{\beta_j}$, $j=1,\dots, m$, 
subdivide the annulus $A$ into $2m$ closed components $A^\pm_1, A'_1, \dots, A^\pm_m, A'_m$ so that $A^\pm_j=A\cap D_{\beta_j}$ and $A_j'\subset A_j$ for each $j$. 
Up to relabeling, we may assume that these arcs and components are met in the order  
    \[\beta^-_1,\, A^\pm_1, \,\beta^+_1, \,A'_1,\, \beta^-_2,\, A^\pm_2,\,\beta^+_2,\, A'_2,\, \dots,\, \beta^-_m, \,A^\pm_m,\,\beta^+_m, \,A'_m\]
    when we walk around $X$ in the counter-clockwise direction; see the top right picture in Figure~\ref{fig: twisting_proof} for a reference.

    Let us decompose every edge $\beta_j$, $j=1, \dots, m$, as the union \[\beta_j=\beta_{j,X} \cup \beta_{j,A} \cup \beta_{j,Y},\]
    where $\beta_{j,X} := \beta_j \cap X$, $\beta_{j,A} := \beta_j \cap A$, and $\beta_{j,Y} := \beta_j \cap Y$. We decompose the edges $\beta^+_j$ and $\beta_j^-$ in the same manner.
    
    Set \[\alpha_j:=\beta_{j,X} \cup T_{1/m,\phi}(\beta_{j,A})\cup \beta_{j+1,Y}.\] Here and in the following, all indices are understood modulo $m$. By the definition of the graph rotation, $\alpha_1,\dots, \alpha_m$ are all the edges of $H$ that intersect $A$. Therefore, 
    \[E(H)=\{\alpha_1,\dots, \alpha_m\}\cup \big(E(G)\setminus E^A(G)\big);\]
    see the left picture in Figure \ref{fig: twisting_proof}, where $H$ is shown in solid black and red lines.  
    We set $\alpha_{j,A}:=T_{1/m,\phi}(\beta_{j,A})= \alpha_j\cap A$. Note that $\alpha_{j,A} \subset A_j$ for all $j=1,\dots, m$.

    Let $\alpha \in E(H)$ be arbitrary. We will now define triples $(\alpha^+, \alpha^-, U_\alpha)$ that satisfy conditions \ref{item: blow-up-tripple-i}-\ref{item: blow-up-tripple-iii} with respect to the map $g=T_\gamma^{-1} \circ f$. 
    
    First, if $\alpha\in E(G)\setminus E^A(G)$, then we set $\alpha^+:=\partial D^+_\alpha$, $\alpha^-:= \partial D^-_\alpha$, and $U_\alpha:=\inter(D_\alpha)$.
    
    Otherwise, $\alpha=\alpha_j$ for some $j=1,\dots, m$.     Recall that the Dehn twist $T_\gamma$ about $\gamma$ is defined with respect to the same homeomorphism $\phi\colon \partial \D \times \I \to A$ as the $1/m$-twist $T_{1/m,\phi}$ of $A$. Thus
    \[T_\gamma(\alpha_j)= \beta_{j,X} \cup T_\gamma(\alpha_{j,A})\cup \beta_{j+1,Y},\]
    where $\inter(T_\gamma(\alpha_{j,A}))$ intersects (transversely) each of the arcs $\beta_{j,A}$ and $\beta_{j+1,A}$ exactly once; see the red arc in the bottom middle picture in Figure \ref{fig: twisting_proof}. By construction, each of the maps $f|\inter(A^\pm_j)\colon \inter(A^\pm_j) \to \inter(A)\setminus \beta_j$, $f|A'_j\colon A'_j\to A_j$,  and $f|\inter(A^\pm_{j+1})\colon \inter(A^\pm_{j+1}) \to \inter(A)\setminus \beta_{j+1}$ is a homeomorphism. It follows that there are two lifts $\alpha^+_j$ and $\alpha^-_j$ of $\alpha_j$ under $g=T_\gamma^{-1} \circ f$ such that 
    \[\alpha^-_j= \beta^-_{j,X}\cup \alpha^-_{j,A} \cup \beta^-_{j+1,Y}\]
    and
    \[\alpha^+_j= \beta^+_{j,X}\cup \alpha^+_{j,A} \cup \beta^+_{j+1,Y},\]
    where $\alpha^-_{j,A}$ and $\alpha^+_{j,A}$ are the lifts of $T_\gamma(\alpha_{j,A})$ under $f$ that satisfy 
    \[\alpha^-_{j,A} \subset A^\pm _j \cup A_j' \text{\quad and \quad} \alpha^+_{j,A} \subset A_j'\cup A^\pm_{j+1}.\]
    Finally, let us set $U_{\alpha_j}$ to be the connected component of $S^2\setminus (\alpha_j^+\cup \alpha_j^-)$ that contains $\inter(D_{\beta_j})\cap X$ (and $\inter(D_{\beta_{j+1}})\cap Y$); see the bottom right picture in Figure \ref{fig: twisting_proof} for a reference. 
    
    \begin{claim}
    The triples $(\alpha^+, \alpha^-, U_\alpha)$, $\alpha\in E(H)$, constructed above satisfy conditions \ref{item: blow-up-tripple-i}-\ref{item: blow-up-tripple-iii} for $g=T_\gamma^{-1} \circ f$.
    \end{claim}
    
    This follows easily from the construction. Indeed, condition \ref{item: blow-up-tripple-ii} is immediate, as well as condition \ref{item: blow-up-tripple-i} for $\alpha\in E(G)\setminus E^A(G)$. Now, if $\alpha = \alpha_j$ for some $j=1,\dots, m$, then the Jordan arcs $\alpha_j$, $\alpha_j^+$, and $\alpha_j^-$ satisfy $\partial \alpha_j = \partial \alpha_j^+ = \partial \alpha_j^-$. Moreover, these arcs are inside the closed Jordan region \[W_j':= (D_{\beta_j}\cap X) \cup A_j^\pm \cap A_j' \cup A_{j+1}^\pm \cup (D_{\beta_{j+1}}\cap Y). 
    \]
    Lemma \ref{lem: Buser_isotopy_arcs} now implies that $\alpha_j$, $\alpha_j^+$, and $\alpha_j^-$ are all isotopic rel.\ $V(H)=V(G)$ (take $W_j$ to be an open Jordan region with $W_j'\subset W_j$ and $W_j\cap V(H) = W_j' \cap V(H) = \partial \alpha_j$). Finally, condition \ref{item: blow-up-tripple-iii} for the triples $(\alpha^+, \alpha^-, U_\alpha)$, $\alpha\in E(H)$, follows from the fact that the triples $(\partial D^+_e, \partial D^-_e, \inter(D_e))$, $e\in E(G)$, satisfy it by Proposition \ref{prop: blow-up-triples-properties}. We leave this straightforward verification to the reader.

 \medskip
 
 The claim above and Proposition \ref{prop: Blow-up triples} imply that $g=T^{-1}_\gamma\circ f$ is isotopic to a critically fixed Thurston map obtained by blowing up an admissible pair $(H, \varphi)$ for some $\varphi\in \Homeo^+(\Sp, V(H))$. 
 To complete the proof of Proposition \ref{prop: Key twisting proposition}, we need to check that $\varphi$ is isotopic to $\id_{\Sp}$ rel.\ $V(H)=\Crit(f)$. Let $H'$ be a connected planar embedded graph with $V(H')=\Crit(f)$ and $$H'\cap \left(A \cup  \bigcup_{\alpha\in E(H)} \overline U_\alpha\right) = \bigcup_{\alpha\in E(H)} \alpha^+.$$ 
 Note that $g(\alpha') \sim \alpha'$ rel.\ $V(H)$ for every edge $\alpha'\in E(H')$. Indeed, if $\alpha'=\alpha^+$ for some $\alpha\in E(H)$, then $g(\alpha')=\alpha \sim \alpha'$ rel.\ $V(H)$ by \ref{item: blow-up-tripple-iii}. Otherwise, $g(\alpha')=T_\gamma^{-1}(f(\alpha'))=f(\alpha') \sim \alpha'$ rel.\ $V(H)$ by \ref{item: blow-up-mapping-iv}. It now follows from Proposition \ref{prop: homeo from graph} and Corollary~\ref{cor: Homeo rigidity} that 
 $\varphi\in \Homeo_0^+(\Sp, V(H))$. 
 This finishes the proof of the proposition. 
\end{proof}

The following result establishes Main Theorem \ref{thm_intro: twisting}.

\begin{theorem}\label{thm: Main twisting theorem}
    Let $f \colon \Sp \to \Sp$ be a critically fixed Thurston map obtained by blowing up an admissible pair $(G, \id_{\Sp})$. Suppose $\gamma$ is a simple transversal with respect to $G$, and let $n\in \Z$ be arbitrary.

\begin{enumerate}[label=\normalfont(\roman*)]
\item\label{item: twisting_part_i} If $|G\cap \gamma|\geq 1$,   then the twisted map $T_\gamma^n \circ f$ is isotopic to a critically fixed Thurston map obtained by blowing up the admissible pair $(H, \id_{\Sp})$, where $H$ is the $(-n)$-rotation of $G$ about the curve $\gamma$.
\item\label{item: twisting_part_ii}  If $|G\cap \gamma|= 0$, then the twisted map $T_\gamma^n \circ f$ is isotopic to a critically fixed Thurston map obtained by blowing up the admissible pair $(G, T_\gamma^n)$.
\end{enumerate}
    
\end{theorem}

\begin{proof}
  
    \ref{item: twisting_part_i} Proposition \ref{prop: Key twisting proposition} implies the statement for $n=-1$. By a similar argument, we can also show the statement for $n=1$. 
    The rest of the statement follows from these two cases and Proposition \ref{prop: rotations up to isotopy} by induction on $n$. 

    \ref{item: twisting_part_ii} The statement immediately follows from Proposition \ref{prop: from_id_to_any_homeo}, 
    because the pair $(G, T_\gamma^n)$ is admissible.
\end{proof}

We can now easily deduce Corollary \ref{cor_intro: periodic} from the introduction.

\begin{corollary}\label{cor: Periodicity}
    Let $f \colon \Sp \to \Sp$ be a critically fixed Thurston map obtained by blowing up an admissible pair $(G, \id_{\Sp})$
    and $\gamma$ be a simple transversal with respect to $G$ with $|G\cap \gamma|\geq 1$.
    
    Then the sequence of the combinatorial equivalence classes of $\{T_\gamma^n \circ f\}_{n\in\Z}$ is strictly periodic with the period dividing $|G\cap \gamma|$. In other words, if  $n_1 \equiv n_2 \,\,\, \mod\, |G\cap \gamma|$, then the twisted maps $T_\gamma^{n_1} \circ f$ and $T_\gamma^{n_2} \circ f$ are combinatorially equivalent.
\end{corollary}

\begin{proof}
    Let $n\in \Z$ be arbitrary, and suppose that $n = mk + r$, where $m := |G\cap \gamma|$, $k \in \Z$, and $0 \leq r < m$.
    By Theorem \ref{thm: Main twisting theorem}\ref{item: twisting_part_i}, the twisted map $T_\gamma^n \circ f$ is isotopic to a critically fixed Thurston map obtained by blowing up the admissible pair  
    $(G_n, \id_{\Sp})$, where 
    $G_n$ is the $(-n)$-rotation of $G$ about the curve $\gamma$.

    By the definition of the graph rotation, the graph $G_n$ is (isotopic to) the $(-mk)$-rotation of the graph $G_{r}$ (rel.\ $V(G)$). It follows that $G_n$ is isotopic to $T_\gamma^{-k}(G_{r})$ rel.\ $V(G)$, where we view $T_\gamma^{-k}(G_{r})$ as a planar embedded graph with the vertex set $V(G)$. Since the graphs $G_r$ and $T_\gamma^{-k}(G_{r})$ are isomorphic, it follows that the admissible pairs  $(T_\gamma^{-k}(G_{r}), \id_{\Sp})$ and $(G_r, \id_{\Sp})$ are equivalent. Proposition \ref{prop: admis_equiv} now implies that the twisted maps $T_\gamma^n \circ f$ and $T_\gamma^r \circ f$ are combinatorially equivalent. This finishes the proof of the corollary.
\end{proof}

\begin{example} 
    Consider the critically fixed Thurston map $f_\square$ obtained by blowing up the admissible pair $(G_\square, \id_{\Sp})$; see Example \ref{ex: Topological square map}. Let $\gamma_1$ and $\gamma_2$ be the simple transversals with respect to $G_\square$ as in Figure \ref{fig: Rotation}. By Proposition \ref{prop: Key twisting proposition},  the charge graphs of $T^{-1}_{\gamma_1}\circ f_\square$ and $T^{-1}_{\gamma_2}\circ f_\square$ are as shown on the top right and bottom right in the same figure, respectively. In particular, we see that the map $T^{-1}_{\gamma_1}\circ f_\square$ is combinatorially equivalent to $f_{\square}$ and the map $T^{-1}_{\gamma_2}\circ f_\square$ is obstructed. (Note that the former agrees with what we obtained in Example~\ref{ex: Twisting by lifting}.) 
Corollary~\ref{cor: Periodicity} now implies that $T_{\gamma_1}^n \circ f_{\square}$ is combinatorially equivalent to $f_{\square}$ for all $n \in \Z$. Similarly, using Theorem \ref{thm: Main twisting theorem}\ref{item: twisting_part_i}, one can verify that the twisted map $T_{\gamma_2}^n \circ f_{\square}$ is combinatorially equivalent to $T_{\gamma_2} \circ f_{\square}$, and thus obstructed, for odd $n\in\Z$; and it is combinatorially equivalent to~$f_{\square}$, and thus realized, for even $n\in\Z$.
\end{example}

\begin{rem} The results in this section remain valid for a critically fixed Thurston map $f$ obtained by blowing up a (not necessarily admissible) pair $(G, \id_{\Sp})$, when $f$ is considered as a marked Thurston map $(f, V(G))$. 
\end{rem}

% \nocite{*}
% \bibliographystyle{alpha}
% \bibliography{main.bib}

\newcommand{\etalchar}[1]{$^{#1}$}

\end{document}